\documentclass[9pt,a4paper]{amsart}
\usepackage{verbatim}

\usepackage[toc]{appendix}
\usepackage[T1]{fontenc}

\usepackage{graphicx}
\usepackage{enumerate}
\usepackage{amsmath,amsfonts,amssymb}
\usepackage{color}
\def\loc{\operatorname{loc}}
\usepackage{cite}
\definecolor{citation}{rgb}{0.11,0.67,0.84}
\definecolor{formula}{rgb}{0.1,0.2,0.6}
\definecolor{url}{rgb}{0.11,0.67,0.84}
\usepackage{pgf,tikz}
\usepackage{mathrsfs}
\usepackage{fancyhdr}
\usepackage{dutchcal}

\usepackage{dsfont}

\newcommand{\reqnomode}{\tagsleft@false}

\usepackage{hyperref}

\def\dx{\,{\rm d}x}

\def\ds{\,{\rm d}s}
\def\dt{\,{\rm d}t}
\def\dy{\,{\rm d}y}

\def \d{\,{\rm d}}

\def\dist{\,{\rm dist}}

\def\supp{\,{\rm supp}}

\allowdisplaybreaks
\makeatletter
\DeclareRobustCommand*{\bfseries}{%
  \not@math@alphabet\bfseries\mathbf
  \fontseries\bfdefault\selectfont
  \boldmath
}

\makeatother

\newlength{\defbaselineskip}
\newcommand{\setlinespacing}[1]
           {\setlength{\baselineskip}{#1 \defbaselineskip}}
\newcommand{\mint}{\mathop{\int\hskip -1,05em -\, \!\!\!}\nolimits}

\newtheorem{theorem}{Theorem}
\newtheorem{corollary}{Corollary}[section]
\newtheorem{definition}{Definition}
\newtheorem{remark}{Remark}[section]
\newtheorem{lemma}{Lemma}[section]
\newtheorem{proposition}{Proposition}[section]
\numberwithin{equation}{section}
\allowdisplaybreaks[4]

\newcommand{\kk}{\kappa}

\def\er{\mathbb R}

\newcommand{\ti}[1]{\tilde{#1}}
\newcommand{\mf}[1]{\mathfrak{#1}}

\newcommand\eps\varepsilon

\def\eqn#1$$#2$${\begin{equation}\label#1#2\end{equation}}

\newcommand{\be}{\begin{equation}}
\newcommand{\ee}{\end{equation}}

\newcommand{\rr}{\varrho}

\newcommand{\const}{\operatorname{const}}
\newcommand{\snr}[1]{\lvert #1\rvert}

\newcommand{\nr}[1]{\lVert #1 \rVert}

\newcommand{\uu}{\mathfrak{u}}

\newcommand{\RN}{\mathbb{R}^{N}}

\newcommand{\N}{\mathbb{N}}

\def\name[#1, #2]{#1 #2}

\newcommand{\rif}[1]{(\ref{#1})}

\title[Quasiconvexity and nonlinear potentials]{Quasiconvexity and partial regularity\\via nonlinear potentials}

\author[De Filippis]{Cristiana De Filippis}  \address{Cristiana De Filippis\\Dipartimento SMFI, Universit\'a di Parma\\ Parco Area delle Scienze 53/A, 43124 Parma, Italy} \email{\url{cristiana.defilippis@unipr.it}}

\begin{document}

\subjclass[2020]{35B65, 31C45 \vspace{1mm}} 

\keywords{Quasiconvexity, $(p,q)$-growth, Nonlinear potential theory, Degenerate variational integrals.\vspace{1mm}}

\thanks{{\it Acknowledgements.}\ This work is supported by the University of Turin via the project "Regolarit\'a e propriet\'a qualitative delle soluzioni di equazioni alle derivate parziali". The author thanks the referees for their sharp comments that eventually improved some aspects of the paper.
\vspace{1mm}}


\maketitle

\begin{abstract}
We show how to infer sharp partial regularity results for relaxed minimizers of degenerate, nonuniformly elliptic quasiconvex functionals, using tools from Nonlinear Potential Theory. In particular, in the setting of functionals with $(p,q)$-growth - according to the terminology of Marcellini \cite{ma2} - we derive optimal local regularity criteria under minimal assumptions on the data. 
 \end{abstract}

\renewcommand{\abstractname}{R\'esum\'e}

\begin{abstract}
    Nous \'etudions des conditions optimales pour la r\'egularit\'e partielle des minimiseurs d'une fonctionelle relax\'ee. Il s'agit d'une fonctionelle d\'egener\'ee, quasiconvexe et non-uniformément elliptique. En particulier, pour les fonctionelles avec une croissance de type $(p,q)$ (suivant la terminologie de Marcellini \cite{ma2}) nous prouvons la r\'egularit\'e locale optimale avec des hypoth\`eses minimales sur les données.

\end{abstract}
\vspace{3mm}

\setlinespacing{1.00}
\section{Introduction}\label{si}
In this paper we connect the classical partial regularity theory of quasiconvex functionals with nonlinear potential theory for degenerate elliptic equations. Specifically, we establish sharp $\varepsilon$-regularity criteria for relaxed minimizers of nonhomogeneous, degenerate, quasiconvex functionals with $(p,q)$-growth of the form
\begin{eqnarray}\label{fun}
W^{1,p}(\Omega,\mathbb{R}^{N})\ni w\mapsto \mathcal{F}(w;\Omega):= \int_{\Omega}\left[F(Dw)-f\cdot w\right] \dx,
\end{eqnarray}
i.e. minimizers of the Lebesgue-Serrin-Marcellini extension of $\mathcal{F}(\cdot)$:
\begin{flalign}\label{exfun}
\bar{\mathcal{F}}(w;\Omega):=\inf\left\{\liminf_{j\to \infty}\mathcal{F}(w_{j};\Omega)\colon \{w_{j}\}_{j\in \N}\subset W^{1,q}_{\loc}(\Omega,\mathbb{R}^{N})\colon w_{j}\rightharpoonup w \ \mbox{in} \ W^{1,p}(\Omega,\mathbb{R}^{N})\right\},
\end{flalign}
guaranteeing gradient (partial) continuity of minima under optimal assumptions on $f$. Here $\Omega$ denotes a bounded open domain of $\er^n$, $n\geq 2$ with Lipschitz boundary. The problem of determining the best conditions on $f$ implying continuity of $Du$ is classical. When $F(z)\equiv |z|^2/2$, we are back to the Poisson equation $-\Delta u=f$. In this case optimal conditions on $f$ are formulated via the borderline Lorentz space $L(n,1)$, that is
\eqn{lorenzo}
$$
\| f\|_{L(n,1)(\Omega)}:=  \int_0^\infty |\{x \in \Omega\,  \colon  |f(x)|> \lambda\}|^{1/n} \, \d\lambda <\infty\,.
$$
This is a consequence of a classical result of Stein \cite{stein}, and condition \rif{lorenzo} is sharp \cite{CiGA}. Surprisingly enough, the same conclusion continues to hold in the case of the $p$-Laplacian system, that is when $F(z)=|z|^p/p$, as shown by Kuusi \& Mingione \cite{kumi1, kumi2}; see also the global Lipschtz regularity results of Cianchi \& Maz'ya \cite{cm, cm1}. Finally, again \rif{lorenzo} appears in the setting of nonuniformly elliptic problems, as found in \cite{bm, demi1}. In particular, the techniques of \cite{kumi2} rely on the possibility of controlling, locally, the gradient of solutions via the (truncated) Riesz potential of $f$, that is 
\eqn{rizzo}
$$\mathbf{I}^{f}_{1}(x,\rr)= \int_{0}^{\rr}\left(\mint_{B_{\sigma}(x)}\snr{f} \ \dy\right) \ \d\sigma \lesssim \int_{\er^n} \frac{\snr{f(y)}}{|x-y|^{n-1}}\dy \,,
$$
exactly as in the linear case, see \cite{kumig} for a review of gradient potential estimates in nonlinear problems. 
The main result in \cite{kumi2} asserts that $Du$ is continuous in $\Omega$ provided 
$$
\lim_{\sigma\to 0}\mathbf{I}^{f}_{1}(x,\sigma)=0\qquad \mbox{locally uniformly with respect to} \ \ x\in \Omega.
$$
This condition is automatically implied by the $L(n,1)$-condition in \rif{lorenzo} and therefore yields the optimal continuity result for the gradient. We remark that while the use of linear potentials is standard for the Poisson equation, it is actually an additional surprising feature in the nonlinear case. 
Instead, in \cite{bm, demi1} a nonlinear potential of Wolff type is used to control the gradient of minimizers, that is, 
\eqn{nonline}
$$
\mathbf{I}^{f}_{1,m}(x,\rr):=\int_{0}^{\rr}\left(\sigma^{m}\mint_{B_{\sigma}(x)}\snr{f}^{m} \dx\right)^{1/m} \, \frac{\d\sigma}{\sigma},\qquad \quad m>1\,.
$$
Note that H\"older's inequality gives $\mathbf{I}^{f}_{1}(x,\rr)\leq \mathbf{I}^{f}_{1,m}(x,\rr)$ and that the two potentials share the same homogeneity and, therefore, mapping properties on function spaces, see \cite{demi1} for more details. In fact, $\mathbf{I}^{f}_{1,m}(\cdot)$ yields the same sharp condition on $f$ of those implied by $\mathbf{I}^{f}_{1}(\cdot)$. Nonlinear potentials as in \rif{nonline} were originally introduced by Havin \& Maz'ya \cite{hm1} and their use lies at the core of Nonlinear Potential Theory, see \cite{kumig} for more details and references. The results in \cite{bm, demi1, kumi1, kumi2} hold for general equations and systems and, in fact, minimizers of \rif{fun} with $F(\cdot)$ convex are treated via the use of their Euler-Lagrange system. In this case there is no difference between minimizers and critical points. This is typical when dealing with convex integrands $F(\cdot)$. The more challenging case of quasiconvex integrands we are interested in is different. We recall that $F(\cdot)$ is said to be quasiconvex when
\begin{flalign}\label{qc}
\mint_{B_{1}(0)}F(z+D\varphi) \ \dx \ge F(z)\quad \mbox{holds for all} \ \ z\in \mathbb{R}^{N\times n}, \ \ \varphi\in C^{\infty}_{\rm c}(B_{1}(0),\mathbb{R}^{N}).
\end{flalign}
This definition was first introduced by Morrey \cite{mo}. Under polynomial growth conditions, quasiconvexity guarantees lower semicontinuity with respect to natural weak topology of Sobolev spaces. Quasiconvexity is also necessary for lower semicontinuity and therefore it is a natural condition in the multidimensional Calculus of Variations. A distintive feature of quasiconvex integrals is that critical points and minimizers might behave quite differently. Specifically, a classical result of Evans \cite{ev} asserts that minimizers of quasiconvex integrals are partially regular, i.e., $Du$ is locally H\"older continuous outside a negligible closed subset. On the other hand, another result by M\"uller \& \v{S}ver\'ak \cite{musv} claims that stationary points, i.e., solution to the Euler-Lagrange systems, might develop singularities on a dense subset of $\Omega$. Therefore regularity of minimizers must be obtained making use of minimality rather than stationarity. We note that, already in the case of standard elliptic systems, classical counterexamples \cite{svya} show that singularities might occur and one cannot go beyond partial regularity. We refer to \cite{wild} for a review of results and counterexamples. 

In this paper we use nonlinear potential theoretic methods to describe partial regularity results for relaxed minimizers of quasiconvex functionals as in \rif{fun}. As far as we know, this is the first paper where such an approach is taken in the quasiconvex case. Partial regularity results of this kind for elliptic systems have been derived by Kuusi \& Mingione in \cite{kumi} and Byun \& Youn in \cite{by}. Although our results are already completely new in the case of quasiconvex functionals with $p$-polynomial growth, i.e., $F(z)\approx |z|^p$, we shall treat the most general case of nonuniformly elliptic functionals with different polynomial growth conditions, that is
\eqn{roughass}
$$
F(z)\lesssim \snr{z}^{q}+1 \quad \mbox{and} \quad  \snr{z}^{p-2}\snr{\xi}^{2}\snr{\zeta}^{2} \lesssim  \partial^{2}F(z)\langle\xi\otimes \zeta,\xi\otimes \zeta\rangle, \qquad 1 < p \leq q \,.
$$
The standard case is given by $p=q$. Such nonuniform ellipticity can be measured by introducing a proper notion of uniform/nonuniform Legendre-Hadamard ellipticity, measured by the ratio
\begin{eqnarray}\label{ellr}
\mathcal{R}_{\rm{LH}}(z):=\frac{\sup_{|\xi|=| \zeta|=1}\partial^{2}F(z)\langle\xi\otimes \zeta,\xi\otimes \zeta\rangle}{\inf_{|\xi|=| \zeta|=1} \partial^{2}F(z)\langle\xi\otimes \zeta,\xi\otimes \zeta\rangle} \lesssim 1+ |z|^{q-p} \,, 
\end{eqnarray}
where  $\xi \in \er^N$,  $\zeta \in \er^n$ and $z\in \er^{N\times n}$. This replaces, in our setting, the traditional notion of uniform ellipticity  formulated in terms of ratio between maximal and minimal eigenvalues of the Hessian $\partial^2F(\cdot)$, which is classical in the theory of elliptic equations, i.e.,  
\eqn{rapporto}
$$
\mathcal{R}(z):=\frac{\mbox{highest eigenvalue of} \ \partial^{2}F(z)}{\mbox{lowest eigenvalue of} \ \partial^{2}F(z)}\,.
$$
A brief review of results on quasiconvex functional, especially in the context of nonuniformly ellipticity, is placed in Section \ref{pqq}. The precise assumptions we are going to adopt, together with the notation we employ, are described in Section \ref{pre}. The standard notion of relaxed local minimizer \cite{ts1} is the following:
\begin{definition}\label{d1}
Let $p\in (1,\infty)$. A function $u\in W^{1,p}(\Omega,\mathbb{R}^{N})$ is a local minimizer of \eqref{exfun} on $\Omega$ with $f\in W^{1,p}(\Omega,\mathbb{R}^{N})^{*}$ if and only if every $x_{0}\in \Omega$ admits a neighborhood $B\Subset \Omega$ so that $\bar{\mathcal{F}}(u;B)<\infty$ and $\bar{\mathcal{F}}(u;B)\le \bar{\mathcal{F}}(w;B)$ for all $w\in W^{1,p}(B,\mathbb{R}^{N})$ so that $\supp(u-w)\Subset B$. 
\end{definition}
Analogous definition holds for local minimizers of functional \eqref{fun}. When considering \rif{fun}-\eqref{exfun}, without loss of generality we shall always assume that $f$ is defined on the whole $\er^n$, eventually letting $f \equiv 0$ outside $\Omega$. For this reason, when indicating that $f$ belongs to a certain function space, we will often omit to specify the underlying domain. 
In this setting, our main result reads as
\begin{theorem}\label{t3}
Under assumptions \eqref{assf}-\eqref{p0} and \eqref{f}, let $u\in W^{1,p}(\Omega,\mathbb{R}^{N})$ be a local minimizer of \eqref{exfun} and assume that
\begin{eqnarray}\label{con.5}
\lim_{\sigma\to 0}\mathbf{I}^{f}_{1,m}(x,\sigma)=0\qquad \mbox{locally uniformly with respect to} \ \ x\in \Omega.
\end{eqnarray}
Then there exists an open set $\Omega_{u}\subset \Omega$ such that
\begin{eqnarray}\label{pe.35}
\snr{\Omega\setminus \Omega_{u}}=0\quad \mbox{and}\quad   \mbox{$Du$ is continuous in} \ \Omega_{u}
\end{eqnarray}
which can be characterized as
\begin{eqnarray}\label{ou}
\Omega_{u}&:=&\left\{\frac{}{}x_{0}\in \Omega\colon \exists \ M\equiv M(x_{0})\in (0,\infty),\ \ti{\varepsilon}\equiv \ti{\varepsilon}(\textnormal{\texttt{data}}_{\textnormal{c}},M^{q-p}),\ \ti{\rr}\equiv \ti{\rr}(\textnormal{\texttt{data}}_{\textnormal{c}},M^{q-p},f(\cdot))<d_{x_{0}}\right.\nonumber \\
&&\qquad   \left.\frac{}{}\mbox{such that} \ 
\snr{(Du)_{B_{\rr}(x_{0})}}<M \ \mbox{and} \ \mf{F}(u;B_{\rr}(x_{0}))<\ti{\varepsilon} \ \mbox{for some} \ \rr\in (0,\ti{\rr}]\frac{}{}\right\}.
\end{eqnarray}
\end{theorem}
We note that Theorem \ref{t3} is the natural partial regularity version of the results in \cite{bm,demi1,kumi1, kumi2}, in that it matches the smallness excess criteria, that are typical in partial regularity theory \cite{af,bddms,dlsv,dumi, ev, fomi, gk,kt} and uses the potential displayed in \rif{nonline}. This choice yields optimal conditions on $f$ in terms of Lorentz spaces, as indicated in the following:
\begin{theorem}\label{t4}
Assume \eqref{assf}-\eqref{p0} and \eqref{f} and let $u\in W^{1,p}(\Omega,\mathbb{R}^{N})$ be a local minimizer of \eqref{exfun}. There exists a full measure open set $\Omega_{u}\subset \Omega$, characterized as in \eqref{ou} such that:
\begin{itemize}
    \item[(\emph{i})] $f\in L(n,1) \  \Longrightarrow \ Du$ \mbox{is continuous in} $\Omega_{u}$;
    \item[(\emph{ii})]$f\in L^{d}$ for some $d>n$ $\ \Longrightarrow \ Du\in C^{0,\ti{\alpha}}_{\loc}(\Omega_{u},\mathbb{R}^{N\times n})$ with $\ti{\alpha}\equiv \ti{\alpha}(n,N,p,d)$.
\end{itemize}
\end{theorem}
Theorems \ref{t3}-\ref{t4} are a consequence of a finer criterion, allowing to detect Lebesgue points of $Du$ in terms of the poitwise behaviour of the potential $\mathbf{I}^{f}_{1,m}(\cdot)$.
\begin{theorem}\label{t2}
Under assumptions \eqref{assf}-\eqref{p0} and \eqref{f}, let $u\in W^{1,p}(\Omega,\mathbb{R}^{N})$ be a local minimizer of \eqref{exfun}, $x_{0}\in \Omega$ be a point such that
\begin{eqnarray}\label{con.1.1}
\mathbf{I}^{f}_{1,m}(x_{0},1)<\infty
\end{eqnarray}
and $M\equiv M(x_{0})$ be a positive, finite constant. Then there are $\breve{\varepsilon}\equiv \breve{\varepsilon}(\textnormal{\texttt{data}}_{\textnormal{c}},M^{q-p})\in (0,1)$ and a threshold radius $\breve{\rr}\equiv \breve{\rr}(\textnormal{\texttt{data}}_{\textnormal{c}},M^{q-p},f(\cdot))\in (0,d_{x_{0}})$ such that if
\begin{eqnarray}\label{con.1}
\left\{
\begin{array}{c}
\displaystyle 
\ \snr{(Du)_{B_{\rr}(x_{0})}}<M \\[17pt]\displaystyle
\ \mf{F}(u;B_{\rr}(x_{0}))+\left(\mathbf{I}^{f}_{1,m}(x_{0},\rr)\right)^{\frac{1}{p-1}}+\left(\mathbf{I}^{f}_{1,m}(x_{0},\rr)\right)^{\frac{q}{p(p-1)}}<\breve{\varepsilon} ,
\end{array}
\right.
\end{eqnarray}
is satisfied for some $\rr\in (0,\breve{\rr}]$, then
\begin{eqnarray}\label{lp}
\lim_{\sigma\to 0}(Du)_{B_{\sigma}(x_{0})}=Du(x_{0}) 
\end{eqnarray}
and
\begin{flalign}\label{lp.1}
\snr{Du-(Du)_{B_{\varsigma}(x_{0})}}\le c\mf{F}(u;B_{\varsigma}(x_{0}))+c\left[\mathbf{I}^{f}_{1,m}(x_{0},\varsigma)+\left(\mathbf{I}^{f}_{1,m}(x_{0},\varsigma)\right)^{q/p}\right]^{\frac{1}{p-1}}
\end{flalign}
hold for all $\varsigma\in (0,\rr]$, with $c\equiv c(\textnormal{\texttt{data}}_{\textnormal{c}},M^{q-p})$. In particular, the set of all points satisfying \eqref{con.1.1} and \eqref{con.1} coincides with the set of Lebesgue points of $Du$.
\end{theorem}
The next theorem is a technical reformulation of Theorem \ref{t2}. The results about the gradient are presented in terms of the nonlinear vector field $V_p(z)= |z|^{(p-2)/4}z$, which is typical in the regularity theory of degenerate problems of $p$-Laplacian type. 
\begin{theorem}\label{t5}
Under assumptions \eqref{assf}-\eqref{p0} and \eqref{f}, let $u\in W^{1,p}(\Omega,\mathbb{R}^{N})$ be a local minimizer of \eqref{exfun}, $x_{0}\in \Omega$ be a point such that \eqref{con.1.1} is verified and $M\equiv M(x_{0})$ be a positive, finite constant. Then there  are a positive number $\breve{\varepsilon}\equiv \breve{\varepsilon}(\textnormal{\texttt{data}}_{\textnormal{c}},M^{q-p})<1 $ and a threshold radius $\breve{\rr}\equiv \breve{\rr}(\textnormal{\texttt{data}}_{\textnormal{c}},M^{q-p},f(\cdot))\in (0,d_{x_{0}})$ such that if
\begin{eqnarray}\label{con.6}
\left\{
\begin{array}{c}
\displaystyle 
\ \snr{(Du)_{B_{\rr}(x_{0})}}<M\\[17pt]\displaystyle
\ \ti{\mf{F}}(u;B_{\rr}(x_{0}))+\left(\mathbf{I}^{f}_{1,m}(x_{0},\rr)\right)^{\frac{p}{2(p-1)}}+\left(\mathbf{I}^{f}_{1,m}(x_{0},\rr)\right)^{\frac{q}{2(p-1)}}<\breve{\varepsilon} ,
\end{array}
\right.
\end{eqnarray}
is verified for some $\rr\in (0,\breve{\rr}]$, then
\begin{eqnarray}\label{lv}
\lim_{\sigma\to 0}(V_{p}(Du))_{B_{\sigma}(x_{0})}=V_{p}(Du(x_{0}))
\end{eqnarray}
and
\begin{flalign}\label{lv.1}
\snr{V_{p}(Du)-(V_{p}(Du))_{B_{\varsigma}(x_{0})}}\le c\ti{\mf{F}}(u;B_{\varsigma}(x_{0}))+c\left[\mathbf{I}^{f}_{1,m}(x_{0},\varsigma)+\left(\mathbf{I}^{f}_{1,m}(x_{0},\varsigma)\right)^{q/p}\right]^{\frac{p}{2(p-1)}}
\end{flalign}
hold for all $\varsigma\in (0,\rr]$ with $c\equiv c(\textnormal{\texttt{data}}_{\textnormal{c}},M^{q-p})$. In particular, the set of all points satisfying \eqref{con.1.1} and \eqref{con.6} coincides with the set of Lebesgue points of $V_{p}(Du)$.
\end{theorem}
Let us point out that if the functional $\mathcal{F}(\cdot)$ coincides with its Lebesgue-Serrin-Marcellini extension $\bar{\mathcal{F}}(\cdot)$ or $\mathcal{F}(\cdot)$ is strictly $W^{1,p}$-quasiconvex in the sense of Remark \ref{rpsqc} below, then the results in Theorems \ref{t3}-\ref{t5} hold for local minimizers of $\mathcal{F}(\cdot)$.

\subsection{Nonuniform ellipticity, quasiconvexity, and nonlinear potentials} \label{pqq} Nonuniform ellipticity is a classical topic in regularity theory. In the case of elliptic and parabolic equations a by now large literature is available; see for instance \cite{laur, iv, serrin} for a review of some classical results. In the standard terminology, nonuniform ellipticity of an integrand $F(\cdot)$ occurs when the ratio \rif{rapporto} is not bounded. When $\partial^2F(\cdot)$ is elliptic only on rank-one tensors, a definition like the one in \rif{ellr} appears to be more suitable to qualify (the rate of) nonuniform ellipticity. For variational integrals, nonuniform ellipticity has been framed by Marcellini in the setting of functionals with so-called $(p,q)$-growth conditions \cite{ma4}. This terminology refers to the case of functionals as in \rif{fun}, satisfying growth conditions of the type 
$
|z|^p \lesssim F(z) \lesssim |z|^q+1.
$
This is precisely our setting. Marcellini's papers deal with convex integrands, both in the scalar and in the vectorial case \cite{ma4, maexp}. We also note that Marcellini has proved the first lower semicontinuity results in quasiconvex case \cite{ma3, ma2} under $(p,q)$-growth conditions, see also the recent surveys \cite{masu1, masu2}. In this respect, further results can be found in \cite{foma, k1, kna}. The common underlying feature of all such papers is the assumption that the gap $q/p$ is close to one, depending on the dimension $n$, i.e., 
\eqn{generalbound}
$$
\frac qp < 1 + \texttt{o}(n)\,, \qquad  \texttt{o}(n) \approx \frac 1n\,.
$$
Such a bound serves to slow down the possible blow up of the ellipticity ratio in \rif{rapporto} and it is unavoidable in this setting, as shown by counterexamples \cite{ma3}. Recent work on finding better bounds in the autonomous case includes for instance \cite{BS, hs,s}. Conditions of the type in \rif{generalbound} also intervene in proving lower semicontinuity results for quasiconvex multiple integrals, with respect to suitable weak convergences \cite{ma2, foma, k1}. In the case of nonautonomous integrands $F(x,z)$, the relevant bounds, still of the type in \rif{generalbound}, also incorporate informations concerning the regularity of the integrand $F(\cdot)$ with respect to the $x$-variable \cite{demi1, demi2, demi3, ELM}. We refer to the recent survey \cite{dark2} for a review of such results and various versions of \rif{generalbound} in connection with the structural assumptions on $F(\cdot)$. Not surprisingly, conditions as \rif{generalbound} also appear in the vectorial setting, when  partial regularity comes into play. Partial regularity in vectorial problems is a classical topic that has been started by Evans \cite{ev} as far as quasiconvexity is concerned, see \cite{wild} for references and historical notes. Subsequently, Evans's result has been extended in various directions; for instance, local minimizers are treated in \cite{kt} by Kristensen \& Taheri, and an estimate on the size of the singular set of solutions can be found in \cite{km}; degenerate cases are considered in \cite{dumi}, partial regularity of solutions, rather then the gradients, is proved in \cite{fomi}. As far as $(p,q)$-growth conditions are concerned, partial regularity results have been obtained in the convex case, see for instance \cite{BF,s}. Better results can be achieved when considering special structures, including estimates for the singular set, as for instance done by Tachikawa \cite{tac}; see also \cite{cdk,px, demi2, rata}. The first partial regularity results for quasiconvex integrals have been obtained by Schmidt \cite{ts, ts1, ts2}, who treated bot the case of $W^{1,p}$-quasiconvexity and the one of relaxed minimizers, that is the same we deal with here. In all such cases, the forcing term $f$ in \rif{fun} does not appear. Our approach to partial regularity builds on new ingredients coming from recent developments in Nonlinear Potential Theory. For the scalar case we refer to \cite{kumig}, where a survey of results and samples of proofs and techniques can be found. As far as the nonstandard growth cases and measurable coefficients are concerned, results can be found in \cite{ba2, bk, by1, cyz}. For the vectorial case, we refer to \cite{by, kumi, kumi2}, and in fact we shall also build on the techniques introduced there.  \subsection{Technical novelties} The major tool to prove gradient partial regularity for minimizers of multiple integrals is a Morrey-type decay estimate for the excess functional. In case of nondegenerate, homogeneous problems, excess decay essentially comes by combining the approximate "$\mathcal{A}$-harmonic" character of minimizers together with Caccioppoli inequality, \cite{ev,gk,ts}. When dealing with homogeneous, degenerate problems, the situation becomes slightly more involved as it is necessary to analyze two possible situations: the nondegenerate regime, in which the gradient average is not very small compared to the excess: in this case one can quantify the distance (in average) between a minimizer and a suitable harmonic function; and the degenerate one, characterized by the smallness of the gradient average with respect to the excess functional - solutions now will be close to a $p$-harmonic map, \cite{af,dumi,ts2}. At this stage it is important to notice that the passage from nondegenerate to degenerate regime may feature a change of scale, and the estimates obtained within each of them hold at different scales, \cite{dumi, ts2}. A delicate exit time argument then rules the succession of degenerate or nondegenerate phases and matches the corresponding decay estimates. The crucial aspect of this iterative procedure is the stability of the nondegenerate regime, in the sense that if it holds at one scale, it can be automatically transferred at all successive scales. The stability of the nondegenerate regime drastically fails when the problem is nonhomogeneous, regardless the degree of smoothness of the right-hand side term $f$. To cope with this issue, we introduce a new iterative "blocks and chains" scheme based on potential theoretic techniques \cite{kumig,kumi,kumi2}, which works as sketched in the following bullets.
\begin{itemize}
    \item A (possibly countably infinite) sequence of maximal iteration chains is introduced to compare the size of the composite excess functional (i.e. gradient average plus excess functional) to that of a certain nonlinear function accounting for the contribution of $f$. The reference interval (domain of the excess functional, \eqref{excess} below) is then split into disjoint subintervals determined according to the behavior of the composite excess functional with respect to the $f$-depending term. 
    \item Whenever the composite excess functional is smaller than the $f$-depending term, the decay estimate follows almost tautologically; while, if the opposite scenario occurs, it is necessary to analyze the degenerate/nondegenerate behavior of the integrand via exit time arguments of the type considered in \cite{dumi}, that will not only merge the decay estimates obtained at different scales, but will also point out that in correspondence of large (compared to the $f$-component) composite excess functional, the nondegenerate regime is stable and cannot end as long as the composite excess remains larger than the $f$-term. 
    \item The analysis made at the previous point provides a series of decay estimates whose validity depends on each subinterval determined by the iteration chains. This at a first sight might seem quite endless, however, exploiting the connections established by the iteration chains between the composite excess functional and the $f$-depending term it is possible to assemble the various subintervals into disjoint blocks, whose union coincides with the whole reference interval and, for any radius belonging to each block, a complete decay estimate for the excess functional is produced. This approach seems to be new already for smooth data $f$.
\end{itemize}
Since we are working within the Legendre-Hadamard nonuniformly elliptic setting \eqref{ellr}, an important aspect of the whole procedure described so far consists in assuring uniform control on the size of gradient averages during iterations. This is quite a standard tool in the literature, cf. \cite{gk,ts,ts2,ts1}. However, the only information available on $f$, i.e. the mere finiteness of the potential $\textbf{I}^{f}_{1,m}(\cdot)$, prevents us from following the classical approach that would unavoidably lead to a violation of the borderline regularity condition imposed on $f$. Rather, we exploit a delicate inductive scheme from \cite{kumi,kumi2} relying on the finiteness of $\textbf{I}^{f}_{1,m}(\cdot)$, that allows transferring the controlled boundedness of the gradient average at all scales. We refer to Section \ref{ex} for more details on this matter.
\section{Preliminaries, notation, assumptions}\label{pre}
In this section we shall record our notation, describe the basic structural assumptions governing the integrand $F(\cdot)$ and the forcing term $f$ displayed in \eqref{fun} and collect some auxiliary results that will be helpful at various stages of the paper.
\subsection{Notation}\label{notsec}
In the following, $\Omega\subset \er^n$ denotes an open, bounded domain with Lipschitz boundary, and $n \geq 2$. We denote by $c$ a general constant larger than one. Different occurrences of constant $c$ from line to line will be still denoted by $c$. Special occurrences will be denoted by $c_*,  \tilde c$ or likewise. Relevant dependencies on parameters will be as usual emphasized by putting them in parentheses. We denote by $ B_r(x_0):= \{x \in \er^n  : |x-x_0|< r\}$ the open ball with center $x_0$ and radius $r>0$; we omit denoting the center when it is not necessary, i.e., $B \equiv B_r \equiv B_r(x_0)$; this especially happens when various balls in the same context share the same center. For $x_{0}\in \Omega$, it is $d_{x_{0}}:=\min\left\{1,\dist(x_{0},\partial \Omega)\right\}$. 
With $\mathcal B \subset \er^{n}$ being a measurable set with bounded positive measure $0<|\mathcal B|<\infty$, and with $a \colon \mathcal B \to \er^{k}$, $k\geq 1$, being a measurable map, we denote $$
   (a)_{\mathcal B} \equiv \mint_{\mathcal B}  a(x) \dx  :=  \frac{1}{|\mathcal B|}\int_{\mathcal B}  a(x) \dx\;.
$$
We shall often use the elementary "minimality" property of the average, i.e.:
\begin{eqnarray}\label{minav}
\left(\mint_{B}\snr{a-(a)_{B}}^{t} \ \dx\right)^{1/t}\le 2\left(\mint_{B}\snr{a-z}^{t} \dx\right)^{1/t}
\end{eqnarray}
for all $z\in \mathbb{R}^{N\times n}$ and any $t\ge 1$. Whenever $t\ge 1$, $s\ge 0$, $q\ge p> 1$, we abbreviate:
\begin{eqnarray}\label{ik}
\mf{I}_{t}(a;\mathcal{B}):=\left(\mint_{\mathcal{B}}\snr{a(x)}^{t} \dx\right)^{\frac{1}{t}},\qquad \qquad \mf{K}(s):=s+s^{q/p}
\end{eqnarray}
and define
\begin{eqnarray}\label{ik.1}
\mathds{1}_{\{q>p\}}:=\begin{cases}
\ 1\quad &\mbox{if} \ \ q>p\\
\ 0\quad &\mbox{if} \ \ q=p, 
\end{cases}\qquad\qquad \mathds{1}_{\{p>2\}}:=\begin{cases}
\ 1\quad &\mbox{if} \ \ p>2\\
\ 0\quad &\mbox{if} \ \ p=2.
\end{cases}
\end{eqnarray}
Finally, if $t>1$ we will denote its conjugate by $t':=t/(t-1)$ and its Sobolev esponent as $t^{*}:=nt/(n-t)$ if $t<n$ or any number larger than one for $t\ge n$. To simplify the notation, we shall collect the main parameters governing the problem under investigation in the shorthands
$$
\textnormal{\texttt{data}}:=\left(n,N,\lambda,\Lambda,p,q,m\right)\qquad\mbox{and}\qquad  \textnormal{\texttt{data}}_{\textnormal{c}}:=\left(\textnormal{\texttt{data}},\mu(\cdot),\omega(\cdot)\right),
$$
we refer to Section \ref{assec} for more details on the various quantities appearing above.
\subsection{Structural assumptions}\label{assec} We assume that the integrand $F\colon \mathbb{R}^{N\times n}\to \mathbb{R}$ satisfies:
\begin{flalign}\label{assf}
\begin{cases}
\ F\in C^{2}_{\loc}(\mathbb{R}^{N\times n})\\
\ \Lambda^{-1}\snr{z}^{p}\le F(z)\le \Lambda \left[\snr{z}^{p}+\snr{z}^{q}\right]\\
\ \snr{\partial^{2}F(z)}\le \Lambda \left[\snr{z}^{p-2}+\snr{z}^{q-2}\right]\\
\ \snr{\partial^{2} F(z_{1})-\partial^{2}F(z_{2})}\le\mu\left(\frac{\snr{z_{2}-z_{1}}}{\snr{z_{2}}+\snr{z_{1}}}\right)\left[ \left(\snr{z_{1}}^{2}+\snr{z_{2}}^{2}\right)^{\frac{p-2}{2}}+\left(\snr{z_{1}}^{2}+\snr{z_{2}}^{2}\right)^{\frac{q-2}{2}}\right]
\end{cases}
\end{flalign}
for all $z,z_{1},z_{2}\in \mathbb{R}^{N\times n}$. In \eqref{assf}, $\Lambda\ge 1$ is an absolute constant, exponents $(p,q)$ verify condition
\begin{eqnarray}\label{pq}
q\ge p\ge 2,\qquad  \qquad q<p+\frac{1}{n}
\end{eqnarray}
and $\mu\colon [0,\infty)\to [0,1]$ is a modulus of continuity, i.e. a bounded, concave and non-decreasing function. The crucial assumption is that $F(\cdot)$ is strictly degenerate quasiconvex, in the sense that whenever $B\Subset \Omega$ is a ball it holds that 
\begin{flalign}\label{sqc}
\int_{B}\left[F(z+D\varphi)-F(z)\right] \dx\ge \lambda\int_{B}(\snr{z}^{2}+\snr{D\varphi}^{2})^{\frac{p-2}{2}}\snr{D\varphi}^{2} \dx\qquad \mbox{for all} \ \ z\in \mathbb{R}^{N\times n}, \ \ \varphi\in C^{\infty}_{c}(B,\mathbb{R}^{N}),
\end{flalign}
where $\lambda$ is a positive, absolute constant. Moreover, we assume that $F(\cdot)$ features degeneracy of $p$-Laplacean type at the origin, i.e., 
\begin{eqnarray}\label{p0}
\left| \ \frac{\partial F(z)-\partial F(0)-\snr{z}^{p-2}z}{\snr{z}^{p-1}} \ \right|\to_{\snr{z}\to 0}0\,,
\end{eqnarray}
which means that we can find a function $\omega\colon (0,\infty)\to (0,\infty)$ such that
\begin{eqnarray}\label{p0.1}
\snr{z}\le \omega(s) \ \Longrightarrow \ \snr{\partial F(z)-\partial F(0)-\snr{z}^{p-2}z}\le s\snr{z}^{p-1},
\end{eqnarray}
for every $z\in \mathbb{R}^{N\times n}$ and all $s\in (0,\infty)$.
 The linear ingredient $f\colon \Omega\to \mathbb{R}^{N}$ appearing in \eqref{fun} is such that
\begin{eqnarray}\label{f}
f\in L^{m}(\Omega,\mathbb{R}^{N})\quad \mbox{with} \ \ 2>m>\begin{cases}
\ 2n/(n+2)\quad &\mbox{if} \ \ n>2\\
\ 3/2\quad &\mbox{if} \ \ n=2
\end{cases}
\end{eqnarray}
holds. Let us briefly discuss some consequences of our assumptions. The validity of \eqref{f}, immediately implies that
\begin{eqnarray}\label{f.0}
f\in W^{1,p}(\Omega,\mathbb{R}^{N})^{*}\qquad \mbox{and}\qquad m'<2^{*}\stackrel{\eqref{pq}_{1}}{\le}p^{*}.
\end{eqnarray}
Of course, with $\Lambda$ being the constant in $\eqref{assf}_{2}$, we may always assume that $0<\lambda\le \Lambda$. From $\eqref{assf}_{2}$ and \eqref{qc} we deduce that
\begin{eqnarray}\label{df}
\snr{\partial F(z)}\le c\left[\snr{z}^{p-1}+\snr{z}^{q-1}\right],
\end{eqnarray}
with $c\equiv c(n,N, \Lambda,p,q)$, cf. \cite[proof of Theorem 2.1]{ma3}. Finally, \eqref{sqc} yields that for all $z\in \mathbb{R}^{N\times n}$, $\xi\in \mathbb{R}^{N}$, $\zeta\in \mathbb{R}^{n}$ it holds that
\begin{eqnarray}\label{sqc.1}
\partial^{2}F(z)\langle\xi\otimes \zeta,\xi\otimes \zeta\rangle\ge 2\lambda\snr{z}^{p-2}\snr{\xi}^{2}\snr{\zeta}^{2},
\end{eqnarray}
see \cite[Chapter 5]{giu} and \cite[Lemma 7.14]{ts1}.
\begin{remark}\label{rpsqc}
\emph{A stronger notion of quasiconvexity than \eqref{qc} prescribes that}
\begin{flalign}\label{qcp}
\mint_{B_{1}(0)}F(z+D\varphi) \ \dx\ge F(z)\qquad \mbox{holds for all} \ \ z\in \mathbb{R}^{N\times n}, \ \ \varphi\in W^{1,p}_{0}(B_{1}(0),\mathbb{R}^{N}),
\end{flalign}
\emph{which can be further strengthened into}
\begin{flalign}\label{sqcp}
\int_{B}\left[F(z+D\varphi)-F(z)\right] \dx\ge \lambda\int_{B}(\snr{z}^{2}+\snr{D\varphi}^{2})^{\frac{p-2}{2}}\snr{D\varphi}^{2} \dx\qquad \mbox{for all} \ \ z\in \mathbb{R}^{N\times n}, \ \ \varphi\in W^{1,p}_{0}(B,\mathbb{R}^{N}),
\end{flalign}
\emph{Condition \eqref{qcp} is known in the literature as $W^{1,p}$-quasiconvexity \cite{bamu}, and it is necessary for the lower semicontinuity of functional. Condition \eqref{sqcp}, the strict counterpart of \eqref{qcp}, plays a crucial role in partial regularity for degenerate elliptic problems \cite{dumi,ts2}. By basic density arguments, \eqref{qcp}-\eqref{sqcp} are equivalent to \eqref{qc}-\eqref{sqc} if the integrand $F(\cdot)$ has standard $p$-growth. On the other hand, for integrands with $(p,q)$-growth, it seems that the only way to relate \eqref{qc}-\eqref{sqc} to \eqref{qcp}-\eqref{sqcp} involves the Lebesgue-Serrin-Marcellini extension of functional $w\mapsto \int F(Dw) \ \dx$, but it is unclear whether these two objects may be related to each other as in the convex setting or not.}
\end{remark}
\begin{remark}\label{assumptions}
\emph{Conditions $\eqref{assf}_{2,3}$ are assumed in their full strength in order to keep at a reasonable level the considerable amount of technicalities involved in this paper. In fact, instead of $\eqref{assf}_{2}$ we can assume
\begin{eqnarray}\label{assf.1.1}
\Lambda^{-1}\snr{z}^{p}\le F(z)\le \Lambda(1+\snr{z}^{q}),
\end{eqnarray}
while $\eqref{assf}_{3}$ can be relaxed by imposing power growth on second derivatives only for small values of their argument: there is $\delta_{1}\equiv \delta_{1}(F(\cdot))>0$ such that
\begin{eqnarray}\label{mm.1}
\snr{z}\le \delta_{1} \ \Longrightarrow \ \snr{\partial^{2}F(z)}<c\snr{z}^{p-2},
\end{eqnarray}
holds true for some $c\equiv c(F(\cdot))$. In turn, $\eqref{assf}_{1}$ and \eqref{mm.1} imply that whenever $L>0$ is a positive constant, we have
\begin{eqnarray}\label{mm.1.1}
\snr{z}\le L+1 \ \Longrightarrow \ \snr{\partial^{2}F(z)}\le c\snr{z}^{p-2},
\end{eqnarray}
with $c\equiv c(F(\cdot),p,L)$, see \cite[Section II]{af} and \cite[Remark 4.2]{ts2}. Combining \eqref{assf.1.1} with \eqref{qc}, we obtain $\snr{\partial F(z)}\le c(1+\snr{z}^{q-1})$ for $c\equiv c(n,N,\Lambda,p,q)$ and, using also \eqref{mm.1.1} we get that whenever $z_{0}\in \mathbb{R}^{N\times n}$ verifies $\snr{z_{0}}\le L+1$, it is
\begin{eqnarray*}
\left\{
\begin{array}{c}
\displaystyle 
\ \snr{F(z_{0}+z)-F(z_{0})-\langle\partial F(z_{0}),z\rangle}\le c\left(\snr{V_{\snr{z_{0}},p}(z)}^{2}+\snr{V_{\snr{z_{0}},q}(z)}^{2}\right) \\[8pt]\displaystyle
\ \snr{\partial F(z_{0}+z)-\partial F(z_{0})}\le c\snr{z}^{-1}\left(\snr{V_{\snr{z_{0}},p}(z)}^{2}+\snr{V_{\snr{z_{0}},q}(z)}^{2}\right) ,
\end{array}
\right.
\end{eqnarray*}
for all $z\in \mathbb{R}^{N\times n}$, with $c\equiv c(n,N,\Lambda,p,q,F(\cdot),L)$, cf. \cite[Lemma 4.3]{ts2}. Moreover, $\eqref{assf}_{4}$ can be weakened by assuming its validity only in a neighborhood of the origin (of course in this case the term raised to the $(p-2)$-power dominates the one at the $(q-2)$-power). Precisely, we shall ask that there is $\delta_{2}\equiv \delta_{2}(F(\cdot))>0$ and a modulus of continuity $\mu(\cdot)$ such that for $\snr{z_{1}},\snr{z_{2}}\le \delta_{2}$ we have
\begin{flalign}\label{mm.2}
\snr{\partial^{2}F(z_{1})-\partial^{2}F(z_{2})}\le \mu\left(\frac{\snr{z_{2}-z_{1}}}{\snr{z_{2}}+\snr{z_{1}}}\right) \left(\snr{z_{1}}^{2}+\snr{z_{2}}^{2}\right)^{\frac{p-2}{2}}
\end{flalign}
and, via \eqref{mm.2} and $\eqref{assf}_{1}$ we deduce that, whenever $L>0$ is a constant, $\partial^{2}F(\cdot)$ is uniformly continuous on the strip $\left\{z\in \mathbb{R}^{N\times n}\colon \delta_{2}/2\le \snr{z}\le 2L\right\}$, which means that \eqref{mm.2} actually holds for all $z_{1},z_{2}\in \mathbb{R}^{N\times n}$ with $0<\snr{z_{1}}\le L$ and $0\le \snr{z_{2}}\le 2L$, with a modulus of continuity $\mu_{L}(\cdot)$ depending on $L$, see \cite[Remark 5.4]{ts2}. Clearly, there is no loss of generality in assuming that \eqref{mm.1} and \eqref{mm.2} admit the same smallness thresholds, i.e.: $\delta_{1}=\delta_{2}$. This means that all our results are true also with \eqref{assf.1.1}, \eqref{mm.1} and \eqref{mm.2} replacing $\eqref{assf}_{2}$, $\eqref{assf}_{3}$ and $\eqref{assf}_{4}$ respectively - the price to pay is that in this new framework all the bounding constants depend on $L$, which will be always chosen as a function of $M$ (the limiting constant controlling averages appearing in Sections \ref{ndeg.1}-\ref{pgc}) in an unquantified way.
}
\end{remark}

\subsection{The Lebesgue-Serrin-Marcellini extension}\label{lex}
Let $\mathbb{B}_{\Omega}$ be the family of all open subsets of $\Omega$ and $B\in \mathbb{B}_{\Omega}$. For $1<p\le q<\infty$, a continuous integrand $F(\cdot)$, and functions $f\in W^{1,p}(\Omega,\mathbb{R}^{N})^{*}$, $w\in W^{1,p}(\Omega,\mathbb{R}^{N})$, the Lebesgue-Serrin-Marcellini extension of functional \eqref{fun} is defined as
\begin{eqnarray*}
\bar{\mathcal{F}}(w;B):=\inf_{\{w_{j}\}_{j\in \N}\in \mathcal{C}(w;B)}\liminf_{j\to \infty}\mathcal{F}(w_{j};B),
\end{eqnarray*}
with
\begin{eqnarray*}
\mathcal{C}(w;B):=\left\{\{w_{j}\}_{j\in \N}\subset W^{1,q}_{\loc}(B,\mathbb{R}^{N})\cap W^{1,p}(B,\mathbb{R}^{N})\colon w_{j}\rightharpoonup w \ \mbox{weakly in} \ W^{1,p}(B,\mathbb{R}^{N})\right\}.
\end{eqnarray*}

Let us remark that $\mathcal{C}(w;B)\not =\{\emptyset\}$ as smooth functions are $W^{1,p}$-dense. Moreover, since the datum $f$ belongs to the dual of $W^{1,p}(B,\mathbb{R}^{N})$ cf. \eqref{f}-\eqref{f.0}, and so $w\mapsto \int_{B}f\cdot w \ \dx$ is a linear, continuous functional, it is possible to simplify the above definition by noticing that it yields
\begin{eqnarray}\label{ls.7}
\bar{\mathcal{F}}(w;B)=\bar{\mathcal{F}}_{0}(w;B)-\int_{B}f\cdot w \ \dx,
\end{eqnarray}
where $\bar{\mathcal{F}}_{0}(\cdot)$ is the Lebesgue-Serrin-Marcellini extension of $\mathcal{F}_{0}(w;B):=\int_{B}F(Dw) \ \dx$. 
Before proceeding further, let us recall that the quasiconvex envelope of a continuous integrand $F(\cdot)$ is defined as:
\begin{eqnarray*}
\mathbb{R}^{N\times n}\ni z\mapsto QF(z):=\inf_{\varphi\in C^{\infty}_{c}(B_{1}(0),\mathbb{R}^{N})}\int_{B_{1}(0)}F(z+D\varphi) \ \dx.
\end{eqnarray*}
It is well-known that if $F(\cdot)$ satisfies \eqref{qc}, then
\eqn{ls.8}
$$
QF(z)=F(z)\qquad \mbox{for all} \ \ z\in \mathbb{R}^{N\times n}
$$
and, if in addition $F(z)\le \Lambda(1+\snr{z}^{q})$ and $1<p\le q<np/(n-1)$, it is 
\eqn{ls.16}
$$
\bar{\mathcal{F}}_{0}(w;B)=\int_{B}QF(Dw) \ \dx=\mathcal{F}_{0}(w;B)\qquad \mbox{for all} \ \ B\in \mathbb{B}_{\Omega}, \ \ w\in W^{1,q}(B,\mathbb{R}^{N}),
$$
see \cite{foma,ts2,ts1}, thus $\bar{\mathcal{F}}_{0}(\cdot)$ can be seen as a way to extend by semicontinuity $\mathcal{F}_{0}(\cdot)$ from $W^{1,q}_{\loc}$ to $W^{1,p}$, \cite{ma2}. We stress that already in the quasiconvex setting $\mathcal{F}_{0}(\cdot)$ does not necessarily coincide with its Lebesgue-Serrin-Marcellini extension on $W^{1,p}$, \cite{foma}. Following \cite{bfm,foma,ts2,ts1}, we list some relevant properties of $\bar{\mathcal{F}}_{0}(\cdot)$. We first report the measure representation result from \cite{bfm,foma}.
\begin{proposition}\label{ls.p1}
Assume that $1<p\le q<np/(n-1)$ and $F\colon \mathbb{R}^{N\times n}\to \mathbb{R}$ is a continuous integrand so that $F(z)\le \Lambda(1+\snr{z}^{q})$. Then, for every $w\in W^{1,p}(\Omega,\mathbb{R}^{N})$ with $\bar{\mathcal{F}}_{0}(w;\Omega)<\infty$ there exists a uniquely determined finite outer Radon measure $\mu_{w}$ on $\Omega$ so that $\bar{\mathcal{F}}_{0}(w;\cdot)=\left.\mu_{w}\right|_{\mathbb{B}_{\Omega}}$ whose absolutely continuous part has density equal to $QF(Dw)$ with respect to the Lebesgue measure $\mathcal{L}^{n}$ on $\Omega$, that is $$\frac{\d\mu_{w}}{\d\mathcal{L}^{n}}=QF(Dw)\quad \mbox{holds almost everywhere on} \ \ \Omega.$$
\end{proposition}
Next, we recall from \cite{foma,ts2,ts1} the semicontinuity features of $\bar{\mathcal{F}}_{0}(\cdot)$.
\begin{lemma}\label{ls.l1}
Let $\Omega\subset \mathbb{R}^{n}$ be an open, bounded domain with Lipschitz boundary and $F\colon \mathbb{R}^{N\times n}\to \mathbb{R}$ be a continuous integrand so that $F(z)\ge \Lambda^{-1}\snr{z}^{p}$. Then, $\bar{\mathcal{F}}_{0}(\cdot;\Omega)$ is sequentially weakly lower semicontinuous on $W^{1,p}(\Omega,\mathbb{R}^{N})$.
\end{lemma}
We finally collect from \cite{ts2,ts1} several properties of $\bar{\mathcal{F}}_{0}(\cdot)$ that will be crucial for deriving a Caccioppoli type inequality for minimizers of \eqref{exfun}. 
\begin{lemma}\label{ls.l2}
Let $B\in \mathbb{B}_{\Omega}$, $F\in C^{1}_{\loc}(\mathbb{R}^{N\times n})$ be an integrand verifying $\eqref{assf}_{2}$ and \eqref{df} with $1<p\le q<\min\{np/(n-1),p+1\}$. Then, if $w\in W^{1,p}(\Omega,\mathbb{R}^{N})$ is so that $\bar{\mathcal{F}}_{0}(w;\Omega)<\infty$, the following holds.
\begin{itemize}
    \item For all $\varphi\in W^{1,\frac{p}{p+1-q}}(\Omega,\mathbb{R}^{N})$, $B\in \mathbb{B}_{\Omega}$ it is
    \begin{eqnarray}\label{ls.3}
    \bar{\mathcal{F}}_{0}(w+\varphi;B)-\bar{\mathcal{F}}_{0}(w;B)=\int_{B}\left[QF(Dw+D\varphi)-QF(Dw)\right] \ \dx.
    \end{eqnarray}
    \item For all $B\in \mathbb{B}_{\Omega}$, $z_{0}\in\mathbb{R}^{N\times n}$, $v_{0}\in \mathbb{R}^{N}$ and $\varphi\in W^{1,p}(B,\mathbb{R}^{N})$ with $\supp(\varphi)\Subset B$ we have
    \begin{eqnarray}\label{ls.4}
    \bar{\mathcal{F}}_{0}(\ell+\varphi;B)\ge \bar{\mathcal{F}}_{0}(\ell;B),
    \end{eqnarray}
    where $\ell(x):=v_{0}+\langle z_{0},x-x_{0}\rangle$.
    \item If $B_{\rr}(x_{0})\Subset \Omega$ is a ball and $w$ satisfies also the boundary regularity condition
    \begin{flalign}\label{ls.5}
     \limsup_{\sigma\to 0}\frac{1}{\sigma^{\delta}}\int_{B_{\rr+\sigma}(x_{0})\setminus B_{\rr-\sigma}(x_{0})}\snr{Dw}^{p} \ \dx<\infty,
    \end{flalign}
    for some $\delta\in \left(n(q-p)/q,1\right]$, then
    \begin{flalign}\label{ls.6}
     \bar{\mathcal{F}}_{0}(w;\Omega)=\bar{\mathcal{F}}_{0}(w;B_{\rr}(x_{0}))+\bar{\mathcal{F}}_{0}(w;\Omega\setminus \bar{B}_{\rr}(x_{0})).
    \end{flalign}
\end{itemize}
\end{lemma}
We stress that some of the results appearing in Lemma \ref{ls.l2} hold under slightly more general assumptions than those recorded above, but this is not an issue as such results will be applied in the setting described in Section \ref{assec}.
\subsection{A significant model example}
A reference model in nonlinear elasticity \cite{bb,bamu,ma2,ma5} is the variational integral
\begin{eqnarray*}
W^{1,p}(\Omega,\mathbb{R}^{n})\ni w\mapsto \mathcal{F}(w;\Omega):=\int_{\Omega}\left[\snr{Dw}^{p}+H(\det(Dw))-f\cdot w\right] \ \dx,
\end{eqnarray*}
where $f\in L(n,1)$ is any map, $H(t):=(1+t^{2})^{\frac{q}{2n}}$ for all $t\in \mathbb{R}$ and
\begin{eqnarray}\label{mm.6}
n\ge 3,\qquad q=n,\qquad n-\frac{1}{n}<p\le n.
\end{eqnarray}
Notice that $H\in C^{2}(\mathbb{R})$, it is convex and has linear growth at infinity - such aspect makes functional $\mathcal{F}(\cdot)$ physically interesting, \cite[Section 1]{ma5}. Setting $ F(z):=\snr{z}^{p}+H(\det(z))$ for all $z\in \mathbb{R}^{n\times n}$, we are going to check the validity of the conditions listed in Remark \ref{assumptions}. It is immediate to see that $F(\cdot)$ satisfies $\eqref{assf}_{1}$ and \eqref{assf.1.1}. It is also well-known that, with $H(\cdot)$ convex, the composition $H( \det(\cdot))$ is polyconvex, therefore quasiconvex in the sense of \eqref{qc}, see \cite[Chapter 5]{giu}. This, together with the strict convexity of $z\mapsto \snr{z}^{p}$ yields that $F(\cdot)$ verifies \eqref{sqc} for some positive $\lambda$ depending on $(n,p)$, cf. \eqref{ls.2} below. Next, a straightforward computation shows that
\begin{flalign}\label{mm.4}
\snr{\partial F(z)}\lesssim \left(\snr{z}^{p-1}+\snr{z}^{n-1}\right)\qquad \mbox{and}\qquad \snr{\partial^{2}F(z)}\lesssim\left(\snr{z}^{p-2}+\snr{z}^{n-2}+\snr{z}^{2n-2}\right),
\end{flalign}
with constants implicit in "$\lesssim$" depending on $(n,p)$. At this stage, it is worth observing that \eqref{mm.1}, \eqref{mm.2} and \eqref{p0} are implied by
\begin{eqnarray}\label{mm.10}
\lim_{\snr{z}\to 0}\frac{\snr{\partial^{2}F(z)-\partial^{2}(\snr{z}^{p})}}{\snr{z}^{p-2}}=0,
\end{eqnarray}
cf. \cite[Section 2]{ts2}, therefore we only need to check the validity of \eqref{mm.10} for the model integrand $F(\cdot)$. Since
\begin{eqnarray*}
\frac{\snr{\partial^{2}F(z)-\partial ^{2}(\snr{z}^{p})}}{\snr{z}^{p-2}}\stackrel{\eqref{mm.4}_{2}}{\lesssim}\snr{z}^{2n-p}+\snr{z}^{n-p}\stackrel{\eqref{mm.6}}{\to}_{\snr{z}\to 0}0,
\end{eqnarray*}
we have that \eqref{mm.10} holds true, so our sharp partial regularity results cover the Lebesgue-Serrin-Marcellini extension of the model functional $\mathcal{F}(\cdot)$. The example described above in particular suggests that the right approach to a rigorous analysis of functional $\mathcal{F}(\cdot)$ goes through its extension by relaxation \cite{ma2,ma5}. In fact in \eqref{mm.6} $p$ is allowed to be smaller than the space dimension $n$, which means that minimizers and competitors possibly admit discontinuities, the phenomenon of cavitation\footnote{Existence of equilibrium solutions with cavities, i.e. minima of $\mathcal{F}(\cdot)$ that are discontinuous at one point where a cavity forms, \cite{bb,ma2,ma5}.} may occur and the $W^{1,p}$-quasiconvexity of $\mathcal{F}_{0}(w,\Omega):=\int_{\Omega}F(Dw) \ \dx$ fails by \cite[Theorem 4.1]{bamu}; while its relaxation $\bar{\mathcal{F}}_{0}(\cdot)$ preserves such property under the only restriction $p>n-1$, see \cite[Lemma 7.6]{ts1}. We can then prove that relaxed minimizers of $\mathcal{F}(\cdot)$ are almost everywhere regular also in presence of cavitation. We refer to \cite{ma2,ts,ts2,ts1} for further discussions and more general examples.
\subsection{Tools for $p$-Laplacean problems} When dealing with $p$-Laplacean type problems, we shall often use the auxiliary vector field $V_{s,p}\colon \er^{N\times n} \to  \er^{N\times n}$, defined by
\begin{flalign*}
V_{s,p}(z):= (s^{2}+|z|^{2})^{(p-2)/4}z, \qquad p\in (1,\infty), \ \ s\ge 0, \ \ z\in \mathbb{R}^{N\times n}.
\end{flalign*}
If $s=0$ we simply write $V_{s,p}(\cdot)\equiv V_{p}(\cdot)$. A couple of useful related inequalities are
\begin{flalign}\label{Vm}
\snr{V_{s,p}(z_{1})-V_{s,p}(z_{2})}\approx (s^{2}+\snr{z_{1}}^{2}+\snr{z_{2}}^{2})^{(p-2)/4}\snr{z_{1}-z_{2}},
\end{flalign}
and
\begin{eqnarray}\label{equiv.1}
\snr{V_{s,p}(z)}^{2}\approx s^{p-2}\snr{z}^{2}+\snr{z}^{p}\qquad \mbox{with} \ \ p\ge 2,
\end{eqnarray}
where the equivalence holds up to constants depending only on $n,N,p$. An important property which is usually related to such field is recorded in the following lemma.
\begin{lemma}\label{l6}
Let $t>-1$, $s\in [0,1]$ and $z_{1},z_{2}\in \mathbb{R}^{N\times n}$ be such that $s+\snr{z_{1}}+\snr{z_{2}}>0$. Then
\begin{flalign*}
\int_{0}^{1}\left[s^2+\snr{z_{1}+y(z_{2}-z_{1})}^{2}\right]^{\frac{t}{2}} \ \dy\approx (s^2+\snr{z_{1}}^{2}+\snr{z_{2}}^{2})^{\frac{t}{2}},
\end{flalign*}
with constants implicit in "$\approx$" depending only on $n,N,t$. 
\end{lemma}
With $B\in \mathbb{B}_{\Omega}$ and $p\in (1,\infty)$, Lemma \ref{l6} directly implies that for all $z_{0}\in \mathbb{R}^{N\times n}$ and any $\varphi\in W^{1,p}_{0}(\Omega,\mathbb{R}^{N})$ it holds that
\begin{eqnarray}\label{ls.2}
\tilde{c}^{-1}\int_{B}(\snr{z_{0}}^{2}+\snr{D\varphi}^{2})^{\frac{p-2}{2}}\snr{D\varphi}^{2} \ \dx &\le& \int_{B}\left[\snr{z_{0}+D\varphi}^{p}-\snr{z_{0}}^{p}\right] \ \dx\nonumber \\
&\le& \tilde{c}\int_{B}(\snr{z_{0}}^{2}+\snr{D\varphi}^{2})^{\frac{p-2}{2}}\snr{D\varphi}^{2} \ \dx,
\end{eqnarray}
with $\tilde{c}\equiv \tilde{c}(n,N,p)$. Throughout the paper, for $B_{\rr}(x_{0})\Subset \Omega$, $w\in W^{1,p}(B_{\rr}(x_{0}),\mathbb{R}^{N})$ and $z_{0}\in \mathbb{R}^{N\times n}$, we shall several time use the excess type functional
\begin{eqnarray}\label{excess}
\mf{F}(w,z_{0};B_{\rr}(x_{0})):=\left(\mint_{B_{\rr}(x_{0})}\left[\snr{z_{0}}^{p-2}\snr{Dw-z_{0}}^{2}+\snr{Dw-z_{0}}^{p}\right] \dx\right)^{\frac{1}{p}}
\end{eqnarray}
and its "quadratic" version
\begin{eqnarray*}
\ti{\mf{F}}(w,z_{0};B_{\rr}(x_{0})):=\left(\mint_{B_{\rr}(x_{0})}\snr{V_{p}(Dw)-z_{0}}^{2} \dx\right)^{1/2}.
\end{eqnarray*}
When $z_{0}=(Dw)_{B_{\rr}(x_{0})}$ in $\mf{F}(\cdot)$ or $z_{0}=(V_{p}(Dw))_{B_{\rr}(x_{0})}$ in $\ti{\mf{F}}(\cdot)$ we shall simply write 
\begin{flalign*}
\mf{F}(w,(Dw)_{B_{\rr}(x_{0})};B_{\rr}(x_{0}))\equiv \mf{F}(w;B_{\rr}(x_{0})),\quad \quad \ti{\mf{F}}(w,(V_{p}(Dw))_{B_{\rr}(x_{0})};B_{\rr}(x_{0}))\equiv \ti{\mf{F}}(w;B_{\rr}(x_{0}))
\end{flalign*}
respectively. As a direct consequence of \eqref{minav} and of \cite[(2.6)]{gm} it holds that
\eqn{equiv.11}
$$
\tilde{\mathfrak{F}}(w;B_{\rr}(x_{0}))\approx \tilde{\mathfrak{F}}(w,V_{p}((Dw)_{B_{\rr}(x_{0})});B_{\rr}(x_{0}))
$$
so by \eqref{Vm} we have
\begin{eqnarray}\label{equiv}
\mf{F}(w;B_{\rr}(x_{0}))^{p/2}\approx \ti{\mf{F}}(w;B_{\rr}(x_{0})).
\end{eqnarray}
In \eqref{equiv.11}-\eqref{equiv}, the constants implicit in "$\approx$" depend on $(n,N,p)$. We further recall a straightforward variation of \cite[Lemmas 3.1 and 6.2]{kumi}.
\begin{lemma}
Let $p\ge 2$ be a number, $B_{\rr}(x_{0})\subset \mathbb{R}^{n}$ be a ball, and $w\in W^{1,p}(B_{\rr}(x_{0}),\mathbb{R}^{N})$ be any function. With $\nu\in (0,1)$ it holds that
\begin{flalign}\label{tri.1}
\mf{F}(w;B_{\nu\rr}(x_{0}))^{p/2}\le \frac{2^{3p}}{\nu^{n/2}}\mf{F}(w;B_{\rr}(x_{0}))^{p/2}
\end{flalign}
and
\begin{eqnarray}\label{tri.1.1}
\left|\snr{(Dw)_{B_{\nu\rr}(x_{0})}}^{p/2}-\snr{(Dw)_{B_{\rr}(x_{0})}}^{p/2}\right|\le \frac{2^{3p}\mf{F}(w;B_{\rr}(x_{0}))^{p/2}}{\nu^{n/2}}.
\end{eqnarray}
Moreover, if for $\sigma\le \rr$ there is $\kk\in \N\cup \{0\}$ satisfying $\nu^{\kk+1}\rr<\sigma\le \nu^{\kk}\rr$, then
\begin{eqnarray}\label{ls.42.1}
\mf{F}(w;B_{\nu^{\kk+1}\rr}(x_{0}))^{p/2}\le \frac{2^{3p}}{\nu^{n/2}}\mf{F}(w;B_{\sigma}(x_{0}))^{p/2}\le \frac{2^{6p}}{\nu^{n}}\mf{F}(w;B_{\nu^{\kk}\rr}(x_{0}))^{p/2}
\end{eqnarray}
and
\begin{eqnarray}\label{tri.1.2}
\left\{
\begin{array}{c}
\displaystyle 
\ \snr{(Dw)_{B_{\nu^{\kk+1}\rr}(x_{0})}}^{p/2}\le \frac{2^{3p}}{\nu^{n/2}}\mf{F}(w;B_{\sigma}(x_{0}))^{p/2}+\snr{(Dw)_{B_{\sigma}(x_{0})}}^{p/2} \\[8pt]\displaystyle
\ \snr{(Dw)_{B_{\sigma}(x_{0})}}^{p/2}\le \frac{2^{3p}}{\nu^{n/2}}\mf{F}(w;B_{\nu^{\kk}\rr}(x_{0}))^{p/2}+\snr{(Dw)_{B_{\nu^{\kk}\rr}(x_{0})}}^{p/2}.
\end{array}
\right.
\end{eqnarray}
\end{lemma}
Now, let $B_{\rr}(x_{0})\subset \mathbb{R}^{n}$ be a ball, $w\in L^{2}(B_{\rr}(x_{0}),\mathbb{R}^{N})$ be any function and $\ell_{\rr;x_{0}}$ be the unique affine function realizing the distance of $w$ from the space of affine functions, i.e.:
\begin{eqnarray*}
\ell_{\rr;x_{0}}\mapsto \min_{\ell \ \textnormal{affine}}\mint_{B_{\rr}(x_{0})}\snr{w-\ell}^{2} \dx.
\end{eqnarray*}
It is $\ell_{\rr;x_{0}}:=(w)_{B_{\rr}(x_{0})}+\langle D\ell_{\rr;x_{0}},x-x_{0}\rangle$ with
\begin{eqnarray*}
D\ell_{\rr;x_{0}}:=\frac{n+2}{\rr^{2}}\mint_{B_{\rr}(x_{0})}w(x)\otimes (x-x_{0}) \dx.
\end{eqnarray*}
Let us recall some well-known properties of $\ell_{\rr;x_{0}}$, see \cite[Section 2]{kumi}.
\begin{lemma}
Let $p\ge 2$ and $B_{\rr}(x_{0})\subset \mathbb{R}^{n}$ be a ball. For all $w\in L^{p}(B_{\rr}(x_{0}),\mathbb{R}^{N})$ the following inequalities hold true:
\begin{flalign}\label{ll.0}
\snr{D\ell_{\rr;x_{0}}-D\ell_{\nu\rr;x_{0}}}\le \frac{c}{(\nu\rr)^{p}}\mint_{B_{\nu\rr}(x_{0})}\snr{w-\ell_{\rr;x_{0}}}^{p} \dx\qquad \mbox{for any} \ \ \nu\in (0,1)
\end{flalign} 
and 
\begin{flalign}\label{ll.0.1}
\mint_{B_{\rr}(x_{0})}\snr{w-\ell_{\rr;x_{0}}}^{p}\le c\mint_{B_{\rr}(x_{0})}\snr{w-\ell}^{p} \dx\qquad \mbox{for any affine} \ \ \ell,
\end{flalign}
with $c\equiv c(n,N,p)$. Moreover, if $w\in W^{1,p}(B_{\rr}(x_{0}),\mathbb{R}^{N})$
\begin{flalign}\label{ll.1}
\snr{D\ell_{\rr;x_{0}}-(Dw)_{B_{\rr}(x_{0})}}^{p}\le c\mint_{B_{\rr}(x_{0})}\snr{Dw-(Dw)_{B_{\rr}(x_{0})}}^{p} \dx,
\end{flalign} 
for $c\equiv c(n,p)$.
\end{lemma}
Finally, a couple of classical iteration lemmas, cf. \cite[Lemmas 8.18 and 5.13]{giamar} respectively.
\begin{lemma}\label{l5}
Let $h\colon [\rr_{0},\rr_{1}]\to \mathbb{R}$ be a non-negative and bounded function, and let $\theta \in (0,1)$, $A,B,\gamma_{1},\gamma_{2}\ge 0$ be numbers. Assume that $h(t)\le \theta h(s)+A(s-t)^{-\gamma_{1}}+B(s-t)^{-\gamma_{2}}$ holds for all $\rr_{0}\le t<s\le \rr_{1}$. Then the following inequality holds $h(\rr_{0})\le c(\theta,\gamma_{1},\gamma_{2})[A(\rr_{1}-\rr_{0})^{-\gamma_{1}}+B(\rr_{1}-\rr_{0})^{-\gamma_{2}}].$
\end{lemma}
\begin{lemma}\label{iter+}
Let $h$ be a nonnegative function satisfying $h(\sigma_{1})\le c_{1}[(\sigma_{1}/\sigma_{2})^{\alpha}+\varepsilon]h(\sigma_{2}) +c_{2}\sigma_{2}^{\beta}$, for some positive, absolute constants $c_{1},c_{2},\alpha,\beta,\varepsilon$, with $\alpha>\beta$ and for all $0<\sigma_{1}\le \sigma_{2}\le \rr$, with $\rr$ being a given number. Assume further that, if $\sigma_{1}$ satisfies $\nu^{j+1}\sigma_{2}<\sigma_{1}\le \nu^{j}\sigma_{2}$ for some $\nu\in (0,1)$, $j\in \N\cup \{0\}$, then $h(\sigma_{1})\le c_{3}h(\sigma_{2}),$ for some positive constant $c_{3}$ depending on $\nu$ and on the structure of $h$. There exist constants $\varepsilon_{0}\equiv \varepsilon_{0}(c_{1},\alpha,\beta)$, $c\equiv c(c_{1},\alpha,\beta,c_{3})$ such that if $\varepsilon\le \varepsilon_{0}$ it is $h(\sigma_{1})\le c\sigma_{1}^{\beta}(\sigma_{2}^{-\beta}h(\sigma_{2})+c_{2}),$ for all $0<\sigma_{1}\le \sigma_{2}\le \rr$.
\end{lemma}

\subsection{Harmonic approximation lemmas}\label{harsec}
In this section we shall collect some basic properties of $\mathcal{A}$-harmonic maps and of $p$-harmonic maps. Let $\mathcal{A}$ be a constant bilinear form on $\mathbb{R}^{N\times n}$ verifying
\begin{eqnarray}\label{h.0}
\snr{\mathcal{A}}\le L\qquad \mbox{and}\qquad \mathcal{A}\langle \xi\otimes\zeta,\xi\otimes\zeta \rangle \ge L^{-1}\snr{\xi}^{2}\snr{\zeta}^{2}, 
\end{eqnarray}
where $L\ge 1$ is an absolute constant and $\xi\in \mathbb{R}^{N}$, $\zeta\in \mathbb{R}^{n}$ are vectors. By $\mathcal{A}$-harmonic map on an open set $\Omega\subset \mathbb{R}^{n}$ we mean a function $h\in W^{1,2}(\Omega,\mathbb{R}^{N})$ such that
\begin{eqnarray*}
\int_{\Omega}\mathcal{A}\langle Dh ,D\varphi\rangle \dx=0\qquad \mbox{for all} \ \ \varphi\in C^{\infty}_{\rm c}(\Omega,\mathbb{R}^{N}).
\end{eqnarray*}
As one could expect, $\mathcal{A}$-harmonic maps enjoy good regularity properties. In fact, for $B_{\rr}(x_{0})\Subset \Omega$ with $\rr\in (0,1]$ and all $d>1$ it is
\begin{eqnarray}\label{h.1}
\rr^{d}\sup_{B_{\rr/2}(x_{0})}\snr{D^{2}h}^{d}\le c\mint_{B_{3\rr/4}(x_{0})}\snr{Dh}^{d} \dx,
\end{eqnarray}
with $c\equiv c(n,N,L,d)$, cf. \cite[Chapter 10]{giu}. Let us recall from \cite[Lemma 2.4]{kumi} the following version of $\mathcal{A}$-harmonic approximation lemma.
\begin{lemma}\label{ahar}
Let $B_{\rr}(x_{0})\Subset \Omega$ be a ball, $p\ge2$ be any number and $\mathcal{A}$ a bilinear form on $\mathbb{R}^{N\times n}$ satisfying \eqref{h.0}. Then for any $\varepsilon>0$, $\sigma\in (0,1]$, there exists a number $\delta\equiv \delta(n,N,L,p,\varepsilon)$ with the following property: if $v\in W^{1,2}(B_{\rr}(x_{0}),\mathbb{R}^{N})$ verifies
\begin{eqnarray*}
\mint_{B_{\rr}(x_{0})}\snr{Dv}^{2} \dx+\sigma^{p-2}\mint_{B_{\rr}(x_{0})}\snr{Dv}^{p} \dx\le 1
\end{eqnarray*}
and it is approximately $\mathcal{A}$-harmonic in the sense that
\begin{eqnarray}\label{ahar.0}
\left| \ \mint_{B_{\rr}(x_{0})}\mathcal{A}\langle Dv,D\varphi\rangle \dx \ \right|\le \delta\nr{D\varphi}_{L^{\infty}(B_{\rr}(x_{0}))}
\end{eqnarray}
holds for all $\varphi\in C^{\infty}_{\rm c}(B_{\rr}(x_{0}),\mathbb{R}^{N})$, then there exists a $\mathcal{A}$-harmonic map $h\in W^{1,2}(B_{\rr}(x_{0}),\mathbb{R}^{N})$ such that
\begin{eqnarray*}
\mint_{B_{3\rr/4}(x_{0})}\snr{Dh}^{2} \dx+\sigma^{p-2}\mint_{B_{3\rr/4}(x_{0})}\snr{Dh}^{p} \dx\le 8^{2np}
\end{eqnarray*}
and
\begin{eqnarray*}
\mint_{B_{3\rr/4}(x_{0})}\left|\frac{v-h}{\rr} \right|^{2}+\sigma^{p-2}\left|\frac{v-h}{\rr}\right|^{p} \dx\le \varepsilon.
\end{eqnarray*}
\end{lemma}
\begin{proof}
The proof is essentially contained in \cite[Lemma 3.2]{dumist}, where the authors consider a bilinear form satisfying the usual strict ellipticity condition i.e. $\mathcal{A}\langle\xi,\xi\rangle\gtrsim \snr{\xi}^{2}$, which is stronger than the Legendre-Hadamard one in \eqref{h.0}$_{2}$. However, the ellipticity of $\mathcal{A}$ is used only on the gradient of functions with zero boundary value, and, since $\mathcal{A}$ is constant, $\eqref{h.0}_{2}$ is equivalent to the usual $W^{1,2}$-coercivity by means of G{\aa}rding inequality \cite[Theorem 10.1]{giu}, see also \cite[Lemma 3.3]{dust}.
\end{proof}
Let us recall the definition of $p$-harmonic map, i.e. a function $h\in W^{1,p}(\Omega,\mathbb{R}^{N})$ such that
\begin{eqnarray*}
\int_{\Omega}\langle \snr{Dh}^{p-2}Dh,D\varphi\rangle \dx=0\qquad \mbox{for all} \ \ \varphi\in C^{\infty}_{\rm c}(\Omega,\mathbb{R}^{N}).
\end{eqnarray*}
Also $p$-harmonic maps have good regularity properties, \cite{uh,ur}: whenever $B_{\rr}(x_{0})\Subset B_{r}(x_{0})\Subset \Omega$ are concentric balls, it holds that
\begin{flalign}\label{h.2}
\sup_{x\in B_{\rr/2}(x_{0})}\snr{Dh}^{p}\le c\mint_{B_{\rr}(x_{0})}\snr{Dh}^{p} \dx \quad \mbox{and}\quad \mf{F}(h;B_{\rr}(x_{0}))\le c\left(\frac{\rr}{r}\right)^{\alpha}\mf{F}(h;B_{r}(x_{0})),
\end{flalign}
for some $\alpha\equiv \alpha(n,N,p)\in (0,1)$ with $c\equiv c(n,N,p)$. We conclude this section by recalling that almost $p$-harmonic maps can be approximated by genuine $p$-harmonic maps, \cite{dsv}.
\begin{lemma}\label{phar}
Let $p\ge 2$ be any number. For any $\varepsilon>0$ there exists a positive constant $\delta\equiv \delta(n,N,p,\varepsilon)\in (0,1]$ such that if $v\in W^{1,p}(B_{\rr}(x_{0}),\mathbb{R}^{N})$ with $\mf{I}_{p}(Dv;B_{\rr}(x_{0}))\le 1$ is approximately $p$-harmonic in the sense that
\begin{eqnarray*}
\left|\ \mint_{B_{\rr}(x_{0})}\langle \snr{Dv}^{p-2}Dv,D\varphi\rangle \dx \ \right|\le \delta\nr{D\varphi}_{L^{\infty}(B_{\rr}(x_{0}))}
\end{eqnarray*}
holds for all $\varphi\in C^{\infty}_{\rm c}(B_{\rr}(x_{0}),\mathbb{R}^{N})$, then there exists a $p$-harmonic map $h\in v+W^{1,p}_{0}(B_{\rr}(x_{0}),\mathbb{R}^{N})$ such that 
\begin{eqnarray*}
\mf{I}_{p}(Dh;B_{\rr}(x_{0}))\le c,\qquad \left(\mint_{B_{\rr}(x_{0})}\snr{Dh-Dv}^{\ti{p}} \dx\right)^{\frac{p}{\ti{p}}}\le c\varepsilon,\qquad  \mint_{B_{\rr}(x_{0})}\left|\frac{v-h}{\rr}\right|^{p} \dx\le c\varepsilon,
\end{eqnarray*}
where $\ti{p}:=\frac{2n}{n+2}$ if $p=2$, $\ti{p}:=\max\left\{2,\frac{np}{n+p}\right\}$ when $p>2$ and $c\equiv c(n,N,p)$.
\end{lemma}
\subsection{Existence of minima and the Euler-Lagrange system}
Since we assumed that $\partial \Omega$ is Lipschitz-regular, once assigned a boundary datum $u_{0}\in W^{1,p}( \Omega,\mathbb{R}^{N})$ so that $\bar{\mathcal{F}}_{0}(u_{0};\Omega)<\infty$ - keep \eqref{ls.7} in mind - the existence of a map $u\in u_{0}+W^{1,p}_{0}(\Omega,\mathbb{R}^{N})$ minimizing \eqref{exfun} follows by Lemma \ref{ls.l1} and direct methods as long as $f\in W^{1,p}(\Omega,\mathbb{R}^{N})^{*}$, the dual space of $W^{1,p}(\Omega,\mathbb{R}^{N})$ cf. \eqref{f.0}, and provided that $F(\cdot)$ is a continuous integrand satisfying $F(z)\ge \Lambda^{-1}\snr{z}^{p}$ for all $z\in \mathbb{R}^{N\times n}$, see \eqref{assf}$_{1,2}$. 
Now, if $u\in W^{1,p}(\Omega,\mathbb{R}^{N})$ is a local minimizer of \eqref{exfun}, by minimality it is natural to expect that $u$ weakly solves some sort of integral identity. This is indeed the case provided that the integrand $F(\cdot)$ is locally $C^{1}$-regular and \eqref{qc}, $\eqref{assf}_{2}$, \eqref{f} are verified with exponents $(p,q)$ so that $1<p\le q<\min\{np/(n-1),p+1\}$. In fact it is
\begin{eqnarray}\label{3}
\int_{\Omega}\left[\langle\partial F(Du), D\varphi\rangle-f\cdot\varphi \right] \dx=0\qquad \mbox{for all} \ \ \varphi\in C^{\infty}_{\rm c}(\Omega,\mathbb{R}^{N}).
\end{eqnarray}
Recalling \eqref{ls.7}, the proof of \eqref{3} goes as in \cite[Lemma 7.3]{ts1}.
\subsection{An extension lemma}
The following is a variant of the smoothing (with variable radius) result from \cite{foma}, obtained in \cite[Lemmas 4.4 and 4.6]{ts1}, which will be fundamental in order to construct suitable comparison maps for minima of \eqref{exfun}.
\begin{lemma}\label{exlem}
Let $0<\tau_{1}<\tau_{2}$ be two numbers and $B_{\tau_{2}}\Subset \Omega$ be a ball. There exists a bounded, linear smoothing operator $\mf{T}_{\tau_{1},\tau_{2}}\colon W^{1,1}(\Omega,\mathbb{R}^{N})\to W^{1,1}(\Omega,\mathbb{R}^{N})$ such that
\begin{eqnarray*}
W^{1,1}(\Omega,\mathbb{R}^{N})\ni w\mapsto \mf{T}_{\tau_{1},\tau_{2}}[w](x):=\mint_{B_{1}(0)}w(x+\vartheta(x)y) \ \dy,
\end{eqnarray*}
where it is $\vartheta(x):=\frac{1}{2}\max\left\{\min\left\{\snr{x}-\tau_{1},\tau_{2}-\snr{x}\right\},0\right\}$. If $w\in W^{1,p}(\Omega,\RN)$ for some $p\ge 1$, the map $\mf{T}_{\tau_{1},\tau_{2}}[w]$ has the following features:
\begin{itemize}
    \item[(\emph{i.})] $\mf{T}_{\tau_{1},\tau_{2}}[w]\in W^{1,p}(\Omega,\mathbb{R}^{N})$;
    \item[(\emph{ii.})] $w=\mf{T}_{\tau_{1},\tau_{2}}[w]$ almost everywhere on $(\Omega\setminus B_{\tau_{2}})\cup B_{\tau_{1}}$;
    \item[(\emph{iii.})] $\mf{T}_{\tau_{1},\tau_{2}}[w]\in w+W^{1,p}_{0}(B_{\tau_{2}}\setminus \bar{B}_{\tau_{1}},\mathbb{R}^{N})$;
    \item[(\emph{iv.})] $\snr{D\mf{T}_{\tau_{1},\tau_{2}}[w]}\le c(n)\mf{T}_{\tau_{1},\tau_{2}}[\snr{Dw}]$ almost everywhere in $\Omega$.
\end{itemize}
Moreover,
\begin{flalign}\label{ex.0}
\left\{
\begin{array}{c}
\displaystyle 
\ \nr{\mf{T}_{\tau_{1},\tau_{2}}[w]}_{L^{p}(B_{\tau_{2}}\setminus B_{\tau_{1}})}\le c\nr{w}_{L^{p}(B_{\tau_{2}}\setminus B_{\tau_{1}})} \\[8pt]\displaystyle
\ \nr{D\mf{T}_{\tau_{1},\tau_{2}}[w]}_{L^{p}(B_{\tau_{2}}\setminus B_{\tau_{1}})}\le c\nr{Dw}_{L^{p}(B_{\tau_{2}}\setminus B_{\tau_{1}})}\\[8pt]\displaystyle
\nr{D\mf{T}_{\tau_{1},\tau_{2}}[w]}_{L^{p}(B_{\varsigma}\setminus B_{\tau_{1}})}\le c\nr{Dw}_{L^{p}(B_{2\varsigma-\tau_{1}}\setminus B_{\tau_{1}})}\quad \mbox{for} \ \ \tau_{1}\le \varsigma\le (\tau_{1}+\tau_{2})/2\\[8pt]\displaystyle
\nr{D\mf{T}_{\tau_{1},\tau_{2}}[w]}_{L^{p}(B_{\tau_{2}}\setminus B_{\varsigma})}\le c\nr{Dw}_{L^{p}(B_{\tau_{2}}\setminus B_{2\varsigma-\tau_{2}})}\quad \mbox{for} \ \ (\tau_{1}+\tau_{2})/2\le \varsigma\le \tau_{2},
\end{array}
\right.
\end{flalign}
for $c\equiv c(n,p)$. Finally, let $\mathcal{N}\subset \mathbb{R}$ be a set with zero Lebesgue measure. There are 
\begin{flalign}\label{ex.2}
\ti{\tau}_{1}\in \left(\tau_{1},\frac{2\tau_{1}+\tau_{2}}{3}\right)\setminus \mathcal{N},\ \ti{\tau}_{2}\in \left(\frac{\tau_{1}+2\tau_{2}}{3},\tau_{2}\right)\setminus \mathcal{N} \quad \mbox{verifying} \ \ (\tau_{2}-\tau_{1})\approx (\ti{\tau}_{2}-\ti{\tau}_{1})
\end{flalign}
up to absolute constants, such that for all $p\le d<\frac{np}{n-1}$ it is
\begin{eqnarray}\label{ex.1}
\left\{
\begin{array}{c}
\displaystyle 
\ \nr{\mf{T}_{\ti{\tau}_{1},\ti{\tau}_{2}}[w]}_{L^{d}(B_{\ti{\tau}_{2}}\setminus B_{\ti{\tau}_{1}})}\le \frac{c}{(\tau_{2}-\tau_{1})^{n\left(\frac{1}{p}-\frac{1}{d}\right)}}\nr{w}_{L^{p}(B_{\tau_{2}}\setminus B_{\tau_{1}})} \\[17pt]\displaystyle
\ \nr{D\mf{T}_{\ti{\tau}_{1},\ti{\tau}_{2}}[w]}_{L^{d}(B_{\ti{\tau}_{2}}\setminus B_{\ti{\tau}_{1}})}\le \frac{c}{(\tau_{2}-\tau_{1})^{n\left(\frac{1}{p}-\frac{1}{d}\right)}}\nr{Dw}_{L^{p}(B_{\tau_{2}}\setminus B_{\tau_{1}})},
\end{array}
\right.
\end{eqnarray}
with $c\equiv c(n,p,d)$. Of course, operator $\mf{T}_{\ti{\tau}_{1},\ti{\tau}_{2}}$ satisfies properties \emph{(}i.\emph{)}-\emph{(}iv.\emph{)} and \eqref{ex.0} with $\ti{\tau}_{1}$, $\ti{\tau}_{2}$ replacing $\tau_{1}$, $\tau_{2}$.
\end{lemma}

\section{The nondegenerate scenario}\label{ndeg.1}
In this section we treat the nondegenerate scenario, i.e. the large gradient case. The first step for proving our partial regularity results is showing a Caccioppoli type inequality, see \cite[Lemma 7.13]{ts1} for the homogeneous case $f\equiv 0$.
\begin{lemma}\label{ndegc}
Assume \eqref{assf}$_{1,2,3}$, \eqref{pq}, \eqref{sqc} and \eqref{f}, let $B_{\rr}(x_{0})\Subset \Omega$ be any ball with radius $\rr\in (0,1]$, $u\in W^{1,p}(\Omega,\mathbb{R}^{N})$ be a local minimizer of \eqref{exfun} and $\ell(x):=v_{0}+\langle z_{0},x-x_{0}\rangle$ be any affine function with $v_{0}\in \mathbb{R}^{N}$ and $z_{0}\in \left(\mathbb{R}^{N\times n}\setminus \{0\}\right)\cap\left\{\snr{z}\le 80000(M+1)\right\}$, for some constant $M\ge 0$. Then it holds that
\begin{eqnarray}\label{cacc}
\mf{F}(u,z_{0};B_{\rr/2}(x_{0}))^{p}&\le &c\mathfrak{K}\left(\mint_{B_{\rr}(x_{0})}\snr{z_{0}}^{p-2}\left|\frac{u-\ell}{\rr}\right|^{2} +\left|\frac{u-\ell}{\rr}\right|^{p}\dx\right)\nonumber \\
&&+\frac{c}{\snr{z_{0}}^{p-2}}\left(\rr^{m}\mint_{B_{\rr}(x_{0})}\snr{f}^{m} \dx\right)^{\frac{2}{m}}+c\mathds{1}_{\{q>p\}}\mathfrak{F}(u,z_{0};B_{\rr}(x_{0}))^{q},
\end{eqnarray}
where $c\equiv c(\textnormal{\texttt{data}},M^{q-p})$ and $\mathfrak{K}(\cdot)$, $\mathds{1}_{\{q>p\}}$ are defined in \eqref{ik}-\eqref{ik.1} respectively.
\end{lemma}
\begin{proof}
Let us define the auxiliary integrand
\begin{eqnarray*}
\mathbb{R}^{N\times n}\ni z\mapsto G_{0}(z):=F(z)-c_{0}\snr{z}^{p},\qquad c_{0}:=\frac{\min\{\Lambda^{-1},\lambda\}}{2\max\{\tilde{c},1\}},
\end{eqnarray*}
where $\tilde{c}\equiv \tilde{c}(n,N,p)$, $\Lambda$ and $\lambda$ are the positive constants appearing in \eqref{ls.2}, \eqref{assf} and \eqref{sqc} respectively. By \eqref{sqc} and \eqref{ls.2} we easily get that $G_{0}(\cdot)$ satisfies \eqref{sqc} and, via $\eqref{assf}_{2}$ we have $\snr{z}^{p}\lesssim G_{0}(z)\lesssim \snr{z}^{p}+\snr{z}^{q}$, with constants implicit in "$\lesssim$" depending on $ (n,N,\lambda,\Lambda,p)$. For $B\in \mathbb{B}_{\Omega}$, we introduce functional $$W^{1,p}(B,\mathbb{R}^{N})\ni w\mapsto \mathcal{G}_{0}(w;B):=\int_{B}G_{0}(Dw) \ \dx$$
and its Lebesgue-Serrin-Marcellini extension $\bar{\mathcal{G}}_{0}(\cdot)$. Now if $\{w_{j}\}_{j\in \N}\in \mathcal{C}(w;B)$ is any sequence, the weak $W^{1,p}$-semicontinuity of $z\mapsto \snr{z}^{p}$ yield that $\bar{\mathcal{G}}_{0}(w;B)\le \liminf_{j\to \infty}\mathcal{F}_{0}(w_{j};B)-c_{0}\nr{Dw}_{L^{p}(B)}^{p}$, therefore the very definition of $\bar{\mathcal{F}}_{0}(\cdot)$ yields that
\begin{eqnarray}\label{ls.9}
\bar{\mathcal{G}}_{0}(w;B)\le \bar{\mathcal{F}}_{0}(w;B)-c_{0}\int_{B}\snr{Dw}^{p} \ \dx.
\end{eqnarray}
Next, we select parameters $\rr/2\le \tau_{1}< \tau_{2}\le \rr$, introduce the set
\begin{eqnarray*}
\mathcal{N}:=\left\{t\in (\tau_{1},\tau_{2})\colon t\mapsto \int_{B_{t}(x_{0})}\snr{Du}^{p} \ \dx \ \mbox{is nondifferentiable at} \ t\right\},
\end{eqnarray*} 
notice that $\mathcal{H}^{1}(\mathcal{N})=0$ cf. \cite[Section 2.3]{gk}, and fix numbers $\ti{\tau}_{1}<\ti{\tau}_{2}\in (\tau_{1},\tau_{2})\setminus \mathcal{N}$ as in \eqref{ex.2}. We then let $\eta\in C^{1}_{\rm c}(B_{\ti{\tau}_{2}}(x_{0}))$ be such that
\begin{eqnarray*}
\mathds{1}_{B_{\ti{\tau}_{1}}(x_{0})}\le \eta\le \mathds{1}_{B_{\ti{\tau}_{2}}(x_{0})}\qquad \mbox{and}\qquad \snr{D\eta}\lesssim \frac{1}{\ti{\tau}_{2}-\ti{\tau}_{1}},
\end{eqnarray*}
set $S(x_{0}):=B_{\tau_{2}}(x_{0})\setminus B_{\tau_{1}}(x_{0})$, $\ti{S}(x_{0}):=B_{\ti{\tau}_{2}}(x_{0})\setminus B_{\ti{\tau}_{2}}(x_{0})$, $\uu(x):=u(x)-\ell(x)$ and define maps 
\begin{eqnarray*} \varphi_{1}(x):=\mathfrak{T}_{\ti{\tau}_{1},\ti{\tau}_{2}}[(1-\eta)\uu](x),\qquad  \varphi_{2}(x):=\uu(x)-\varphi_{1}(x).
\end{eqnarray*}
By Lemma \ref{exlem} (\emph{ii.})-(\emph{iii.}) we have that
\begin{flalign}\label{0}
\varphi_{1}\equiv 0 \ \  \mbox{on} \ \  B_{\ti{\tau}_{1}}(x_{0}),\quad\varphi_{2}\in W^{1,p}_{0}(B_{\ti{\tau}_{2}}(x_{0}),\mathbb{R}^{N}),\quad\varphi_{2}\equiv \uu \ \ \mbox{on} \ \ B_{\ti{\tau}_{1}}(x_{0}), \quad D\uu=D\varphi_{1}+D\varphi_{2};
\end{flalign}
in particular, $\eqref{pq}_{2}$ assures that $p/(p+1-q)<np/(n-1)$ so by \eqref{ex.1} it is
\eqn{ls.10}
$$\varphi_{1}\in W^{1,\frac{p}{p+1-q}}(B_{\ti{\tau}_{2}}(x_{0}),\mathbb{R}^{N})\cap W^{1,q}(B_{\ti{\tau}_{2}}(x_{0}),\mathbb{R}^{N}).$$ 
Moreover, the fact that $\ti{\tau}_{2}\not \in \mathcal{N}$ renders that
\begin{eqnarray}\label{ls.11}
\limsup_{\sigma\to 0}\frac{1}{\sigma}\int_{B_{\ti{\tau}_{2}+\sigma}(x_{0})\setminus B_{\ti{\tau}_{2}-\sigma}(x_{0})}\snr{Du}^{p} \ \dx<\infty.
\end{eqnarray}
By construction, we see that \eqref{ls.11} holds for $\uu$ and, by \eqref{ex.0}$_{3,4}$ we also deduce that \eqref{ls.11} is verified also for $\varphi_{1}$, $\ell+\varphi_{2}$ and $u-\varphi_{1}$. Recalling that \eqref{ls.11} is sufficient for the additivity of functional $\bar{\mathcal{G}}_{0}(\cdot)$, cf. \eqref{ls.5}-\eqref{ls.6}, we can conclude that
\begin{flalign}
 \mbox{the additivity property} \ \eqref{ls.6} \ \mbox{holds for functions} \  \uu,\ \varphi_{1}, \ \ell+\varphi_{2}, \  u-\varphi_{1}.\label{ls.12}
\end{flalign}
We then estimate
\begin{eqnarray}\label{name}
0&\stackrel{\eqref{ls.4}}{\le}&\bar{\mathcal{G}}_{0}(\ell+\varphi_{2};B_{\rr}(x_{0}))-\bar{\mathcal{G}}_{0}(\ell;B_{\rr}(x_{0}))\nonumber \\
&\stackrel{\eqref{ls.12},\eqref{0}_{2}}{=}&\bar{\mathcal{G}}_{0}(\ell+\varphi_{2};B_{\ti{\tau}_{2}}(x_{0}))+\bar{\mathcal{G}}_{0}(\ell;B_{\rr}(x_{0})\setminus \bar{B}_{\ti{\tau}_{2}}(x_{0}))-\bar{\mathcal{G}}_{0}(\ell;B_{\rr}(x_{0}))\nonumber\\
&=&\bar{\mathcal{G}}_{0}(\ell+\varphi_{2};B_{\ti{\tau}_{2}}(x_{0}))+\mathcal{G}_{0}(\ell;B_{\rr}(x_{0})\setminus \bar{B}_{\ti{\tau}_{2}}(x_{0}))-\mathcal{G}_{0}(\ell;B_{\rr}(x_{0}))\nonumber \\
&\stackrel{\eqref{ls.9}, \eqref{0}_{2}}{\le}&\bar{\mathcal{F}}_{0}(\ell+\varphi_{2};B_{\ti{\tau}_{2}}(x_{0}))-\int_{B_{\ti{\tau}_{2}}(x_{0})}F(z_{0}) \ \dx\nonumber \\
&&-c_{0}\int_{B_{\ti{\tau}_{2}}(x_{0})}\left[\snr{z_{0}+D\varphi_{2}}^{p}-\snr{z_{0}}^{p} \right]\ \dx\nonumber \\
&\stackrel{\eqref{ls.7}}{=}&\bar{\mathcal{F}}(u-\varphi_{1};B_{\ti{\tau}_{2}}(x_{0}))-\int_{B_{\ti{\tau}_{2}}(x_{0})}F(z_{0}) \ \dx+\int_{B_{\ti{\tau}_{2}}(x_{0})}f\cdot(u-\varphi_{1}) \ \dx\nonumber \\
&&-c_{0}\int_{B_{\ti{\tau}_{2}}(x_{0})}\left[\snr{z_{0}+D\varphi_{2}}^{p}-\snr{z_{0}}^{p} \right]\ \dx,
\end{eqnarray}
which in turn yields that
\begin{eqnarray}\label{ls.13}
c\int_{B_{\ti{\tau}_{2}}(x_{0})}\snr{V_{\snr{z_{0}},p}(D\varphi_{2})}^{2} \ \dx&\stackrel{\eqref{Vm},\eqref{ls.2}}{\le}&c_{0}\int_{B_{\ti{\tau}_{2}}(x_{0})}\left[\snr{z_{0}+D\varphi_{2}}^{p}-\snr{z_{0}}^{p} \right]\ \dx\nonumber \\
&\stackrel{\eqref{name}}{\le}&\bar{\mathcal{F}}(u-\varphi_{1};B_{\ti{\tau}_{2}}(x_{0}))-\int_{B_{\ti{\tau}_{2}}(x_{0})}F(z_{0}) \ \dx\nonumber \\
&&+\int_{B_{\ti{\tau}_{2}}(x_{0})}f\cdot(u-\varphi_{1}) \ \dx
\end{eqnarray}
with $c\equiv c(n,N,p,\lambda,\Lambda)$. The minimality of $u$ together with \eqref{ls.11}, \eqref{ls.12} and $\eqref{0}_{2}$ imply that
\begin{eqnarray}\label{ls.15.1}
\bar{\mathcal{F}}(u;B_{\ti{\tau}_{2}}(x_{0}))\le \bar{\mathcal{F}}(u-\varphi_{2};B_{\ti{\tau}_{2}}(x_{0}))
\end{eqnarray}
and, by \eqref{qc} and the basic property \eqref{ls.16} of Lebesgue-Serrin-Marcellini extension, it is 
\begin{eqnarray}\label{ls.17}
\bar{\mathcal{F}}(u-\varphi_{2};B_{\ti{\tau}_{2}}(x_{0}))&\stackrel{\eqref{ls.7}}{=}&\bar{\mathcal{F}}_{0}(\ell+\varphi_{1};B_{\ti{\tau}_{2}}(x_{0}))-\int_{B_{\ti{\tau}_{2}}(x_{0})}f\cdot(u-\varphi_{2}) \ \dx\nonumber \\
&\stackrel{\eqref{ls.8},\eqref{ls.10}}{=}&\int_{B_{\ti{\tau}_{2}}(x_{0})}F(z_{0}+D\varphi_{1}) \ \dx-\int_{B_{\ti{\tau}_{2}}(x_{0})}f\cdot(u-\varphi_{2}) \ \dx.
\end{eqnarray}
Moreover, using again \eqref{qc} and \eqref{ls.16} we get
\begin{eqnarray}\label{ls.14}
\bar{\mathcal{F}}(u-\varphi_{1};B_{\ti{\tau}_{2}}(x_{0}))-\bar{\mathcal{F}}(u;B_{\ti{\tau}_{2}}(x_{0}))&\stackrel{\eqref{ls.7}}{=}&\bar{\mathcal{F}}_{0}(u-\varphi_{1};B_{\ti{\tau}_{2}}(x_{0}))-\bar{\mathcal{F}}_{0}(u;B_{\ti{\tau}_{2}}(x_{0}))\nonumber \\
&&+\int_{B_{\ti{\tau}_{2}}(x_{0})}f\cdot \varphi_{1} \ \dx\nonumber \\
&\stackrel{\eqref{ls.3},\eqref{ls.10}}{=}&\int_{B_{\ti{\tau}_{2}}(x_{0})}\left[QF(Du-D\varphi_{1})-QF(Du)\right] \ \dx\nonumber \\
&&+\int_{B_{\ti{\tau}_{2}}(x_{0})}f\cdot \varphi_{1} \ \dx\nonumber \\
&\stackrel{\eqref{ls.8},\eqref{0}_{1}}{=}&\int_{\ti{S}(x_{0})}\left[F(Du-D\varphi_{1})-F(Du)\right] \ \dx\nonumber \\
&&+\int_{B_{\ti{\tau}_{2}}(x_{0})}f\cdot \varphi_{1} \ \dx.
\end{eqnarray}
At this stage, by means of $\eqref{0}_{1}$, \eqref{ls.15.1}, \eqref{ls.17} and \eqref{ls.14} we complete the estimate in \eqref{ls.13} as
\begin{eqnarray*}
c\int_{B_{\ti{\tau}_{2}}(x_{0})}\snr{V_{\snr{z_{0}},p}(D\varphi_{2})}^{2} \ \dx&\le&\left[\bar{\mathcal{F}}(u-\varphi_{1};B_{\ti{\tau}_{2}}(x_{0}))-\bar{\mathcal{F}}(u;B_{\ti{\tau}_{2}}(x_{0}))\right]\nonumber \\
&&+\left[\bar{\mathcal{F}}(u;B_{\ti{\tau}_{2}}(x_{0}))-\bar{\mathcal{F}}(u-\varphi_{2};B_{\ti{\tau}_{2}}(x_{0}))\right]\nonumber \\
&&+\left[\bar{\mathcal{F}}(u-\varphi_{2};B_{\ti{\tau}_{2}}(x_{0}))-\int_{B_{\ti{\tau}_{2}}(x_{0})}F(z_{0}) \ \dx\right]\nonumber \\
&&+\int_{B_{\ti{\tau}_{2}}(x_{0})}f\cdot(u-\varphi_{1}) \ \dx\nonumber \\
&\le &\int_{\ti{S}(x_{0})}\left[F(Du-D\varphi_{1})-F(Du)\right] \ \dx\nonumber \\
&&+\int_{\ti{S}(x_{0})}\left[F(z_{0}+D\varphi_{1})-F(z_{0})\right] \ \dx\nonumber \\
&&+\int_{B_{\ti{\tau}_{2}}(x_{0})}f\cdot \varphi_{2} \ \dx=:\left[\mbox{(I)}+\mbox{(II)}+\mbox{(III)}\right],
\end{eqnarray*}
for $c\equiv c(n,N,p,\lambda,\Lambda)$. We then split
\begin{eqnarray*}
\mbox{(I)}+\mbox{(II)}&=&\int_{\ti{S}(x_{0})}\left\langle\left(\int_{0}^{1}\left[\partial F(z_{0})-\partial F(z_{0}+D\uu-s D\varphi_{1})\right] \ \ds\right), D\varphi_{1}\right\rangle \ \dx\nonumber \\
&&+\int_{\ti{S}(x_{0})}\left\langle\left(\int_{0}^{1}\left[\partial F(z_{0}+s D\varphi_{1})-\partial F(z_{0})\right] \ \ds\right), D\varphi_{1}\right\rangle\  \dx=:\mbox{(I')}+\mbox{(II')}.
\end{eqnarray*}
Via Young inequality we get
\begin{eqnarray*}
\snr{\mbox{(I')}}&=&\left| \ \int_{\ti{S}(x_{0})}\left[\int_{0}^{1}\left(\int_{0}^{1}\partial^{2}F(z_{0}+t(D\uu-s D\varphi_{1})) \ \dt\right)\langle (D\uu-s D\varphi_{1}),D\varphi_{1}\rangle \ \ds\right] \dx \ \right|\nonumber \\
&\stackrel{\eqref{assf}_{3}}{\le}&c\int_{\ti{S}(x_{0})}\int_{0}^{1}\left(\snr{z_{0}}^{2}+\snr{D\uu-s D\varphi_{1}}^{2}\right)^{\frac{p-2}{2}}\snr{D\uu-s D\varphi_{1}} \ \ds \snr{D\varphi_{1}} \dx\nonumber \\
&&+c\int_{\ti{S}(x_{0})}\int_{0}^{1}\left(\snr{z_{0}}^{2}+\snr{D\uu-s D\varphi_{1}}^{2}\right)^{\frac{q-2}{2}}\snr{D\uu-s D\varphi_{1}} \ \ds \snr{D\varphi_{1}} \dx\nonumber \\
&\le&c\int_{\ti{S}(x_{0})\cap\{\snr{D\varphi_{1}}>\snr{D\uu}\}}\snr{V_{\snr{z_{0}},p}(D\varphi_{1})}^{2} \dx+c\int_{\ti{S}(x_{0})\cap\{\snr{D\varphi_{1}}\le\snr{D\uu}\}}\snr{V_{\snr{z_{0}},p}(D\uu)}^{2} \dx\nonumber \\
&&+c\int_{\ti{S}(x_{0})\cap\{\snr{D\varphi_{1}}\le \snr{D\uu}\}}\left(\snr{z_{0}}^{2}+\snr{D\uu}^{2}\right)^{\frac{q-2}{2}}\snr{D\uu}\snr{D\varphi_{1}} \dx\nonumber \\
&&+c \int_{\ti{S}(x_{0})\cap\{\snr{D\varphi_{1}}>\snr{D\uu}\}}\snr{V_{\snr{z_{0}},q}(D\varphi_{1})}^{2} \dx\nonumber \\
&\le&c\int_{\ti{S}(x_{0})}\left[\snr{V_{\snr{z_{0}},p}(D\uu)}^{2}+\snr{V_{\snr{z_{0}},p}(D\varphi_{1})}^{2}\right] \dx\nonumber \\
&&+c\int_{\ti{S}(x_{0})}\left[\snr{V_{\snr{z_{0}},q}(D\varphi_{1})}^{2}+\left(\snr{z_{0}}^{2}+\snr{D\uu}^{2}\right)^{\frac{q-2}{2}}\snr{D\uu}\snr{D\varphi_{1}}\right] \dx,
\end{eqnarray*}
with $c\equiv c(n,N,\Lambda,p,q)$. Similarly, it is
\begin{eqnarray*}
\snr{\mbox{(II')}}&=&\left| \ \int_{B_{\ti{\tau}_{2}}(x_{0})}\int_{0}^{1}\left(\int_{0}^{1}\partial^{2}F(z_{0}+s t D\varphi_{1}) \ \dt\right)s D\varphi_{1}\cdot D\varphi_{1} \ \ds \dx\ \right|\nonumber \\
&\stackrel{\eqref{assf}_{3}, \eqref{Vm}}{\le}&\Lambda\int_{\ti{S}(x_{0})}\left[\snr{V_{\snr{z_{0}},p}(D\varphi_{1})}^{2}+\snr{V_{\snr{z_{0}},q}(D\varphi_{1})}^{2}\right] \dx,
\end{eqnarray*}
for $c\equiv c(n,N,\Lambda,p,q)$ and, trivially, we have
\begin{eqnarray*}
\snr{\mbox{(III)}}&\le&\int_{B_{\ti{\tau}_{2}}(x_{0})}\snr{f}\snr{\varphi_{2}} \ \dx.
\end{eqnarray*}
Collecting all the above estimates, we obtain
\begin{eqnarray*}
\int_{B_{\ti{\tau}_{2}}(x_{0})}\snr{V_{\snr{z_{0}},p}(D\varphi_{2})}^{2} \dx &\le&c\int_{\ti{S}(x_{0})}\snr{V_{\snr{z_{0}},p}(D\uu)}^{2} \dx +\int_{B_{\ti{\tau}_{2}}(x_{0})}\snr{f}\snr{\varphi_{2}} \dx\nonumber \\
&&+c\int_{\ti{S}(x_{0})}\left[\snr{V_{\snr{z_{0}},p}(D\varphi_{1})}^{2}+\snr{V_{\snr{z_{0}},q}(D\varphi_{1})}^{2}\right] \dx\nonumber \\
&&+c\int_{\ti{S}(x_{0})}(\snr{z_{0}}^{2}+\snr{D\uu}^{2})^{\frac{q-2}{2}}\snr{D\uu}\snr{D\varphi_{1}} \dx\nonumber \\
&=:&c\int_{\ti{S}(x_{0})}\snr{V_{\snr{z_{0}},p}(D\uu)}^{2} \dx +c\left(\mbox{T}_{1}+\mbox{T}_{2}+\mbox{T}_{3}\right),
\end{eqnarray*}
for $c\equiv c(n,N,\lambda,\Lambda,p,q)$. Recalling \eqref{f}-\eqref{f.0}, by H\"older, Young and Sobolev-Poincar'\'e inequalities we get
\begin{eqnarray*}
\mbox{T}_{1}
&\le&\snr{B_{\ti{\tau}_{2}}(x_{0})}\left(\ti{\tau}_{2}^{m}\mint_{B_{\ti{\tau}_{2}}(x_{0})}\snr{f}^{m} \dx\right)^{1/m}\left(\ti{\tau_{2}}^{-m'}\mint_{B_{\ti{\tau}_{2}}(x_{0})}\snr{\varphi_{2}}^{m'} \dx\right)^{\frac{1}{m'}}\nonumber \\
&\stackrel{\eqref{f}}{\le}&\snr{B_{\ti{\tau}_{2}}(x_{0})}\left(\ti{\tau}_{2}^{m}\mint_{B_{\ti{\tau}_{2}}(x_{0})}\snr{f}^{m} \dx\right)^{1/m}\left(\mint_{B_{\ti{\tau}_{2}}(x_{0})}\left| \ \frac{\varphi_{2}}{\ti{\tau}_{2}} \ \right|^{2^{*}} \dx\right)^{\frac{1}{2^{*}}}\nonumber \\
&\stackrel{\eqref{0}_{2}}{\le}&\snr{B_{\ti{\tau}_{2}}(x_{0})}\left(\ti{\tau}_{2}^{m}\mint_{B_{\ti{\tau}_{2}}(x_{0})}\snr{f}^{m} \dx\right)^{1/m}\left(\mint_{B_{\ti{\tau}_{2}}(x_{0})}\snr{D\varphi_{2}}^{2} \dx\right)^{\frac{1}{2}}\nonumber \\
&\le&\varepsilon\int_{B_{\ti{\tau}_{2}}(x_{0})}\snr{V_{\snr{z_{0}},p}(D\varphi_{2})}^{2} \dx+\frac{c\snr{B_{\rr}(x_{0})}}{\varepsilon\snr{z_{0}}^{p-2}}\left(\rr^{m}\mint_{B_{\rr}(x_{0})}\snr{f}^{m} \dx\right)^{\frac{2}{m}},
\end{eqnarray*}
where we also used that $\rr/2\le \ti{\tau}_{2}\le \rr$ and $c\equiv c(n,p,m)$. Before estimating terms $\mbox{T}_{2}$-$\mbox{T}_{3}$, let $w\in W^{1,t}(\ti{S}(x_{0}),\mathbb{R}^{N})$ for $t\in \{p,q\}$, define
\begin{eqnarray*}
&\ti{S}_{-}(x_{0};w):=\left\{x\in \ti{S}(x_{0})\colon \snr{z_{0}}>\snr{Dw}\right\},\qquad\ti{S}_{+}(x_{0};w):=\left\{x\in \ti{S}(x_{0})\colon \snr{z_{0}}\le\snr{Dw}\right\}
\end{eqnarray*}
and observe that as $t\in \{p,q\}$ it is
\begin{flalign}\label{1}
\begin{cases}
\ \mathds{1}_{\ti{S}_{-}(x_{0};w)}\snr{V_{\snr{z_{0}},t}(Dw)}^{2}\le 2^{2t}\snr{z_{0}}^{t-2}\snr{Dw}^{2}\\
\ \mathds{1}_{\ti{S}_{+}(x_{0};w)}(\snr{z_{0}}^{2}+\snr{Dw}^{2})^{\frac{t}{2}}\le 2^{2t}\snr{Dw}^{t}.
\end{cases}
\end{flalign}
Moreover, for $0<s_{1}<s_{2}\le1$, set
\begin{flalign*}
&\mf{V}(s_{2}-s_{1}):=\int_{B_{s_{2}}(x_{0})\setminus B_{s_{1}}(x_{0})}\left[\snr{V_{\snr{z_{0}},p}(D\uu)}^{2}+\left|V_{\snr{z_{0}},p}\left(\frac{\uu}{s_{2}-s_{1}}\right) \right|^{2}\right] \dx;\nonumber \\
&\mf{D}(s_{2}-s_{1}):=\int_{B_{s_{2}}(x_{0})\setminus B_{s_{1}}(x_{0})}\left[\snr{D\uu}^{p}+\left|\frac{\uu}{s_{2}-s_{1}} \right|^{p}\right] \dx.
\end{flalign*}
Now, concerning term $\mbox{T}_{2}$, we bound by means of \eqref{1}, \eqref{pq}, \eqref{ex.0}$_{2}$, \eqref{ex.1}$_{2}$, \eqref{ex.2} and \eqref{equiv.1},
\begin{eqnarray*}
\mbox{T}_{2}&\le &c(1+M^{q-p})\int_{\ti{S}_{-}(x_{0};\varphi_{1})}\snr{z_{0}}^{p-2}\snr{D\varphi_{1}}^{2}+c\int_{\ti{S}_{+}(x_{0};\varphi_{1})}\left[\snr{D\varphi_{1}}^{p}+\snr{D\varphi_{1}}^{q}\right] \dx\nonumber \\
&\le &c(1+M^{q-p})\snr{z_{0}}^{p-2}\int_{\ti{S}(x_{0})}\left[\snr{D\uu}^{2}+\left|\frac{\uu}{\ti{\tau}_{2}-\ti{\tau}_{1}}\right|^{2}\right] \dx\nonumber \\
&&+c\mf{D}(\ti{\tau}_{2}-\ti{\tau}_{1})+\frac{c\mathds{1}_{\{q>p\}}\mf{D}(\tau_{2}-\tau_{1})^{q/p}}{(\tau_{2}-\tau_{1})^{n\left(\frac{q}{p}-1\right)}}\nonumber \\
&\le&c(1+M^{q-p})\mf{V}(\ti{\tau}_{2}-\ti{\tau}_{1})+c\mf{D}(\ti{\tau}_{2}-\ti{\tau}_{1})+\frac{c\mathds{1}_{\{q>p\}}\mf{D}(\tau_{2}-\tau_{1})^{q/p}}{(\tau_{2}-\tau_{1})^{n\left(\frac{q}{p}-1\right)}},
\end{eqnarray*}
with $c\equiv c(n,N,p,q)$. Finally, using the definition of $V_{\snr{z_{0}},p}(\cdot)$ and H\"older inequality we estimate
\begin{eqnarray*}
\mbox{T}_{3}&\le&cM^{q-p}\int_{\ti{S}_{-}(x_{0};\uu)}\snr{V_{\snr{z_{0}},p}(D\uu)}\snr{z_{0}}^{\frac{p-2}{2}}\snr{D\varphi_{1}} \dx+c\int_{\ti{S}_{+}(x_{0};\uu)}\snr{D\uu}^{q-1}\snr{D\varphi_{1}} \dx\nonumber \\
&\stackrel{\eqref{pq}}{\le}&cM^{q-p}\left(\int_{\ti{S}(x_{0})}\snr{V_{\snr{z_{0}},p}(D\uu)}^{2} \dx\right)^{1/2}\left(\int_{\ti{S}(x_{0})}\snr{z_{0}}^{p-2}\snr{D\varphi_{1}}^{2} \dx\right)^{1/2}\nonumber \\
&&+c\left(\int_{\ti{S}(x_{0})}\snr{D\uu}^{p} \dx\right)^{\frac{q-1}{p}}\left(\int_{\ti{S}(x_{0})}\snr{D\varphi_{1}}^{\frac{p}{1-q+p}} \dx\right)^{\frac{1-q+p}{p}}\nonumber \\
&\stackrel{\eqref{ex.1}_{2},\eqref{ex.0}_{2}}{\le}&\frac{c\mathds{1}_{\{q>p\}}}{(\ti{\tau}_{2}-\ti{\tau}_{1})^{n\left(\frac{q}{p}-1\right)}}\left(\int_{S(x_{0})}\left[\snr{D\uu}^{p}+\left|\frac{\uu}{\ti{\tau}_{2}-\ti{\tau}_{1}}\right|^{p}\right] \dx\right)^{\frac{q}{p}}+cM^{q-p}\mf{V}(\ti{\tau}_{2}-\ti{\tau}_{1})\nonumber \\
&\stackrel{\eqref{ex.2}}{\le}&cM^{q-p}\mf{V}(\ti{\tau}_{2}-\ti{\tau}_{1})+\frac{c\mathds{1}_{\{q>p\}}\mf{D}(\tau_{2}-\tau_{1})^{q/p}}{(\tau_{2}-\tau_{1})^{n\left(\frac{q}{p}-1\right)}},
\end{eqnarray*}
for $c\equiv c(n,N,p,q)$. Merging the content of the above displays and choosing $\varepsilon>0$ sufficiently small to reabsorb terms, we obtain
\begin{eqnarray*}
\int_{B_{\tau_{1}}(x_{0})}\snr{V_{\snr{z_{0}},p}(D\uu)}^{2} \dx &\le&c(1+M^{q-p})\mf{V}(\ti{\tau}_{2}-\ti{\tau}_{1})+c\mf{D}(\ti{\tau}_{2}-\ti{\tau}_{1})\nonumber \\
&&+\frac{c\mathds{1}_{\{q>p\}}\mf{D}(\tau_{2}-\tau_{1})^{q/p}}{(\tau_{2}-\tau_{1})^{n\left(\frac{q}{p}-1\right)}}+\frac{c\snr{B_{\rr}(x_{0})}}{\varepsilon\snr{z_{0}}^{p-2}}\left(\rr^{m}\mint_{B_{\rr}(x_{0})}\snr{f}^{m} \dx\right)^{\frac{2}{m}}\nonumber \\
&\stackrel{\eqref{equiv.1},\eqref{ex.2}}{\le}&c\int_{S(x_{0})}\snr{V_{\snr{z_{0}},p}(D\uu)}^{2} \dx\nonumber \\
&&+c\int_{S(x_{0})}\left[\snr{z_{0}}^{p-2}\left| \frac{\uu}{\tau_{2}-\tau_{1}} \right|^{2}+\left|\frac{\uu}{\tau_{2}-\tau_{1}}\right|^{p}\right] \dx\nonumber \\
&&+\frac{c\mathds{1}_{\{q>p\}}\mf{D}(\tau_{2}-\tau_{1})^{q/p}}{(\tau_{2}-\tau_{1})^{n\left(\frac{q}{p}-1\right)}}+\frac{c\snr{B_{\rr}(x_{0})}}{\varepsilon\snr{z_{0}}^{p-2}}\left(\rr^{m}\mint_{B_{\rr}(x_{0})}\snr{f}^{m} \dx\right)^{\frac{2}{m}}
\end{eqnarray*}
with $c\equiv c(\textnormal{\texttt{data}},M^{q-p})$. Summing on both sides of the above inequality $c\int_{B_{\tau_{1}}(x_{0})}\snr{V_{\snr{z_{0}},p}(D\uu)}^{2} \dx$ we obtain
\begin{eqnarray}\label{ls.18}
\int_{B_{\tau_{1}}(x_{0})}\snr{V_{\snr{z_{0}},p}(D\uu)}^{2} \dx&\le&\frac{c}{1+c}\int_{B_{\tau_{2}}(x_{0})}\snr{V_{\snr{z_{0}},p}(D\uu)}^{2} \dx\nonumber \\
&&+c\int_{S(x_{0})}\left[\snr{z_{0}}^{p-2}\left| \frac{\uu}{\tau_{2}-\tau_{1}} \right|^{2}+\left|\frac{\uu}{\tau_{2}-\tau_{1}}\right|^{p}\right] \dx\nonumber \\
&&+\frac{c\mathds{1}_{\{q>p\}}\mf{D}(\tau_{2}-\tau_{1})^{q/p}}{(\tau_{2}-\tau_{1})^{n\left(\frac{q}{p}-1\right)}}+\frac{c\snr{B_{\rr}(x_{0})}}{\varepsilon\snr{z_{0}}^{p-2}}\left(\rr^{m}\mint_{B_{\rr}(x_{0})}\snr{f}^{m} \dx\right)^{\frac{2}{m}},
\end{eqnarray}
for $c\equiv c(\textnormal{\texttt{data}},M^{q-p})$, therefore Lemma \ref{l5} applies and renders \eqref{cacc}. The proof is complete.
\end{proof}
In the next lemma we run a linearization procedure that will allow to apply Lemma \ref{ahar}. For the ease of notation, we set $\ti{\mu}(t):=\mu(t)^{1-\frac{q-1}{p}}$.
\begin{lemma}\label{linlem}
Under assumptions \eqref{assf}, \eqref{pq}, \eqref{sqc} and \eqref{f}, let $B_{\rr}(x_{0})\Subset \Omega$ be a ball, $u\in W^{1,p}(\Omega,\mathbb{R}^{N})$ be a local minimizer of \eqref{exfun} and $z_{0}\in \left(\mathbb{R}^{N\times n}\setminus \{0\}\right)\cap \left\{\snr{z}\le 80000(M+1)\right\}$, for some constant $M\ge 0$, be any matrix such that $\mf{F}(u,z_{0};B_{\rr}(x_{0}))>0$. Then, it holds that
\begin{flalign}\label{4}
&\left| \ \mint_{B_{\rr}(x_{0})} \frac{\partial^{2}F(z_{0})}{\snr{z_{0}}^{p-2}}\left\langle \frac{\snr{z_{0}}^{\frac{p-2}{2}}(Du-z_{0})}{\mf{F}(u,z_{0};B_{\rr}(x_{0}))^{\frac{p}{2}}},D\varphi \right\rangle \dx \ \right|\nonumber \\
&\qquad\quad  \le \frac{c\nr{D\varphi}_{L^{\infty}(B_{\rr}(x_{0}))}\snr{z_{0}}^{\frac{2-p}{2}}}{\mf{F}(u,z_{0};B_{\rr}(x_{0}))^{\frac{p}{2}}}\left(\rr^{m}\mint_{B_{\rr}(x_{0})}\snr{f}^{m} \dx\right)^{1/m}\nonumber \\
&\qquad\quad\qquad  +c\nr{D\varphi}_{L^{\infty}(B_{\rr}(x_{0}))}\ti{\mu}\left(\frac{\mf{F}(u,z_{0};B_{\rr}(x_{0}))}{\snr{z_{0}}}\right)\left[1+\mf{H}\left(\frac{\mf{F}(u,z_{0};B_{\rr}(x_{0}))}{\snr{z_{0}}}\right)\right],
\end{flalign}
for all $\varphi\in C^{\infty}_{\rm c}(B_{\rr}(x_{0}),\mathbb{R}^{N})$, where $\mf{H}(t):=\left[t^{\frac{p-2}{2}}+t^{\frac{2q-p-2}{2}}\right]$ and $c\equiv c(p,q,M^{q-p})$.
\end{lemma}
\begin{proof}
Let $\varphi\in C^{\infty}_{\rm c}(B_{\rr}(x_{0}),\RN)$ be any map. For simplicity, we shall abbreviate $\nr{D\varphi}_{L^{\infty}(B_{\rr}(x_{0}))}\equiv \nr{D\varphi}_{\infty}$ and $\mf{F}(u,z_{0};B_{\rr}(x_{0}))\equiv \mf{F}_{0}(u)$. By \eqref{3} we see that
\begin{eqnarray*}
0&=&\mint_{B_{\rr}(x_{0})}\left[\langle\partial F(Du)-\partial F(z_{0}),D\varphi\rangle-f\cdot \varphi\right] \ \dx\nonumber \\
&=&\mint_{B_{\rr}(x_{0})}\left(\int_{0}^{1}\partial^{2}F(z_{0}+s(Du-z_{0})) \ \ds\right)\langle Du-z_{0},D\varphi\rangle-f\cdot \varphi \ \dx, 
\end{eqnarray*}
therefore we have
\begin{flalign*}
&\left| \ \mint_{B_{\rr}(x_{0})}\partial^{2} F(z_{0})\langle Du-z_{0}, D\varphi\rangle \dx \ \right|\stackrel{\eqref{3}}{\le}\mint_{B_{\rr}(x_{0})}\snr{f}\snr{\varphi} \dx\nonumber \\
&\qquad \qquad +\mint_{B_{\rr}(x_{0})}\left(\int_{0}^{1}\snr{\partial^{2}F(z_{0})-\partial^{2}F(z_{0}+s(Du-z_{0}))} \ \ds\right)\snr{Du-z_{0}}\snr{D\varphi} \dx\nonumber \\
&\qquad \qquad =:\mbox{(I)}+\mbox{(II)}.
\end{flalign*}
By the mean value theorem and H\"older inequality we have
\begin{eqnarray*}
\mbox{(I)}\le 4\nr{D\varphi}_{\infty}\left(\rr^{m}\mint_{B_{\rr}(x_{0})}\snr{f}^{m} \dx\right)^{1/m}.
\end{eqnarray*}
Moreover, by H\"older inequality, $\eqref{pq}$, the concavity of $t\mapsto \mu(t)$ and Jensen inequality it is
\begin{eqnarray*}
\mbox{(II)}&\stackrel{\eqref{assf}_{4}}{\le}&\nr{D\varphi}_{\infty}\mint_{B_{\rr}(x_{0})}\mu\left(\frac{\snr{Du-z_{0}}}{\snr{z_{0}}}\right) \ (\snr{z_{0}}^{2}+\snr{Du-z_{0}}^{2})^{\frac{p-2}{2}}\snr{Du-z_{0}} \dx\nonumber \\
&&+\nr{D\varphi}_{\infty}\mint_{B_{\rr}(x_{0})}\mu\left(\frac{\snr{Du-z_{0}}}{\snr{z_{0}}}\right) \ (\snr{z_{0}}^{2}+\snr{Du-z_{0}}^{2})^{\frac{q-2}{2}}\snr{Du-z_{0}} \dx\nonumber \\
&\stackrel{\eqref{equiv.1}}{\le}&c(1+M^{q-p})\nr{D\varphi}_{\infty}\snr{z_{0}}^{\frac{p-2}{2}}\left(\mint_{B_{\rr}(x_{0})}\mu\left(\frac{\snr{Du-z_{0}}}{\snr{z_{0}}}\right)^{2} \dx\right)^{\frac{1}{2}}\left(\mint_{B_{\rr}(x_{0})}\snr{z_{0}}^{p-2}\snr{Du-z_{0}}^{2} \dx\right)^{\frac{1}{2}}\nonumber \\
&&+c\nr{D\varphi}_{\infty}\left(\mint_{B_{\rr}(x_{0})}\mu\left(\frac{\snr{Du-z_{0}}}{\snr{z_{0}}}\right)^{p} \dx\right)^{\frac{1}{p}}\left(\mint_{B_{\rr}(x_{0})}\snr{Du-z_{0}}^{p} \dx\right)^{\frac{p-1}{p}}\nonumber \\
&&+c\nr{D\varphi}_{\infty}\left(\mint_{B_{\rr}(x_{0})}\mu\left(\frac{\snr{Du-z_{0}}}{\snr{z_{0}}}\right)^{\frac{p}{p-q+1}} \dx\right)^{\frac{p-q+1}{p}}\left(\mint_{B_{\rr}(x_{0})}\snr{Du-z_{0}}^{p} \dx\right)^{\frac{q-1}{p}}\nonumber \\
&\stackrel{\eqref{pq}}{\le}&c\nr{D\varphi}_{\infty}\ti{\mu}\left(\frac{\mf{F}_{0}(u)}{\snr{z_{0}}} \right)\left[(1+M^{q-p})\snr{z_{0}}^{\frac{p-2}{2}}\mf{F}_{0}(u)^{\frac{p}{2}}+\mf{F}_{0}(u)^{p-1}+\mf{F}_{0}(u)^{q-1}\right],
\end{eqnarray*}
with $c\equiv c(p,q)$. Combining the content of the three previous displays and dividing both sides of the resulting inequality by $\snr{z_{0}}^{\frac{p-2}{2}}\mf{F}_{0}(u)^{\frac{p}{2}}$ we obtain \eqref{4} and the proof is complete.
\end{proof}
In the next proposition, we show that under suitable smallness conditions, a local minimizer of \eqref{exfun} is approximately $\mathcal{A}$-harmonic in the sense of Section \ref{harsec} with $\mathcal{A}\equiv \snr{z_{0}}^{2-p}\partial^{2}F(z_{0})$ for a certain $z_{0}\in \mathbb{R}^{N\times n}\cap\left\{0<\snr{z}\le 40000(M+1)\right\}$, with $M\ge 0$. Assumptions \eqref{assf}$_{1,3}$ and inequality \eqref{sqc.1} assure that the bilinear form $\snr{z_{0}}^{2-p}\partial^{2}F(z_{0})$ satisfies \eqref{h.0} with $L\equiv L(\lambda,\Lambda,M^{q-p})$, see \cite[Chapter 5]{giu} and \cite{k}.
\begin{proposition}\label{p1}
Assume \eqref{assf}, \eqref{pq}, \eqref{sqc}, \eqref{f} and let $u\in W^{1,p}(\Omega,\mathbb{R}^{N})$ be a local minimizer of \eqref{exfun} verifying
\begin{eqnarray}\label{5.1}
\snr{(Du)_{B_{\rr}(x_{0})}}\le 40000(M+1)
\end{eqnarray}
for some constant $M\ge 0$ on a ball $B_{\rr}(x_{0})\Subset \Omega$. Then, there exists $\tau\equiv \tau(\textnormal{\texttt{data}},M^{q-p})\in (0,2^{-10})$, $\varepsilon_{0}\equiv \varepsilon_{0}(\textnormal{\texttt{data}},\mu(\cdot),M^{q-p})\in (0,1)$ and $\varepsilon_{1}\equiv \varepsilon_{1}(\textnormal{\texttt{data}},M^{q-p})\in (0,1)$ such that if the smallness conditions
\begin{eqnarray}\label{5}
\mf{F}(u;B_{\rr}(x_{0}))<\varepsilon_{0}\snr{(Du)_{B_{\rr}(x_{0})}}
\end{eqnarray}
and
\begin{eqnarray}\label{6}
\left(\rr^{m}\mint_{B_{\rr}(x_{0})}\snr{f}^{m} \dx\right)^{1/m}\le \varepsilon_{1}\snr{(Du)_{B_{\rr}(x_{0})}}^{\frac{p-2}{2}}\mf{F}(u;B_{\rr}(x_{0}))^{\frac{p}{2}}
\end{eqnarray}
are verified on $B_{\rr}(x_{0})$, it holds that
\begin{eqnarray}\label{30}
\
\mf{F}(u;B_{\tau\rr}(x_{0}))\le \tau^{\beta_{0}}\mf{F}(u;B_{\rr}(x_{0})),
\end{eqnarray}
for all $\beta_{0}\in (0,2/p)$, with $c_{0}\equiv c_{0}(\textnormal{\texttt{data}},M^{q-p})$.
\end{proposition}
\begin{proof}
For the sake of exposition, we shall adopt some abbreviations. Since all the balls considered here will be concentric to $B_{\rr}(x_{0})$ we will omit denoting the center. Moreover, given any ball $B_{\varsigma}(x_{0})$ with $0<\varsigma\le \rr$, we shall shorten $(Du)_{B_{\varsigma}(x_{0})}\equiv (Du)_{\varsigma}$ and for all $\varphi\in C^{\infty}_{\rm c}(B_{\rr},\mathbb{R}^{N})$ we denote $\nr{D\varphi}_{L^{\infty}(B_{\rr})}\equiv \nr{D\varphi}_{\infty}$. To avoid trivialities, we can assume that $\mf{F}(u,B_{\rr})>0$ and, being \eqref{5} in force, also that 
\begin{eqnarray}\label{7}
\snr{(Du)_{\rr}}>0.
\end{eqnarray}
We let $\varepsilon_{0}\in (0,1)$, define map
\begin{eqnarray*}
B_{\rr}\ni x\mapsto u_{0}(x):=\frac{\snr{(Du)_{\rr}}^{\frac{p-2}{2}}\left(u(x)-(u)_{\rr}-\langle(Du)_{\rr},x-x_{0}\rangle\right)}{\mf{F}(u,B_{\rr})^{\frac{p}{2}}}
\end{eqnarray*}
and set 
\begin{eqnarray}\label{18.1}
\sigma:=\left(\frac{\mf{F}(u;B_{\rr})}{\snr{(Du)_{\rr}}}\right)^{\frac{p}{2}}\qquad \mbox{and}\qquad \gamma:=\frac{\mf{F}(u,B_{\rr})^{\frac{p}{2}}}{\snr{(Du)_{\rr}}^{\frac{p-2}{2}}}.
\end{eqnarray}
As $\varepsilon_{0}\in (0,1]$, by $\eqref{5}$ it is $\sigma\in (0,1]$. The very definition of the excess functional $\mf{F}(\cdot)$ and a straightforward computation render that
\begin{eqnarray*}
\mint_{B_{\rr}}\snr{Du_{0}}^{2} \dx+\sigma^{p-2}\mint_{B_{\rr}}\snr{Du_{0}}^{p} \dx&=&\frac{1}{\mf{F}(u;B_{\rr})^{p}}\mint_{B_{\rr}}\snr{Du-(Du)_{B\rr}}^{p} \dx\nonumber \\
&&+\frac{1}{\mf{F}(u;B_{\rr})^{p}}\mint_{B_{\rr}}\snr{(Du)_{\rr}}^{p-2}\snr{Du-(Du)_{\rr}}^{2} \dx\le 1.
\end{eqnarray*}
By \eqref{5.1} and \eqref{7} we see that Lemma \ref{linlem} applies therefore setting
\begin{eqnarray}\label{10}
\mathcal{A}:=\partial^{2}F((Du)_{B_{\rr}})\snr{(Du)_{\rr}}^{2-p}
\end{eqnarray}
we have that
\begin{eqnarray}
\left| \ \mint_{B_{\rr}}\mathcal{A}\langle Du_{0},D\varphi\rangle \dx\ \right|&\stackrel{\eqref{4}}{\le}&\frac{c\nr{D\varphi}_{\infty}\snr{(Du)_{\rr}}^{\frac{2-p}{2}}}{\mf{F}(u;B_{\rr})^{\frac{p}{2}}}\left(\rr^{m}\mint_{B_{\rr}}\snr{f}^{m} \dx\right)^{1/m}\nonumber \\
&& +c\nr{D\varphi}_{\infty}\ti{\mu}\left(\frac{\mf{F}(u;B_{\rr})}{\snr{(Du)_{\rr}}}\right)\left(1+\mf{H}\left(\frac{\mf{F}(u;B_{\rr})}{\snr{(Du)_{B_{\rr}}}}\right)\right)\nonumber \\
& \stackrel{\eqref{5},\eqref{6}}{\le}&c\varepsilon_{1}\nr{D\varphi}_{\infty}+c\ti{\mu}(\varepsilon_{0})\left(1+\mf{H}(\varepsilon_{0})\right)\nr{D\varphi}_{\infty},\label{8}
\end{eqnarray}
with $c\equiv c(p,q,M^{q-p})$. Now let $\varepsilon\in (0,1)$ be a small number to be fixed in a few lines, $\delta\equiv \delta(\textnormal{\texttt{data}},M^{q-p},\varepsilon)\in (0,1)$ be the parameter given by Lemma \ref{ahar} and reduce the size of $\varepsilon_{0}$ and $\varepsilon_{1}$ in such a way that 
\begin{eqnarray}\label{9}
c\varepsilon_{1}+c\ti{\mu}(\varepsilon_{0})\left(1+\mf{H}(\varepsilon_{0})\right)\le \delta.
\end{eqnarray}
This choice establishes the dependencies $\varepsilon_{0}\equiv \varepsilon_{0}(\textnormal{\texttt{data}},M^{q-p},\mu(\cdot),\varepsilon)$ and $\varepsilon_{1}\equiv \varepsilon_{1}(\textnormal{\texttt{data}},M^{q-p},\varepsilon)$. Further restrictions on the size of $\varepsilon_{0}$, $\varepsilon_{1}$ will be imposed in a few lines and, once fixed $\varepsilon$, they will exhibit the dependency announced in the statement. Plugging \eqref{9} in \eqref{8}, we see that $u_{0}$ is approximately $\mathcal{A}$-harmonic in the sense of \eqref{ahar.0} with $\mathcal{A}$ defined as in \eqref{10}, therefore by Lemma \ref{ahar} there exists an $\mathcal{A}$-harmonic function $h_{0}\in W^{1,2}(B_{\rr}(x_{0}),\mathbb{R}^{N})$ satisfying
\begin{eqnarray}\label{11}
\mint_{B_{3\rr/4}}\snr{Dh_{0}}^{2} \dx+\sigma^{p-2}\mint_{B_{3\rr/4}}\snr{Dh_{0}}^{p} \dx\le 8^{2np}
\end{eqnarray}
and 
\begin{eqnarray}\label{12}
\mint_{B_{3\rr/4}}\left|\frac{u_{0}-h_{0}}{\rr}\right|^{2}+\sigma^{p-2}\left|\frac{u_{0}-h_{0}}{\rr}\right|^{p} \dx\le \varepsilon.
\end{eqnarray}
With $\tau\in (0,2^{-10})$ to be determined later on, by \eqref{12}, \eqref{h.1} and \eqref{11} we estimate
\begin{flalign}\label{13}
&\mint_{B_{2\tau\rr}}\left|\frac{u_{0}(x)-h_{0}(x_{0})-\langle Dh_{0}(x_{0}),x-x_{0}\rangle}{\tau\rr}\right|^{2} \dx\nonumber \\
&\qquad \qquad \le c\mint_{B_{2\tau\rr}}\left|\frac{h_{0}(x)-h_{0}(x_{0})-\langle Dh_{0}(x_{0}),x-x_{0}\rangle}{\tau\rr}\right|^{2} \dx+c\mint_{B_{2\tau\rr}}\left|\frac{u_{0}-h_{0}}{\tau\rr}\right|^{2} \dx\nonumber \\
&\qquad \qquad \le c(\tau\rr)^{2}\nr{D^{2}h_{0}}_{L^{\infty}(B_{\rr/2})}^{2}+\frac{c\varepsilon}{\tau^{n+2}}\nonumber \\
&\qquad \qquad\le c\tau^{2}\mint_{B_{3\rr/4}}\snr{Dh_{0}}^{2} \dx+\frac{c\varepsilon}{\tau^{n+2}} \le c\tau^{2}+\frac{c\varepsilon}{\tau^{n+2}},
\end{flalign}
with $c\equiv c(\textnormal{\texttt{data}},M^{q-p})$. We then choose $\varepsilon=\tau^{n+2p}$, so now it is $\varepsilon_{0}\equiv \varepsilon_{0}(\textnormal{\texttt{data}},M^{q-p},\mu(\cdot),\tau)$ and $\varepsilon_{1}\equiv \varepsilon_{1}(\textnormal{\texttt{data}},M^{q-p},\tau)$. The very definition of $u_{0}$, the choice of $\varepsilon$ and \eqref{13} yield that
\begin{flalign}\label{14}
&\mint_{B_{2\tau\rr}}\left|\frac{u-(u)_{\rr}-\langle(Du)_{\rr},x-x_{0}\rangle-\gamma\left(h_{0}(x_{0})-\langle Dh_{0}(x_{0}),x-x_{0}\rangle\right)}{\tau\rr}\right|^{2} \dx 
\le c\gamma^{2}\tau^{2},
\end{flalign}
for $c\equiv c(\textnormal{\texttt{data}},M^{q-p})$. In a similar fashion we have
\begin{flalign}\label{15}
&\sigma^{p-2}\mint_{B_{2\tau\rr}}\left|\frac{u_{0}-h_{0}(x_{0})-\langle Dh_{0}(x_{0}),x-x_{0}\rangle}{\tau\rr}\right|^{p} \dx\nonumber \\
&\qquad \quad \le c\sigma^{p-2}\mint_{B_{2\tau\rr}}\left|\frac{h_{0}(x)-h_{0}(x_{0})-\langle Dh_{0}(x_{0}),x-x_{0}\rangle}{\tau\rr}\right|^{p} \dx+c\sigma^{p-2}\mint_{B_{2\tau\rr}}\left|\frac{u_{0}-h_{0}}{\tau\rr}\right|^{p} \dx\nonumber \\
&\qquad \quad \le c\sigma^{p-2}(\tau\rr)^{p}\nr{D^{2}h_{0}}_{L^{\infty}(B_{\rr/2})}^{p}+\frac{c\varepsilon}{\tau^{n+p}}\le c\tau^{p}
\end{flalign}
where we also used \eqref{h.1}, \eqref{11} and it is $c\equiv c(\textnormal{\texttt{data}},M^{q-p})$, so scaling back to $u$ in \eqref{15} we obtain
\begin{flalign}
&\mint_{B_{2\tau\rr}}\left|\frac{u-(u)_{\rr}-\langle(Du)_{\rr},x-x_{0}\rangle-\gamma\left(h_{0}(x_{0})-\langle Dh_{0}(x_{0}),x-x_{0}\rangle\right)}{\tau\rr}\right|^{p} \dx\le c\tau^{p}\sigma^{2-p}\gamma^{p}\label{16},
\end{flalign}
with $c\equiv c(\textnormal{\texttt{data}},M^{q-p})$. Denote by $\ell_{2\tau\rr}$ the unique affine function such that
\begin{eqnarray*}
\ell_{2\tau\rr}\mapsto \min_{\ell \ \textnormal{affine}}\mint_{B_{2\tau\rr}}\snr{u-\ell}^{2} \dx.
\end{eqnarray*}
By \eqref{ll.0.1}, \eqref{14} and \eqref{16} we have that
\begin{flalign}\label{17}
\mint_{B_{2\tau\rr}}\snr{(Du)_{\rr}}^{p-2}\left|\frac{u-\ell_{2\tau\rr}}{2\tau\rr}\right|^{2}+\left|\frac{u-\ell_{2\tau\rr}}{2\tau\rr}\right|^{p} \dx\le c\tau^{2}\mf{F}(u;B_{\rr})^{p},
\end{flalign}
where $c\equiv c(\textnormal{\texttt{data}},M^{q-p})$ and we also expanded the expression of $\sigma$ and $\gamma$, cf. \eqref{18.1}. The definition \eqref{excess} of the excess functional $\mf{F}(\cdot)$ also renders that
\begin{eqnarray}\label{24}
\snr{D\ell_{2\tau\rr}-(Du)_{\rr}}&\le&\snr{D\ell_{2\tau\rr}-(Du)_{2\tau\rr}}+\snr{(Du)_{2\tau\rr}-(Du)_{\rr}}\nonumber \\
&\stackrel{\eqref{ll.1}}{\le}&c\left(\mint_{B_{2\tau\rr}}\snr{Du-(Du)_{2\tau\rr}}^{2} \dx\right)^{\frac{1}{2}}+\left(\mint_{B_{2\tau\rr}}\snr{Du-(Du)_{\rr}}^{2} \dx\right)^{\frac{1}{2}}\nonumber \\
&\stackrel{\eqref{minav}}{\le}&\frac{c}{\tau^{n/2}}\left(\mint_{B_{\rr}}\snr{Du-(Du)_{\rr}}^{2} \dx\right)^{\frac{1}{2}}\nonumber \\
&=&\frac{c\snr{(Du)_{\rr}}^{\frac{2-p}{2}}}{\tau^{n/2}}\left(\mint_{B_{\rr}}\snr{(Du)_{\rr}}^{p-2}\snr{Du-(Du)_{\rr}}^{2} \dx\right)^{\frac{1}{2}}\nonumber \\
&\le&\frac{c}{\tau^{n/2}}\left(\frac{\mf{F}(u,B_{\rr})}{\snr{(Du)_{\rr}}}\right)^{\frac{p}{2}}\snr{(Du)_{\rr}},
\end{eqnarray}
for $c\equiv c(n)$, so if the size of $\varepsilon_{0}$ is further reduced in such a way that
\begin{eqnarray*}
\left(\frac{\mf{F}(u,B_{\rr})}{\snr{(Du)_{\rr}}}\right)^{\frac{p}{2}}\stackrel{\eqref{5}}{\le} \varepsilon_{0}^{\frac{p}{2}}\le \frac{\tau^{n/2}}{c2^{2n+4}},
\end{eqnarray*}
where $c\equiv c(n)$ is the same constant appearing in \eqref{24}, we then have 
\begin{eqnarray}\label{18}
\snr{D\ell_{2\tau\rr}-(Du)_{\rr}}\le \frac{\snr{(Du)_{\rr}}}{8}.
\end{eqnarray}
Merging \eqref{18} and \eqref{17} and using triangular inequality we obtain
\begin{flalign}\label{19}
\mint_{B_{2\tau\rr}}\snr{D\ell_{2\tau\rr}}^{p-2}\left|\frac{u-\ell_{2\tau\rr}}{2\tau\rr}\right|^{2}+\left|\frac{u-\ell_{2\tau\rr}}{2\tau\rr}\right|^{p} \dx\le c\tau^{2}\mf{F}(u;B_{\rr})^{p},
\end{flalign}
for $c\equiv c(\textnormal{\texttt{data}},M^{q-p})$. Furthermore, by triangular inequality we also get
\begin{eqnarray}\label{20}
\snr{D\ell_{2\tau\rr}}\ge \snr{(Du)_{\rr}}-\snr{D\ell_{2\tau\rr}-(Du)_{\rr}}\stackrel{\eqref{18}}{\ge}\frac{7\snr{(Du)_{\rr}}}{8}
\end{eqnarray}
and as a consequence it is 
\begin{eqnarray}\label{22}
\snr{D\ell_{2\tau\rr}}>0\qquad \mbox{and}\qquad \frac{7}{8}\le \frac{\snr{D\ell_{2\tau\rr}}}{\snr{(Du)_{B_{\rr}}}}\le \frac{9}{8}.
\end{eqnarray}
We then bound via \eqref{cacc}, \eqref{19} and \eqref{22},
\begin{flalign}\label{21}
&\mint_{B_{\tau\rr}}\snr{D\ell_{2\tau\rr}}^{p-2}\snr{Du-D\ell_{2\tau\rr}}^{2} \dx+\inf_{z\in \mathbb{R}^{N\times n}}\mint_{B_{\tau\rr}}\snr{Du-z}^{p} \dx\nonumber \\
&\qquad \quad \le c\mf{K}\left(\mint_{B_{2\tau\rr}}\snr{D\ell_{2\tau\rr}}^{p-2}\left|\frac{u-\ell_{2\tau\rr}}{2\tau\rr}\right|^{2}+\left|\frac{u-\ell_{2\tau\rr}}{2\tau\rr}\right|^{p} \dx\right)\nonumber \\
&\qquad \quad \qquad  +\frac{c}{\snr{D\ell_{2\tau\rr}}^{p-2}}\left((2\tau\rr)^{m}\mint_{B_{2\tau\rr}}\snr{f}^{m} \dx\right)^{\frac{2}{m}}+c\mathds{1}_{\{q>p\}}\mf{F}(u,D\ell_{2\tau\rr};B_{2\tau\rr})^{q}\nonumber \\
&\qquad \quad \le c\mf{K}\left(\tau^{2}\mf{F}(u,B_{\rr})^{p}\right)+\frac{c\tau^{2-2n/m}}{\snr{(Du)_{\rr}}^{p-2}}\left(\rr^{m}\mint_{B_{\rr}}\snr{f}^{m} \dx\right)^{\frac{2}{m}}+c\mathds{1}_{\{q>p\}}\mf{F}(u,D\ell_{2\tau\rr};B_{2\tau\rr})^{q},
\end{flalign}
with $c\equiv c(\textnormal{\texttt{data}},M^{q-p})$. With \eqref{ik}, \eqref{ik.1} and \eqref{7} in mind we keep estimating
\begin{eqnarray}\label{26}
\mf{K}\left(\tau^{2}\mf{F}(u;B_{\rr})^{p}\right)&\le&\tau^{2}\mf{F}(u;B_{\rr})^{p}+\tau^{\frac{2q}{p}}\snr{(Du)_{B_{\rr}}}^{q-p}\left(\frac{\mf{F}(u;B_{\rr})}{\snr{(Du)_{B_{\rr}}}}\right)^{q-p}\mf{F}(u;B_{\rr})^{p}\nonumber \\
&\stackrel{\eqref{5}}{\le}&(1+M^{q-p}) \left(\tau^{2}+\tau^{\frac{2q}{p}}\varepsilon_{0}^{q-p}\right)\mf{F}(u;B_{\rr})^{p},
\end{eqnarray}
for $c\equiv c(p,q,M)$ and, by triangular inequality, 
\begin{eqnarray}\label{25}
\mf{F}(u,D\ell_{2\tau\rr};B_{2\tau\rr})^{p}&\stackrel{\eqref{22}_{2},\eqref{7}}{\le}&c\mint_{B_{2\tau\rr}}\snr{(Du)_{\rr}}^{p-2}\snr{Du-(Du)_{2\tau\rr}}^{2} \ \dx\nonumber \\
&&+c\mint_{B_{2\tau\rr}}\snr{Du-(Du)_{2\tau\rr}}^{p} \ \dx\nonumber \\
&&+c\snr{(Du)_{\rr}}^{p-2}\snr{(Du)_{2\tau\rr}-D\ell_{2\tau\rr}}^{2}+c\snr{D\ell_{2\tau\rr}-(Du)_{2\tau\rr}}^{p}\nonumber \\
&\stackrel{\eqref{ll.1},\eqref{minav}}{\le}&\frac{c}{\tau^{n}}\mf{F}(u;B_{\rr})^{p}+c\snr{(Du)_{\rr}}^{p-2}\mint_{B_{2\tau\rr}}\snr{Du-(Du)_{2\tau\rr}}^{2} \ \dx\nonumber \\
&&+c\mint_{B_{2\tau\rr}}\snr{Du-(Du)_{2\tau\rr}}^{p} \ \dx\stackrel{\eqref{minav}}{\le}\frac{c}{\tau^{n}}\mf{F}(u;B_{\rr})^{p},
\end{eqnarray}
with $c\equiv c(n,p)$, therefore
\begin{eqnarray}\label{27}
\mf{F}(u,D\ell_{2\tau\rr};B_{2\tau\rr})^{q}&\stackrel{\eqref{5},\eqref{25}}{\le}&\frac{cM^{q-p}\varepsilon_{0}^{q-p}}{\tau^{nq/p}}\mf{F}(u;B_{\rr})^{p}
\end{eqnarray}
for $c\equiv c(n,p)$. Combining \eqref{21}, \eqref{26} and \eqref{27} we obtain
\begin{flalign}\label{28}
&\mint_{B_{\tau\rr}}\snr{D\ell_{2\tau\rr}}^{p-2}\snr{Du-D\ell_{2\tau\rr}}^{2} \dx+\inf_{z\in \mathbb{R}^{N\times n}}\mint_{B_{\tau\rr}}\snr{Du-z}^{p} \dx\nonumber \\
&\qquad \qquad \le c\left(\tau^{2}+\tau^{2q/p}\varepsilon_{0}^{q-p}+\mathds{1}_{\{q>p\}}\tau^{-nq/p}\varepsilon_{0}^{q-p}\right)\mf{F}(u;B_{\rr})^{p}+\frac{c\tau^{2-2n/m}}{\snr{(Du)_{\rr}}^{p-2}}\left(\rr^{m}\mint_{B_{\rr}}\snr{f}^{m} \dx\right)^{\frac{2}{m}},
\end{flalign}
with $c\equiv c(\textnormal{\texttt{data}},M^{q-p})$. Furthermore, we have
\begin{eqnarray*}
\mint_{B_{\tau\rr}}\snr{(Du)_{\tau\rr}}^{p-2}\snr{Du-(Du)_{\tau\rr}}^{2} \dx&\le& 2^{4p}\mint_{B_{\tau\rr}}\snr{D\ell_{\tau\rr}-(Du)_{\tau\rr}}^{p-2}\snr{Du-(Du)_{\tau\rr}}^{2} \dx\nonumber \\
&&+2^{4p}\mint_{B_{\tau\rr}}\snr{D\ell_{2\tau\rr}-D\ell_{\tau\rr}}^{p-2}\snr{Du-(Du)_{\tau\rr}}^{2} \dx\nonumber\\
&&+2^{4p}\mint_{B_{\tau\rr}}\snr{D\ell_{2\tau\rr}}^{p-2}\snr{Du-(Du)_{\tau\rr}}^{2} \dx\nonumber \\
&=&\mbox{(I)}+\mbox{(II)}+\mbox{(III)}.
\end{eqnarray*}
By Young and triangular inequalities we get
\begin{eqnarray*}
\mbox{(I)}&\le&c\snr{D\ell_{\tau\rr}-(Du)_{\tau\rr}}^{p}+c\mint_{B_{\tau\rr}}\snr{Du-(Du)_{\tau\rr}}^{p} \dx\nonumber \\
&\stackrel{\eqref{ll.1}}{\le}&c\mint_{B_{\tau\rr}}\snr{Du-(Du)_{\tau\rr}}^{p} \dx\stackrel{\eqref{minav}}{\le}c\inf_{z\in \mathbb{R}^{N\times n}}\mint_{B_{\tau\rr}}\snr{Du-z}^{p} \dx\nonumber \\
&\stackrel{\eqref{28}}{\le}&c\left(\tau^{2}+\tau^{2q/p}\varepsilon_{0}^{q-p}+\mathds{1}_{\{q>p\}}\tau^{-nq/p}\varepsilon_{0}^{q-p}\right)\mf{F}(u;B_{\rr})^{p} \dx\nonumber \\
&&+\frac{c\tau^{2-2n/m}}{\snr{(Du)_{\rr}}^{p-2}}\left(\rr^{m}\mint_{B_{\rr}}\snr{f}^{m} \dx\right)^{\frac{2}{m}},
\end{eqnarray*}
with $c\equiv c(\textnormal{\texttt{data}},M^{q-p})$ and, via Jensen inequality in a similar way we also obtain,
\begin{eqnarray*}
\mbox{(II)}+\mbox{(III)}&\le &c\snr{D\ell_{2\tau\rr}-D\ell_{\tau\rr}}^{p}+c\mint_{B_{\tau\rr}}\snr{Du-(Du)_{\tau\rr}}^{p} \dx\nonumber \\
&&+c\mint_{B_{\tau\rr}}\snr{D\ell_{2\tau\rr}}^{p-2}\snr{Du-D\ell_{2\tau\rr}}^{2} \dx+c\snr{D\ell_{2\tau\rr}}^{p-2}\snr{(Du)_{\tau\rr}-D\ell_{2\tau\rr}}^{2}\nonumber\\
&\stackrel{\eqref{minav},\eqref{ll.0}}{\le}&c\mint_{B_{2\tau\rr}}\left|\frac{u-\ell_{2\tau\rr}}{2\tau\rr}\right|^{p} \dx+c\inf_{z\in \mathbb{R}^{N\times n}}\mint_{B_{\tau\rr}}\snr{Du-z}^{p} \dx\nonumber \\
&&+c\mint_{B_{\tau\rr}}\snr{D\ell_{2\tau\rr}}^{p-2}\snr{Du-D\ell_{2\tau\rr}}^{2} \dx\nonumber \\
&\stackrel{\eqref{19},\eqref{28}}{\le}&c\left(\tau^{2}+\tau^{2q/p}\varepsilon_{0}^{q-p}+\mathds{1}_{\{q>p\}}\tau^{-nq/p}\varepsilon_{0}^{q-p}\right)\mf{F}(u;B_{\rr})^{p}\nonumber \\
&&+\frac{c\tau^{2-2n/m}}{\snr{(Du)_{\rr}}^{p-2}}\left(\rr^{m}\mint_{B_{\rr}}\snr{f}^{m} \dx\right)^{\frac{2}{m}},
\end{eqnarray*}
for $c\equiv c(\textnormal{\texttt{data}},M^{q-p})$. 
Keeping \eqref{minav} in mind, we can merge the previous three displays and use \eqref{6} to get
\begin{flalign}\label{29}
\mf{F}(u;B_{\tau\rr})\le c \left(\tau^{2/p}+\tau^{2q/p^{2}}\varepsilon_{0}^{(q-p)/p}+\mathds{1}_{\{q>p\}}\tau^{-nq/p^{2}}\varepsilon_{0}^{(q-p)/p}+\tau^{2/p-2n/mp}\varepsilon_{1}^{2/p}\right)\mf{F}(u;B_{\rr}),
\end{flalign}
with $c\equiv c(\textnormal{\texttt{data}},M^{q-p})$. In \eqref{29} we reduce further the size of $\varepsilon_{0}$, $\varepsilon_{1}$ in such a way that
\begin{eqnarray}\label{ls.36}
\frac{2^{4np}\varepsilon_{0}^{p/2}}{\tau^{np}}+\mathds{1}_{\{q>p\}}\frac{\varepsilon_{0}^{(q-p)/p}}{\tau^{nq}}\le \tau^{2/p}\qquad \mbox{and}\qquad \tau^{-2n/mp}\varepsilon_{1}^{2/p}\le 1,
\end{eqnarray}
thus getting
\begin{eqnarray*}
\mf{F}(u;B_{\tau\rr}(x_{0}))\le c\tau^{2/p}\mf{F}(u;B_{\rr}(x_{0})),
\end{eqnarray*}
which holds true for all $\tau\in (0,2^{-10})$ with $c\equiv c(\textnormal{\texttt{data}},M^{q-p})$. Fixing any $\beta_{0}\in (0,2/p)$ and choosing $\tau$ so small that
\eqn{ls.36.1.1}
$$
\tau^{1-\beta_{0}}c<2^{-10}\qquad \mbox{and}\qquad \tau^{\beta_{0}}\le 2^{-8}
$$
we obtain \eqref{30}. Let us remark that since $\beta_{0}$ strictly variates between $0$ and $2/p$, we incorporated the dependency on $\beta_{0}$ of the various parameters appearing in \eqref{30} into a dependency on $p$. The proof is complete.
\end{proof}
Let us take care of the case complementary to \eqref{6}.
\begin{proposition}\label{p2}
Assume \eqref{assf}, \eqref{pq}, \eqref{sqc}, \eqref{f}. Let $u\in W^{1,p}(\Omega,\mathbb{R}^{N})$ be a local minimizer of \eqref{exfun} verifying \eqref{5.1} on some ball $B_{\rr}(x_{0})\Subset \Omega$ and $\tau\equiv \tau(\textnormal{\texttt{data}},M^{q-p})$, $\varepsilon_{0}\equiv \varepsilon_{0}(\textnormal{\texttt{data}},\mu(\cdot),M^{q-p})$, $\varepsilon_{1}\equiv \varepsilon_{1}(\textnormal{\texttt{data}},M^{q-p})$ be the parameters appearing in Proposition \ref{p1}. If the smallness conditions \eqref{5} holds and \begin{eqnarray}\label{6.1}
\left(\rr^{m}\mint_{B_{\rr}(x_{0})}\snr{f}^{m} \dx\right)^{1/m}> \varepsilon_{1}\snr{(Du)_{B_{\rr}(x_{0})}}^{\frac{p-2}{2}}\mf{F}(u;B_{\rr}(x_{0}))^{\frac{p}{2}}
\end{eqnarray}
is satisfied on $B_{\rr}(x_{0})$, then
\begin{eqnarray}\label{31}
\mf{F}(u;B_{\tau\rr}(x_{0}))\le c_{0}\left(\rr^{m}\mint_{B_{\rr}(x_{0})}\snr{f}^{m} \dx\right)^{\frac{1}{m(p-1)}},
\end{eqnarray}
with $c_{0}:=8^{2p}\tau^{-2n}\left(\varepsilon_{0}^{(p-2)/2}\varepsilon_{1}^{-1}\right)^{1/(p-1)}$.
\end{proposition}
\begin{proof}
The proof consists in a straightforward manipulation of \eqref{6.1}. In fact we have
\begin{eqnarray}\label{??.?}
\mf{F}(u;B_{\tau\rr}(x_{0}))^{\frac{p}{2}}&\stackrel{\eqref{tri.1}}{\le}&\frac{8^{p}}{\tau^{n/2}}\mf{F}(u;B_{\rr}(x_{0}))^{\frac{p}{2}}\nonumber \\
&\stackrel{\eqref{6.1}}{\le}&\frac{8^{p}}{\varepsilon_{1}\tau^{n/2}\snr{(Du)_{B_{\rr}(x_{0})}}^{\frac{p-2}{2}}}\left(\rr^{m}\mint_{B_{\rr}(x_{0})}\snr{f}^{m} \dx\right)^{1/m}.
\end{eqnarray}
Multiplying both sides of \eqref{??.?} by $\mf{F}(u;B_{\tau\rr}(x_{0}))^{\frac{p-2}{2}}$ and using \eqref{tri.1} and \eqref{5} we obtain \eqref{31}.
\end{proof}

\section{The degenerate scenario}\label{deg.1}
Next, we consider the case in which the smallness condition \eqref{5} fails. We start by proving an alternative version of Lemma \ref{ndegc}.
\begin{lemma}\label{cdeg}
Assume \eqref{assf}$_{1,2,3}$, \eqref{pq}, \eqref{sqc} and \eqref{f}, let $B_{\rr}(x_{0})\Subset \Omega$ be any ball with radius $\rr\in (0,1]$, $u\in W^{1,p}(\Omega,\mathbb{R}^{N})$ be a local minimizer of \eqref{exfun} and $\ell(x):=v_{0}+\langle z_{0},x-x_{0}\rangle$ be any affine function with $v_{0}\in \mathbb{R}^{N}$ and $z_{0}\in \mathbb{R}^{N\times n}$. Then it holds that
\begin{eqnarray}\label{caccdeg}
\mf{F}(u,z_{0};B_{\rr/2}(x_{0}))^{p}&\le&c\mathds{1}_{\{q>p\}}\snr{z_{0}}^{q-p}\mf{F}(u,z_{0};B_{\rr}(x_{0}))^{p}\nonumber \\
&&+c(1+\snr{z_{0}}^{q-p})\mint_{B_{\rr}(x_{0})}\snr{z_{0}}^{p-2}\left|\frac{u-\ell}{\rr} \right|^{2}+\left|\frac{u-\ell}{\rr} \right|^{p} \dx\nonumber \\
&&+c\left(\mint_{B_{\rr}(x_{0})}\snr{z_{0}}^{p-2}\left|\frac{u-\ell}{\rr} \right|^{2}+\left|\frac{u-\ell}{\rr} \right|^{p} \dx\right)^{\frac{q}{p}}\nonumber \\
&&+c\mathds{1}_{\{q>p\}}\mf{F}(u,z_{0};B_{\rr}(x_{0}))^{q}+c\left(\rr^{m}\mint_{B_{\rr}(x_{0})}\snr{f}^{m} \dx\right)^{\frac{p}{m(p-1)}},
\end{eqnarray}
where $c\equiv c(\textnormal{\texttt{data}})$ and $\mathds{1}_{\{q>p\}}$ has been defined in \eqref{ik.1}.
\end{lemma}
\begin{proof}
The proof is the same as the one of Lemma \ref{ndegc} till the estimate of terms $\mbox{T}_{1}$-$\mbox{T}_{3}$, that need to be treated in a slightly different way, owing to the fact that this time $z_{0}\in \mathbb{R}^{N\times n}$ is allowed to vanish. By H\"older, Young and Sobolev Poincar\'e inequalities we have
\begin{eqnarray*}
\mbox{T}_{1}&\le&\snr{B_{\ti{\tau}_{2}}(x_{0})}\left(\ti{\tau}_{2}^{m}\mint_{B_{\ti{\tau}_{2}}(x_{0})}\snr{f}^{m} \dx\right)^{1/m}\left(\mint_{B_{\ti{\tau}_{2}}(x_{0})}\left|\frac{\varphi_{2}}{\ti{\tau}_{2}}\right|^{m'} \dx\right)^{\frac{1}{m'}}\nonumber \\
&\stackrel{\eqref{f.0},\eqref{0}_{2}}{\le}&\snr{B_{\ti{\tau}_{2}}(x_{0})}\left(\ti{\tau}_{2}^{m}\mint_{B_{\ti{\tau}_{2}}(x_{0})}\snr{f}^{m} \dx\right)^{1/m}\left(\mint_{B_{\ti{\tau}_{2}}(x_{0})}\snr{D\varphi_{2}}^{p} \dx\right)^{\frac{1}{p}}\nonumber \\
&\le&\varepsilon\int_{B_{\ti{\tau}_{2}}(x_{0})}\snr{V_{\snr{z_{0}},p}(D\varphi_{2})}^{2} \dx+\frac{c\snr{B_{\rr}(x_{0})}}{\varepsilon^{1/(p-1)}}\left(\rr^{m}\mint_{B_{\rr}(x_{0})}\snr{f}^{m} \dx\right)^{\frac{p}{m(p-1)}},
\end{eqnarray*}
with $c\equiv c(n,p,m)$. Proceeding exactly as in Lemma \ref{ndegc}, but this time maintaining explicit the dependency on $\snr{z_{0}}$, we bound
\begin{eqnarray*}
\mbox{T}_{2}
&\le&c(1+\snr{z_{0}}^{q-p})\int_{\ti{S}(x_{0})}\left[\snr{V_{\snr{z_{0}},p}(D\mf{u})}^{2}+\snr{z_{0}}^{p-2}\left|\frac{\mf{u}}{\ti{\tau}_{2}-\ti{\tau}_{1}}\right|^{2}+\left|\frac{\mf{u}}{\ti{\tau}_{2}-\ti{\tau}_{1}}\right|^{p}\right] \dx\nonumber \\
&&+\frac{c\mathds{1}_{\{q>p\}}}{(\tau_{2}-\tau_{1})^{n\left(\frac{q}{p}-1\right)}}\left(\int_{S(x_{0})}\left[\snr{D\uu}^{p}+\left|\frac{\uu}{\tau_{2}-\tau_{1}}\right|^{p}\right] \dx\right)^{\frac{q}{p}},
\end{eqnarray*}
for $c\equiv c(n,N,p,q)$ and
\begin{eqnarray*}
\mbox{T}_{3}
&\le&c\snr{z_{0}}^{q-p}\int_{\ti{S}(x_{0})}\left[\snr{V_{\snr{z_{0}},p}(D\uu)}^{2}+\snr{z_{0}}^{p-2}\left|\frac{\uu}{\ti{\tau}_{2}-\ti{\tau}_{1}}\right|^{2}\right] \dx\nonumber \\
&&+\frac{c\mathds{1}_{\{q>p\}}}{(\tau_{2}-\tau_{1})^{n\left(\frac{q}{p}-1\right)}}\left(\int_{S(x_{0})}\left[\snr{D\uu}^{p}+\left|\frac{\uu}{\tau_{2}-\tau_{1}}\right|^{p}\right] \dx\right)^{\frac{q}{p}},
\end{eqnarray*}
with $c\equiv c(n,N,p,q)$. With these three estimates at hand, we can continue with the proof of Lemma \ref{ndegc} (without eventually reabsorbing those terms multiplying $\snr{z_{0}}^{q-p}$ in \eqref{ls.18}), use \eqref{equiv.1} and Lemma \ref{l5} to end up with \eqref{caccdeg}.
\end{proof}
At this stage, we prove that local minimizers of \eqref{exfun} are close to $p$-harmonic maps.
\begin{lemma}
Assume \eqref{assf}-\eqref{p0} and \eqref{f} and let $u\in W^{1,p}(\Omega,\mathbb{R}^{N})$ be a local minimizer of \eqref{exfun}. Then, for every ball $B_{\rr}(x_{0})\Subset \Omega$ and all $s\in (0,\infty)$ it holds that
\begin{eqnarray}\label{34}
\left| \ \mint_{B_{\rr}(x_{0})}\langle \snr{Du}^{p-2} Du,D\varphi\rangle \dx\ \right|&\le&s\nr{D\varphi}_{L^{\infty}(B_{\rr}(x_{0}))}\left(\mint_{B_{\rr}(x_{0})}\snr{Du}^{p} \dx\right)^{\frac{p-1}{p}}\nonumber \\
&&+c\nr{D\varphi}_{L^{\infty}(B_{\rr}(x_{0}))}\left(\omega(s)^{-1}+\omega(s)^{q-1-p}\right)\mint_{B_{\rr}(x_{0})}\snr{Du}^{p} \dx\nonumber \\
&&+4\nr{D\varphi}_{L^{\infty}(B_{\rr}(x_{0}))}\left(\rr^{m}\mint_{B_{\rr}(x_{0})}\snr{f}^{m} \dx\right)^{1/m},
\end{eqnarray}
for all $\varphi\in C^{\infty}_{\rm c}(B_{\rr}(x_{0}),\mathbb{R}^{N})$, with $c\equiv c(n,N,\Lambda,p,q)$.
\end{lemma}
\begin{proof}
With $\varphi\in C^{\infty}_{\rm c}(B_{\rr}(x_{0}),\mathbb{R}^{N})$ we split
\begin{eqnarray*}
\left|\ \mint_{B_{\rr}(x_{0})}\langle \snr{Du}^{p-2}Du,D\varphi\rangle  \dx\ \right|&\stackrel{\eqref{3}}{\le}&\left|\ \mint_{B_{\rr}(x_{0})}\langle\partial F(Du)-\partial F(0)-\snr{Du}^{p-2}Du,D\varphi\rangle  \dx\ \right|\nonumber \\
&&+\left| \ \mint_{B_{\rr}(x_{0})}f\cdot \varphi \dx \ \right|=:\mbox{(I)}+\mbox{(II)}.
\end{eqnarray*}
Given any $s\in (0,\infty)$, by \eqref{p0.1}, \eqref{df} and \eqref{pq} we estimate
\begin{eqnarray*}
\mbox{(I)}&\le&\frac{\nr{D\varphi}_{L^{\infty}(B_{\rr}(x_{0}))}}{\snr{B_{\rr}(x_{0})}}\int_{B_{\rr}(x_{0})\cap\left\{\snr{Du}\le \omega(s)\right\}}\snr{\partial F(Du)-\partial F(0)-\snr{Du}^{p-2}Du} \dx\nonumber \\
&&+\frac{\nr{D\varphi}_{L^{\infty}(B_{\rr}(x_{0}))}}{\snr{B_{\rr}(x_{0})}}\int_{B_{\rr}(x_{0})\cap\left\{\snr{Du}> \omega(s)\right\}}\snr{\partial F(Du)-\partial F(0)-\snr{Du}^{p-2}Du} \dx\nonumber \\
&\le&s\nr{D\varphi}_{L^{\infty}(B_{\rr}(x_{0}))}\mf{I}_{p}(Du;B_{\rr}(x_{0}))^{p-1}\nonumber \\
&&+\frac{c\nr{D\varphi}_{L^{\infty}(B_{\rr}(x_{0}))}}{\snr{B_{\rr}(x_{0})}}\int_{B_{\rr}(x_{0)}\cap\left\{\snr{Du}>\omega(s)\right\}}\left[\snr{Du}^{p-1}+\snr{Du}^{q-1}\right] \dx\nonumber\\
&\le&s\nr{D\varphi}_{L^{\infty}(B_{\rr}(x_{0}))}\mf{I}_{p}(Du;B_{\rr}(x_{0}))^{p-1}\nonumber \\
&&+\frac{c\nr{D\varphi}_{L^{\infty}(B_{\rr}(x_{0}))}}{\snr{B_{\rr}(x_{0})}^{\frac{1}{p}}}\left|B_{\rr}(x_{0})\cap\left\{\snr{Du}>\omega(s)\right\}\right|^{\frac{1}{p}}\mf{I}_{p}(Du;B_{\rr}(x_{0}))^{p-1}\nonumber\\
&&+\frac{c\nr{D\varphi}_{L^{\infty}(B_{\rr}(x_{0}))}}{\snr{B_{\rr}(x_{0})}^{\frac{p-q+1}{p}}}\left|B_{\rr}(x_{0})\cap\left\{\snr{Du}>\omega(s)\right\}\right|^{\frac{p-q+1}{p}}\mf{I}_{p}(Du;B_{\rr}(x_{0}))^{q-1}\nonumber \\
&\le&s\nr{D\varphi}_{L^{\infty}(B_{\rr}(x_{0}))}\mf{I}_{p}(Du;B_{\rr}(x_{0}))^{p-1}\nonumber \\
&&+\frac{c\nr{D\varphi}_{L^{\infty}(B_{\rr}(x_{0}))}}{\omega(s)\snr{B_{\rr}(x_{0})}^{\frac{1}{p}}}\left(\int_{B_{\rr}(x_{0})\cap\left\{\snr{Du}>\omega(s)\right\}}\snr{Du}^{p} \dx\right)^{\frac{1}{p}}\mf{I}_{p}(Du;B_{\rr}(x_{0}))^{p-1}\nonumber\\
&&+\frac{c\nr{D\varphi}_{L^{\infty}(B_{\rr}(x_{0}))}}{\omega(s)^{p-q+1}\snr{B_{\rr}(x_{0})}^{\frac{p-q+1}{p}}}\left(\int_{B_{\rr}(x_{0})\cap\left\{\snr{Du}>\omega(s)\right\}}\snr{Du}^{p} \dx\right)^{\frac{p-q+1}{p}}\mf{I}_{p}(Du;B_{\rr}(x_{0}))^{q-1}\nonumber \\
&\le&s\nr{D\varphi}_{L^{\infty}(B_{\rr}(x_{0}))}\mf{I}_{p}(Du;B_{\rr}(x_{0}))^{p-1}\nonumber \\
&&+c\nr{D\varphi}_{L^{\infty}(B_{\rr}(x_{0}))}\left(\omega(s)^{-1}+\omega(s)^{q-1-p}\right)\mf{I}_{p}(Du;B_{\rr}(x_{0}))^{p},
\end{eqnarray*}
for $c\equiv c(n,N,\Lambda,p,q)$. Term $\mbox{(II)}$ can be easily bounded via
\begin{eqnarray*}
\mbox{(II)}\le 4\nr{D\varphi}_{L^{\infty}(B_{\rr}(x_{0}))}\left(\rr^{m}\mint_{B_{\rr}(x_{0})}\snr{f}^{m} \dx\right)^{1/m}.
\end{eqnarray*}
Combining the content of the two above displays, we obtain \eqref{34} and the proof is complete.
\end{proof}
Now we are ready to prove an excess decay result valid for the degenerate case.
\begin{proposition}\label{p3}
Assume \eqref{assf}-\eqref{p0} and \eqref{f} and let $u\in W^{1,p}(\Omega,\mathbb{R}^{N})$ be a local minimizer of \eqref{exfun}. Then for every $\chi\in (0,1]$ there exists $\theta\equiv \theta(\textnormal{\texttt{data}},\chi)\in (0,2^{-10})$, $\varepsilon_{2}\equiv \varepsilon_{2}(\textnormal{\texttt{data}},\omega(\cdot),\chi)\in (0,1)$ such that if the smallness conditions
\begin{eqnarray}\label{32}
\chi\snr{(Du)_{B_{\rr}(x_{0})}}\le \mf{F}(u;B_{\rr}(x_{0}))\qquad \mbox{and}\qquad \mf{F}(u;B_{\rr}(x_{0}))\le \varepsilon_{2}
\end{eqnarray}
are satisfied on a ball $B_{\rr}(x_{0})\subset \mathbb{R}^{n}$, then
\begin{eqnarray}\label{51}
\mf{F}(u;B_{\theta\rr}(x_{0}))\le \theta^{\gamma_{0}}\mf{F}(u;B_{\rr}(x_{0}))+c_{1}\mf{K}\left[\left(\rr^{m}\mint_{B_{\rr}(x_{0})}\snr{f}^{m} \dx\right)^{1/m}\right]^{\frac{1}{p-1}},
\end{eqnarray}
for all $\gamma_{0}\in (0,2\alpha/p)$, where $\alpha\equiv \alpha(n,N,p)\in (0,1)$ is the exponent in $\eqref{h.2}_{2}$ and $c_{1}\equiv c_{1}(\textnormal{\texttt{data}}_{\textnormal{c}},\chi)$.
\end{proposition}
\begin{proof}
We premise that the same abbreviations used in the proof of Proposition \ref{p1} will be employed also here. By \eqref{32}$_{1}$ and triangular inequality we have that
\begin{eqnarray}\label{35}
\mf{I}_{p}(Du;B_{\rr})^{p} \le c_{\chi}\mf{F}(u;B_{\rr})^{p}\qquad \mbox{where} \ \  c_{\chi}:=2^{p}\left(1+\frac{1}{\chi^{p}}\right)>1,
\end{eqnarray}
therefore \eqref{34} reads as
\begin{eqnarray}\label{36}
\left| \ \mint_{B_{\rr}}\langle \snr{Du}^{p-2}Du,D\varphi\rangle \dx\ \right|&\stackrel{\eqref{35}}{\le}&sc_{\chi}^{(p-1)/p}\nr{D\varphi}_{\infty}\mf{F}(u;B_{\rr})^{p-1}\nonumber \\
&&+cc_{\chi}\nr{D\varphi}_{\infty}\left(\omega(s)^{-1}+\omega(s)^{q-1-p}\right)\mf{F}(u;B_{\rr})^{p}\nonumber \\
&&+4\nr{D\varphi}_{\infty}\left(\rr^{m}\mint_{B_{\rr}}\snr{f}^{m} \dx\right)^{1/m},
\end{eqnarray}
with $c\equiv c(n,N,\Lambda,p,q)$. We fix $\varepsilon_{3}\in (0,1]$ and define quantities
\begin{eqnarray*}
\kk:=c_{\chi}^{1/p}\mf{F}(u;B_{\rr})+\left(\frac{4}{\varepsilon_{3}}\right)^{1/(p-1)}\left(\rr^{m}\mint_{B_{\rr}}\snr{f}^{m} \dx\right)^{\frac{1}{m(p-1)}}\quad \mbox{and}\quad u_{0}:=\kk^{-1}u.
\end{eqnarray*}
Then we divide both sides of \eqref{36} by $\kk^{p-1}$ and use $\eqref{32}_{2}$ thus getting
\begin{flalign}
\left| \ \mint_{B_{\rr}}\langle \snr{Du_{0}}^{p-2}Du_{0},D\varphi\rangle \dx\ \right|\stackrel{\eqref{32}_{2}}{\le}(s+\varepsilon_{3})\nr{D\varphi}_{\infty}+cc_{\chi}^{1/p}\nr{D\varphi}_{\infty}\left(\omega(s)^{-1}+\omega(s)^{q-p-1}\right)\varepsilon_{2},
\label{37}
\end{flalign}
for $c\equiv c(n,N,\Lambda,p,q)$ and
\begin{eqnarray}\label{38}
\mf{I}_{p}(Du_{0};B_{\rr})^{p}\le\frac{\mf{I}_{p}(Du;B_{\rr})^{p}}{c_{\chi}\mf{F}(u;B_{\rr})^{p}} \stackrel{\eqref{35}}{\le}1.
\end{eqnarray}
Let $\varepsilon\in (0,1]$, $\theta\in (0,2^{-10})$ be numbers whose size will be quantified later on and $\delta\equiv \delta(n,N,p,\varepsilon)\in (0,1]$ be the small parameter given by Lemma \ref{phar}. We impose the following initial restrictions on the sizes of the various parameters appearing on the right hand side of \eqref{37}:
\begin{eqnarray}\label{52.1}
\varepsilon_{3}=\frac{\delta}{10},\qquad s=\frac{\delta }{10}
\end{eqnarray}
and
\begin{flalign}\label{52}
cc_{\chi}^{1/p}\left(\omega(s)^{-1}+\omega(s)^{q-1-p}\right)\varepsilon_{2}\le \frac{\delta}{10}
\end{flalign}
so we have dependencies $\varepsilon_{3},s\equiv \varepsilon_{3},s(n,N,p,\varepsilon)$ and $\varepsilon_{2}\equiv \varepsilon_{2}(n,N,\Lambda,p,q,\omega(\cdot),\chi,\varepsilon)$.
With these choices, \eqref{37} reads as
\begin{eqnarray*}
\left| \ \mint_{B_{\rr}}\langle \snr{Du_{0}}^{p-2}Du_{0},D\varphi\rangle \dx\ \right|\le \delta\nr{D\varphi}_{\infty},
\end{eqnarray*}
therefore recalling also \eqref{38} we see that Lemma \ref{phar} applies and renders a $p$-harmonic map $h\in u_{0}+ W^{1,p}_{0}(B_{\rr},\mathbb{R}^{N})$ such that
\begin{flalign}\label{39}
\mf{I}_{p}(Dh;B_{\rr})\le c,\qquad \mf{I}_{\ti{p}}(Du_{0}-Dh;B_{\rr})^{p}\le c\varepsilon,\qquad   \mf{I}_{p}\left(\frac{u_{0}-h}{\rr};B_{\rr}\right)^{p}\le c \varepsilon,
\end{flalign}
for $c\equiv c(n,N,p)$. Estimates \eqref{h.2} and \eqref{39}$_{1}$ in particular imply that
\begin{eqnarray}\label{40}
\mf{F}(h;B_{2\theta\rr})\le c \theta^{\alpha}\quad \mbox{and}\quad \mf{F}(h;B_{\rr})\le c,
\end{eqnarray}
with $c\equiv c(n,N,p)$. By triangular and Poincar\'e inequalities, \eqref{39}$_{3}$ and \eqref{40} we have
\begin{eqnarray*}
\mint_{B_{2\theta\rr}}\left|\frac{u_{0}-(h)_{2\theta\rr}-\langle(Dh)_{2\theta\rr},x-x_{0}\rangle}{2\theta\rr}\right|^{p} \dx &\le&c\mint_{B_{2\theta\rr}}\left| \frac{u_{0}-h}{2\theta\rr}\right|^{p} \dx\nonumber \\
&&+c\mint_{B_{2\theta\rr}}\left|\frac{h-(h)_{2\theta\rr}-\langle (Dh)_{2\theta\rr},x-x_{0}\rangle}{2\theta\rr}\right|^{p} \dx\nonumber \\
&\le&c\theta^{-(n+p)}\varepsilon+c\mint_{B_{2\theta\rr}}\snr{Dh-(Dh)_{2\theta\rr}}^{p} \dx\nonumber \\
&\le&c\left(\theta^{-(n+p)}\varepsilon+\theta^{\alpha p}\right),
\end{eqnarray*}
for $c\equiv c(n,N,p)$. Set $h_{0}(x):=\kk\left((h)_{2\theta\rr}+\langle(Dh)_{2\theta\rr},x-x_{0}\rangle\right)$ and notice that
\begin{eqnarray}\label{46}
\snr{Dh_{0}}=\kk\snr{(Dh)_{2\theta\rr}\mathds{I}_{n\times n}}\le c\kk \nr{Dh}_{L^{\infty}(B_{2\theta\rr})}\stackrel{\eqref{h.2}_{1}}{\le}c\kk\mf{I}_{p}(h;B_{\rr})&\stackrel{\eqref{39}_{1}}{\le}&c\kk,
\end{eqnarray}
with $c\equiv c(n,N,p)$. The content of the previous display and the very definition of $u_{0}$ yield that
\begin{eqnarray}\label{42}
\mint_{B_{2\theta\rr}}\left|\frac{u-h_{0}}{2\theta\rr}\right|^{p} \dx&\le& c\kk^{p}\left(\theta^{-(n+p)}\varepsilon+\theta^{\alpha p}\right),
\end{eqnarray}
for $c\equiv c(n,N,p)$. We further estimate by H\"older inequality
\begin{flalign}
\mint_{B_{2\theta\rr}}\snr{Dh_{0}}^{p-2}\left|\frac{u-h_{0}}{2\theta\rr}\right|^{2} \dx\stackrel{\eqref{46}}{\le}c\kk^{p-2}\left(\mint_{B_{2\theta\rr}}\left| \frac{u-h_{0}}{2\theta\rr} \right|^{p} \ \dx\right)^{2/p}\stackrel{\eqref{42}}{\le}c\kk^{p}\left(\theta^{-2(n+p)/p}\varepsilon^{2/p}+\theta^{2\alpha}\right),\label{41}
\end{flalign}
with $c\equiv c(n,N,p)$. We then apply \eqref{caccdeg} with $z_{0}\equiv Dh_{0}$ and \eqref{equiv.1} to get
\begin{eqnarray*}
\mf{F}(u,Dh_{0};B_{\theta\rr})^{p}&\le&c\mathds{1}_{\{q>p\}}\snr{Dh_{0}}^{q-p}\mf{F}(u,Dh_{0};B_{2\theta\rr})^{p}\nonumber \\
&&+c(1+\snr{Dh_{0}}^{q-p})\mint_{B_{2\theta\rr}}\snr{Dh_{0}}^{p-2}\left|\frac{u-h_{0}}{2\theta\rr} \right|^{2}+\left|\frac{u-h_{0}}{2\theta\rr} \right|^{p}\  \dx\nonumber \\
&&+c\left(\mint_{B_{2\theta\rr}}\snr{Dh_{0}}^{p-2}\left|\frac{u-h_{0}}{2\theta\rr} \right|^{2}+\left|\frac{u-h_{0}}{2\theta\rr} \right|^{p}\  \dx\right)^{q/p}\nonumber \\
&&+c\left((2\theta\rr)^{m}\mint_{B_{2\theta\rr}}\snr{f}^{m} \dx\right)^{\frac{p}{m(p-1)}}+c\mathds{1}_{\{q>p\}}\mf{F}(u,Dh_{0};B_{2\theta\rr})^{q}\nonumber \\
&=:&\mbox{(I)}+\mbox{(II)}+\mbox{(III)}+\mbox{(IV)}+\mbox{(V)},
\end{eqnarray*}
for $c\equiv c(\textnormal{\texttt{data}})$. Before estimating the various quantities appearing in the previous display, we first observe that by Young and triangular inequalities it is,
\begin{eqnarray}\label{47}
\mf{F}(u,Dh_{0};B_{2\theta\rr})^{p}&\stackrel{\eqref{46}}{\le}&c\kk^{p}+c\mf{I}_{p}(Du;B_{2\theta\rr})^{p}\stackrel{\eqref{35}}{\le}c\kk^{p}+cc_{\chi}\theta^{-n}\mf{F}(u;B_{\rr})^{p}\nonumber \\
&\le&cc_{\chi}\theta^{-n}\mf{F}(u;B_{\rr})^{p}+c\varepsilon_{3}^{-\frac{p}{p-1}}\left(\rr^{m}\mint_{B_{\rr}}\snr{f}^{m} \dx\right)^{\frac{p}{m(p-1)}},
\end{eqnarray}
with $c\equiv c(n,p)$. Keeping in mind that in $\mbox{(I)}$ and $\mbox{(V)}$ $q$ must be larger than $p$ otherwise both terms vanish, we bound via Young inequality with conjugate exponents $\left(\frac{q}{q-p},\frac{q}{p}\right)$
\begin{eqnarray*}
\mbox{(I)}+\mbox{(V)}&\stackrel{\eqref{46}}{\le}&c\mathds{1}_{\{q>p\}}\kk^{q-p}\mf{F}(u,Dh_{0};B_{2\theta\rr})^{p}+c\mathds{1}_{\{q>p\}}\mf{F}(u,Dh_{0};B_{2\theta\rr})^{q}\nonumber \\
&\le&c\mathds{1}_{\{q>p\}}\kk^{q}+c\mathds{1}_{\{q>p\}}\mf{F}(u,Dh_{0};B_{2\theta\rr})^{q}\nonumber \\
&\stackrel{\eqref{47}}{\le}&c\mathds{1}_{\{q>p\}}\theta^{-nq/p}c_{\chi}^{q/p}\mf{F}(u;B_{\rr})^{q}+c\mathds{1}_{\{q>p\}}\varepsilon_{3}^{-q/(p-1)}\left(\rr^{m}\mint_{B_{\rr}}\snr{f}^{m} \ \dx\right)^{\frac{q}{m(p-1)}}\nonumber \\
&\stackrel{\eqref{32}_{2}}{\le}&c\mathds{1}_{\{q>p\}}\theta^{-nq/p}c_{\chi}^{q/p}\varepsilon_{2}^{q-p}\mf{F}(u;B_{\rr})^{p}+c\mathds{1}_{\{q>p\}}\varepsilon_{3}^{-q/(p-1)}\left(\rr^{m}\mint_{B_{\rr}}\snr{f}^{m} \ \dx\right)^{\frac{q}{m(p-1)}},
\end{eqnarray*}
where it is $c\equiv c(\textnormal{\texttt{data}})$. Moreover we have
\begin{eqnarray*}
\mbox{(II)}+\mbox{(III)}&\stackrel{\eqref{46}}{\le}&c(1+\kk^{q-p})\mint_{B_{2\theta\rr}}\snr{Dh_{0}}^{p-2}\left|\frac{u-h_{0}}{2\theta\rr}\right|^{2}+\left|\frac{u-h_{0}}{2\theta\rr}\right|^{p} \dx\nonumber \\
&&+c\left(\mint_{B_{2\theta\rr}}\snr{Dh_{0}}^{p-2}\left|\frac{u-h_{0}}{2\theta\rr}\right|^{2}+\left|\frac{u-h_{0}}{2\theta\rr}\right|^{p} \dx\right)^{\frac{q}{p}}\nonumber \\
&\stackrel{\eqref{42},\eqref{41}}{\le}&c(1+\kk^{q-p})\kk^{p}\left(\theta^{-(n+p)}\varepsilon^{2/p}+\theta^{2\alpha}\right)+c\kk^{q}\left(\theta^{-q(n+p)/p}\varepsilon^{2/p}+\theta^{2\alpha}\right)\nonumber \\
&\le&cc_{\chi}^{q/p}\mf{F}(u;B_{\rr})^{p}\left(\theta^{-q(n+p)/p}\varepsilon^{2/p}+\theta^{2\alpha}\right)\nonumber \\
&&+c\varepsilon_{3}^{-\frac{q}{p-1}}\left(\theta^{-q(n+p)/p}\varepsilon^{2/p}+\theta^{2\alpha}\right)\mf{K}\left[\left(\rr^{m}\mint_{B_{\rr}}\snr{f}^{m} \ \dx\right)^{1/m}\right]^{\frac{p}{p-1}}
\end{eqnarray*}
for $c\equiv c(\textnormal{\texttt{data}})$ and, trivially,
\begin{eqnarray*}
\mbox{(IV)}\le c\theta^{\frac{p}{p-1}\left(1-\frac{n}{m}\right)}\left(\rr^{m}\mint_{B_{\rr}}\snr{f}^{m} \dx\right)^{\frac{p}{m(p-1)}},
\end{eqnarray*}
with $c\equiv c(\textnormal{\texttt{data}})$. Combining the content of the previous displays we obtain that
\begin{eqnarray}\label{50}
\mf{F}(u,Dh_{0};B_{\theta\rr})^{p}&\le&c\mathcal{S}_{1}\mf{F}(u;B_{\rr})^{p}+c\mathcal{S}_{2}\mf{K}\left[\left(\rr^{m}\mint_{B_{\rr}}\snr{f}^{m} \dx\right)^{1/m}\right]^{\frac{p}{p-1}},
\end{eqnarray}
where $c\equiv c(\textnormal{\texttt{data}})$ and for simplicity, we set
\begin{flalign*}
&\mathcal{S}_{1}\equiv\mathcal{S}_{1}(\chi,\varepsilon,\varepsilon_{2},\theta):=\mathds{1}_{\{q>p\}}c_{\chi}^{q/p}\theta^{-nq/p}\varepsilon_{2}^{q-p}+c_{\chi}^{q/p}\left(\theta^{-q(n+p)/p}\varepsilon^{2/p}+\theta^{2\alpha}\right)\nonumber \\
&\mathcal{S}_{2}\equiv\mathcal{S}_{2}(\chi,\varepsilon,\varepsilon_{3},\theta):=\varepsilon_{3}^{-\frac{q}{p-1}}\left(1+\theta^{-q(n+p)/p}\varepsilon^{2/p}+\theta^{2\alpha}\right)+\theta^{\frac{p}{p-1}\left(1-\frac{n}{m}\right)}.
\end{flalign*}
Using \eqref{50} and \eqref{minav} we then obtain
\begin{flalign}\label{60}
&\mint_{B_{\theta\rr}}\snr{Dh_{0}}^{p-2}\snr{Du-(Du)_{\theta\rr}}^{2}+\snr{Du-(Du)_{\theta\rr}}^{p}\dx\le 8^{p}\mf{F}(u,Dh_{0};B_{\theta\rr})^{p}\nonumber \\
&\qquad \qquad \quad \le c\mathcal{S}_{1}\mf{F}(u;B_{\rr})^{p}+c\mathcal{S}_{2}\mf{K}\left[\left(\rr^{m}\mint_{B_{\rr}}\snr{f}^{m} \dx\right)^{1/m}\right]^{\frac{p}{p-1}},
\end{flalign}
with $c\equiv c(\textnormal{\texttt{data}})$. Now, if $p=2$ we can directly jump to \eqref{43.1} (which would be essentially equivalent to \eqref{60} in this case), so in the next four displays we shall assume $p>2$. With $h$ being the $p$-harmonic map whose features are described in \eqref{39}, set $\bar{h}:=\kk h$ and split
\begin{eqnarray*}
\mf{F}(u;B_{\theta\rr})^{p}&\stackrel{\eqref{minav}}{\le}&\mint_{B_{\theta\rr}}\snr{Du-(Du)_{\theta\rr}}^{p} \ \dx+c\mathds{1}_{\{p>2\}}\mint_{B_{\theta\rr}}\snr{(Du)_{\theta\rr}}^{p-2}\snr{Du-(D\bar{h})_{\theta\rr}}^{2} \ \dx\nonumber \\
&\stackrel{\eqref{60}}{\le}&c\mathcal{S}_{1}\mf{F}(u;B_{\rr})^{p}+c\mathcal{S}_{2}\mf{K}\left[\left(\rr^{m}\mint_{B_{\rr}}\snr{f}^{m} \dx\right)^{1/m}\right]^{\frac{p}{p-1}}\nonumber \\
&&+c\mathds{1}_{\{p>2\}}\mint_{B_{\theta\rr}}\snr{(Du)_{\theta\rr}}^{p-2}\snr{Du-D\bar{h}}^{2} \ \dx+c\mf{F}(\bar{h};B_{\theta\rr})^{p}\nonumber \\
&&+c\mathds{1}_{\{p>2\}}\snr{(Du)_{\theta\rr}-(D\bar{h})_{\theta\rr}}^{p-2}\mint_{B_{\theta\rr}}\snr{D\bar{h}-(D\bar{h})_{\theta\rr}}^{2} \ \dx\nonumber \\
&\le&c\mathcal{S}_{1}\mf{F}(u;B_{\rr})^{p}+c\mathcal{S}_{2}\mf{K}\left[\left(\rr^{m}\mint_{B_{\rr}}\snr{f}^{m} \dx\right)^{1/m}\right]^{\frac{p}{p-1}}+c\mbox{(VI)}+c\mbox{(VII)}+c\mbox{(VIII)},
\end{eqnarray*}
with $c\equiv c(\textnormal{\texttt{data}})$. Keeping in mind that being $p>2$, in Lemma \ref{phar} it is $\ti{p}\ge 2$, we bound
\begin{eqnarray*}
\mbox{(VI)}&\le&\mathds{1}_{\{p>2\}}\kk^{2}\theta^{-n}\snr{(Du)_{\theta\rr}}^{p-2}\left(\mint_{B_{\rr}}\snr{Du_{0}-Dh}^{\ti{p}} \ \dx \right)^{2/\ti{p}}\stackrel{\eqref{39}_{2}}{\le}c\mathds{1}_{\{p>2\}}\varepsilon^{2/p}\kk^{2}\theta^{-n}\snr{(Du)_{\theta\rr}}^{p-2}\nonumber \\
&\le&c\mathds{1}_{\{p>2\}}\varepsilon^{2/p}\theta^{-n}\kk^{p}+c\mathds{1}_{\{p>2\}}\varepsilon^{2/p}\theta^{-n}\snr{(Du)_{\theta\rr}}^{p}\nonumber \\
&\stackrel{\eqref{35}}{\le}&c\mathds{1}_{\{p>2\}}c_{\chi}\varepsilon^{2/p}\theta^{-n}\mf{F}(u;B_{\rr})^{p}+c\mathds{1}_{\{p>2\}}\varepsilon^{2/p}\theta^{-n}\varepsilon_{3}^{-\frac{p}{p-1}}\left(\rr^{m}\mint_{B_{\rr}}\snr{f}^{m} \ \dx\right)^{\frac{p}{m(p-1)}}\nonumber \\
&&+c\mathds{1}_{\{p>2\}}c_{\chi}\varepsilon^{2/p}\theta^{-2n}\mf{F}(u;B_{\rr})^{p}\nonumber \\
&\le&c\mathds{1}_{\{p>2\}}c_{\chi}\varepsilon^{2/p}\theta^{-2n}\mf{F}(u;B_{\rr})^{p}+c\mathds{1}_{\{p>2\}}\varepsilon^{2/p}\theta^{-n}\varepsilon_{3}^{-\frac{p}{p-1}}\left(\rr^{m}\mint_{B_{\rr}}\snr{f}^{m} \ \dx\right)^{\frac{p}{m(p-1)}},
\end{eqnarray*}
for $c\equiv c(\textnormal{\texttt{data}})$. Moreover, combining $\eqref{h.2}_{2}$ and \eqref{40} we get
\begin{eqnarray*}
\mbox{(VII)}&=&\kk^{p}\mf{F}(h;B_{\theta\rr})^{p}\le c\kk^{p}\theta^{\alpha p}\le cc_{\chi}\theta^{\alpha p}\mf{F}(u;B_{\rr})^{p}+c\theta^{\alpha p}\varepsilon_{3}^{-\frac{p}{p-1}}\left(\rr^{m}\mint_{B_{\rr}}\snr{f}^{m} \ \dx\right)^{\frac{p}{m(p-1)}},
\end{eqnarray*}
with $c\equiv c(\textnormal{\texttt{data}})$. Finally, by H\"older inequality we have
\begin{eqnarray*}
\mbox{(VIII)}&\stackrel{\eqref{40}}{\le}&c\mathds{1}_{\{p>2\}}\theta^{-n}\kk^{p}\mf{I}_{\ti{p}}(Du_{0}-Dh;B_{\rr})^{p-2}\mf{F}(h;B_{\theta\rr})^{2}\nonumber \\
&\stackrel{\eqref{h.2}_{2},\eqref{39}_{2}}{\le}&c\mathds{1}_{\{p>2\}}\theta^{2\alpha-n}\varepsilon^{\frac{p-2}{p}}\kk^{p}\le c\mathds{1}_{\{p>2\}}c_{\chi}\theta^{2\alpha-n}\varepsilon^{\frac{p-2}{p}}\mf{F}(u;B_{\rr})^{p}\nonumber \\
&&+c\mathds{1}_{\{p>2\}}\theta^{2\alpha-n}\varepsilon^{\frac{p-2}{p}}\varepsilon_{3}^{-\frac{p}{p-1}}\left(\rr^{m}\mint_{B_{\rr}}\snr{f}^{m} \ \dx\right)^{\frac{p}{m(p-1)}},
\end{eqnarray*}
for $\equiv c(\textnormal{\texttt{data}})$. Merging the content of the three above displays, we obtain
\begin{eqnarray}\label{43.1}
\mf{F}(u;B_{\theta\rr})^{p}&\le&c\left(\mathcal{S}_{1}+c_{\chi}\theta^{-2n}\varepsilon^{2/p}+c_{\chi}\theta^{\alpha p}+\mathds{1}_{\{p>2\}}c_{\chi}\theta^{2\alpha-n}\varepsilon^{\frac{p-2}{p}}\right)\mf{F}(u;B_{\rr})^{p}\nonumber \\
&&+c\left(\mathcal{S}_{2}+\theta^{-n}\varepsilon_{3}^{-\frac{p}{p-1}}\right)\mf{K}\left[\left(\rr^{m}\mint_{B_{\rr}}\snr{f}^{m} \ \dx\right)^{1/m}\right]^{\frac{p}{p-1}},
\end{eqnarray}
with $c\equiv c(\textnormal{\texttt{data}})$. In \eqref{43.1} we first select $\varepsilon\in (0,1)$ in such a way that 
\begin{eqnarray}\label{58.1}
cc_{\chi}^{q/p}\varepsilon^{2/p}\left(\theta^{-q(n+p)/p}+\theta^{-2n}\right)+\mathds{1}_{\{p>2\}}cc_{\chi}\theta^{-n}\varepsilon^{(p-2)/p}\le \frac{\theta^{2\alpha}}{2^{40}},
\end{eqnarray}
and reduce further the size of $\varepsilon_{2}$ with respect to \eqref{52} in such a way that
\begin{eqnarray}\label{52.1.1}
2^{8}c_{\chi}\varepsilon_{2}^{p}+\mathds{1}_{\{q>p\}}c_{\chi}^{q/p}\theta^{-nq/p}\varepsilon_{2}^{q-p}\le \frac{\theta^{2\alpha}}{2^{40}}
\end{eqnarray}
so we can rearrange \eqref{43.1} as
\begin{eqnarray*}
\mf{F}(u;B_{\theta\rr})^{p}&\le&\left(cc_{\chi}^{q/p}\theta^{2\alpha}+\frac{\theta^{2\alpha}}{2^{20}}\right)\mf{F}(u;B_{\rr})^{p}\nonumber \\
&&+c\left(\mathcal{S}_{2}+\theta^{-n}\varepsilon_{3}^{-\frac{p}{p-1}}\right)\mf{K}\left[\left(\rr^{m}\mint_{B_{\rr}}\snr{f}^{m} \ \dx\right)^{1/m}\right]^{\frac{p}{p-1}}.
\end{eqnarray*}
We then take any $\gamma_{0}\in (0,2\alpha/p)$ with $\alpha$ as in \eqref{h.2}$_{2}$ and fix $\theta\in (0,2^{-10})$ so that $$cc_{\chi}^{q/p}\theta^{2\alpha-p\gamma_{0}}<\frac{1}{2^{40}}\qquad \mbox{and}\qquad \theta^{\gamma_{0}}\le 2^{-8}$$ which means that it is $\theta\equiv \theta(\textnormal{\texttt{data}},\chi)$ (where we included the dependencies on $\gamma_{0}$ into those on $\alpha$).
Keeping in mind \eqref{52.1}-\eqref{52} and \eqref{58.1}-\eqref{52.1.1}, we see that the previous position fix dependencies $\varepsilon,\varepsilon_{3}\equiv \varepsilon,\varepsilon_{3}(\textnormal{\texttt{data}},\chi)$ and $\varepsilon_{2}\equiv \varepsilon_{2}(\textnormal{\texttt{data}},\omega(\cdot),\chi)$. With these choices, we obtain \eqref{51} and the proof is complete.
\end{proof}

\section{Excess decay}\label{ex}
In this section we shall prove that the excess functional $\mf{F}(\cdot)$ decays by iterating Propositions \ref{p1}, \ref{p2} and \ref{p3} on a suitable subset of $\Omega$.
\subsection{The regular set}\label{rs}
With $u\in W^{1,p}(\Omega,\mathbb{R}^{N})$ being a local minimizer of \eqref{exfun}, we introduce the set 
\begin{flalign}\label{ru.0}
&\mathcal{R}_{u}:=\left\{\frac{}{} x_{0}\in \Omega\colon \exists \ M\equiv M(x_{0})\subset (0,\infty),\ \bar{\rr}\equiv \bar{\rr}(\textnormal{\texttt{data}}_{\textnormal{c}},M^{q-p},f(\cdot))<d_{x_{0}}, \ \bar{\varepsilon}\equiv \bar{\varepsilon}(\textnormal{\texttt{data}}_{\textnormal{c}},M^{q-p})\frac{}{}\right.\nonumber \\
&\qquad \qquad \qquad \qquad\qquad\qquad \left. \frac{}{}\mbox{such that} \ \snr{(Du)_{B_{\rr}(x_{0})}}<M \ \mbox{and} \ \mf{F}(u;B_{\rr}(x_{0}))<\bar{\varepsilon} \ \mbox{for some} \ \rr\in (0,\bar{\rr}]\frac{}{}\right\}
\end{flalign}
and prove that it is open and of full $n$-dimensional Lebesgue measure. The above position is quite involved and at a first sight all the parameters involved seem to feature mutual dependencies, so dangerous vicious circles may occur or the definition itself might be trivial. Let us show that \eqref{ru.0} makes sense. The idea consists in taking any point $x_{0}$ belonging to the interior $\Omega$, fixing a positive constant $M\equiv M(x_{0})$ depending at most from $x_{0}$ and $\bar{\varepsilon}$, $\bar{\rr}$ with (at most) the dependencies described in \eqref{ru.0}. Once identified $M$, $\bar{\varepsilon}$ and $\bar{\rr}$, we check if the two conditions appearing in \eqref{ru.0} are verified on some ball centered at $x_{0}$ with radius less than or equal to $\bar{\rr}$. This closes the ambiguity related to possible vicious circles, but still $\mathcal{R}_{u}$ may be empty. To show that the latter cannot be the case, let us introduce sets
\begin{flalign*}
&\mathcal{L}_{u}:=\left\{x_{0}\in \Omega\colon \lim_{\sigma\to 0}(Du)_{B_{\sigma}(x_{0})}=Du(x_{0})\right\},\qquad \ 
\mathcal{S}_{u}:=\left\{x_{0}\in \Omega\colon \limsup_{\sigma\to 0}\snr{(Du)_{B_{\sigma}(x_{0})}}<\infty\right\}\\
&\mathcal{W}_{u}:=\left\{x_{0}\in \Omega\colon \lim_{\sigma\to 0}\ti{\mf{F}}(u;B_{\sigma}(x_{0}))=0\right\},\qquad \qquad \ \mathcal{V}_{u}:=\left\{x_{0}\in \Omega\colon \liminf_{\sigma\to 0}\ti{\mf{F}}(u;B_{\sigma}(x_{0}))=0\right\}
\end{flalign*}
and note that by definition and standard Lebesgue theory it is \begin{flalign}\label{ru.1}
 \mathcal{L}_{u}\subset \mathcal{S}_{u},\qquad \mathcal{W}_{u}\subset \mathcal{V}_{u},\qquad \snr{\Omega\setminus (\mathcal{S}_{u}\cap \mathcal{V}_{u})}=0.
\end{flalign}
Now, if $x_{0}\in (\mathcal{S}_{u}\cap\mathcal{V}_{u})$ is any point, the quantity $M:=1+2\limsup_{\sigma\to 0}\snr{(Du)_{B_{\sigma}(x_{0})}}$ is finite and depends only on $x_{0}$, so we can determine the corresponding $\bar{\varepsilon}$ and the threshold radius $\bar{\rr}<d_{x_{0}}$. Moreover, by \eqref{equiv} it is $\mf{F}(\cdot)^{p/2}\approx\ti{\mf{F}}(\cdot)$, thus we can find a radius $\rr\equiv \rr(\textnormal{\texttt{data}}_{\textnormal{c}},M^{q-p},f(\cdot))\in (0,\bar{\rr}]$ such that both $\snr{(Du)_{B_{\rr}(x_{0})}}<M$ and $\mf{F}(u;B_{\rr}(x_{0}))<\bar{\varepsilon}$ hold true. Therefore we can conclude that $\left(\mathcal{S}_{u}\cap \mathcal{V}_{u}\right)\subset \mathcal{R}_{u}$, so by $\eqref{ru.1}_{3}$ it is $\snr{\Omega\setminus \mathcal{R}_{u}}=0$, and the definition in \eqref{ru.0} is nontrivial. Moreover $\mathcal{R}_{u}$ is an open set. To see this, let $x_{0}\in \mathcal{R}_{u}$, then by definition there exists $\rr\in (0,\bar{\rr}]$ such that $\snr{(Du)_{B_{\rr}(x_{0})}}<M$ and $\mf{F}(u;B_{\rr}(x_{0}))<\bar{\varepsilon}$. The absolute continuity of Lebesgue integral yields that there exists an open neighborhood $B(x_{0})$ of $x_{0}$ and a radius $\rr_{x_{0}}\in (0,\bar{\rr}]$ such that $B(x_{0})\Subset \Omega$ and 
\begin{flalign}\label{70}
\snr{(Du)_{B_{\rr_{x_{0}}}(x)}}<M\quad \mbox{and}\quad \mf{F}(u;B_{\rr_{x_{0}}}(x))<\bar{\varepsilon}\qquad \mbox{for all} \ \ x\in B(x_{0}).
\end{flalign}
This proves that $\mathcal{R}_{u}$ is open. In the next sections we shall prove that under suitable assumptions on the potential and for certain values of $\bar{\varepsilon}$ and of $\bar{\rr}$, on the set $\mathcal{R}_{u}$ the gradient of minima of \eqref{exfun} is continuous.
\begin{remark}\label{rx}
\emph{With $x_{0}\in \mathcal{R}_{u}$, all the radii considered from now on will be implicitly assumed to be less than $d_{x_{0}}$, therefore all balls centered at $x_{0}$ will be (compactly) contained in $B_{d_{x_{0}}}(x_{0})$.}
\end{remark}
\subsection{Iteration of the excess functional}\label{itex}
Let $x_{0}\in \mathcal{R}_{u}$ be a point verifying the conditions displayed in \eqref{ru.0} for some positive constant $M\equiv M(x_{0})$ and parameters $\bar{\varepsilon},\bar{\rr}$ still to be fixed, set $\nu:=1/4$, if $\beta_{0},\gamma_{0}\in (0,1)$ are the exponents appearing in \eqref{30}-\eqref{51} respectively, define $\alpha_{0}:=\min\{\beta_{0},\gamma_{0}\}$ and choose $\chi=\varepsilon_{0}$ in Proposition \ref{p3} thus fixing dependencies $\tau,\varepsilon_{1}\equiv \tau,\varepsilon_{1}(\textnormal{\texttt{data}},M^{q-p})$, $\varepsilon_{0},\theta\equiv \varepsilon_{0},\theta(\textnormal{\texttt{data}},\mu(\cdot),M^{q-p})$ and $\varepsilon_{2}\equiv \varepsilon_{2}(\textnormal{\texttt{data}}_{\textnormal{c}},M^{q-p})$. We then define constants
\begin{flalign*}
c_{2}:=4(c_{0}+c_{1}),\qquad c_{3}:=c_{2}\max\left\{\frac{1}{1-\nu^{\alpha_{0}}},\frac{1}{1-\tau^{\alpha_{0}}},\frac{1}{1-\theta^{\alpha_{0}}}\right\},
\end{flalign*}
where $c_{0},c_{1}$ are the same appearing in Propositions \ref{p2}-\ref{p3},
\begin{flalign}\label{hhh}
H_{1}:=\frac{2^{24p}}{(\theta\tau)^{np}\varepsilon_{0}^{p}},\qquad  H_{2}:=\max\left\{2^{16p}c_{2}^{p/2}H_{1},\left(\frac{2^{6p+8}H_{1}}{(\theta\tau)^{n/2}\varepsilon_{1}}\right)^{\frac{p}{2(p-1)}},\left(\frac{2^{80}}{\tau\theta}\right)^{np^{2}q/2}c_{3}^{p/2}H_{1}\right\}
\end{flalign}
and set
\begin{eqnarray}\label{be}
\hat{\varepsilon}:=\varepsilon_{2}\min\left\{\frac{(\tau\theta)^{16npq}}{2^{80npq}},\left(\frac{(\tau\theta)^{8npq}\varepsilon_{0}^{p}}{2^{16npq}}\right)^{2/p},\frac{(\theta\tau)^{16npq}\varepsilon_{0}^{npq}}{2^{180npq}}\right\}.
\end{eqnarray}
We assume that
\begin{eqnarray}\label{fpm.2}
\mathbf{I}_{1,m}^{f}(x_{0},1)<\infty,
\end{eqnarray}
so by the absolute continuity of Lebesgue integral we can find $\hat{\rr}\equiv\hat{\rr}(\textnormal{\texttt{data}}_{\textnormal{c}},M^{q-p},f(\cdot))\in (0,d_{x_{0}})$ such that
\begin{eqnarray}\label{fpm.1}
\frac{2^{64np(q+2)}c_{3}H_{2}^{1+2/p}}{(\varepsilon_{0}\tau\theta)^{4np(q+2)}}\mf{K}\left(\mathbf{I}^{f}_{1,m}(x_{0},s)\right)^{\frac{1}{p-1}}<\frac{\hat{\varepsilon}}{10}\qquad \mbox{for all} \ \ 0<s\le \hat{\rr}.
\end{eqnarray}
Now, if $\delta\in \{\nu,\tau,\theta\}$ and $0<s\le \hat{\rr}$, the very definition of potential $\mathbf{I}^{f}_{1,m}(\cdot)$ yields
\begin{eqnarray*}
\sum_{j=0}^{\infty}\left((\delta^{j+1}s)^{m}\mint_{B_{\delta^{j+1}s}(x_{0})}\snr{f}^{m} \dx\right)^{1/m}&\le&\frac{1}{\delta^{n/m}\log(\delta^{-1})}\sum_{j=0}^{\infty}\int_{\delta^{j+1}s}^{\delta^{j}s}\left(\sigma^{m}\mint_{B_{\sigma}(x_{0})}\snr{f}^{m} \dx\right)^{1/m} \ \frac{\d\sigma}{\sigma}\nonumber \\
&\le&\frac{\mathbf{I}^{f}_{1,m}(x_{0},s)}{\delta^{2n}}\le \frac{2^{4n}\mathbf{I}^{f}_{1,m}(x_{0},s)}{(\tau\theta)^{2n}},
\end{eqnarray*}
therefore
\begin{flalign}\label{d.13}
\max_{\delta\in \{\nu,\tau, \theta\}}\left\{ \sum_{j=0}^{\infty}\left((\delta^{j+1}s)^{m}\mint_{B_{\delta^{j+1}s}(x_{0})}\snr{f}^{m} \dx\right)^{1/m} \right\}\le \frac{2^{4n}\mathbf{I}^{f}_{1,m}(x_{0},s)}{(\tau\theta)^{2n}}.
\end{flalign}
Combining \eqref{d.13} with standard interpolation arguments and recalling that $\nu>\max\{\tau,\theta\}$, we have
\begin{eqnarray}\label{d.4}
\left(\sigma^{m}\mint_{B_{\sigma}(x_{0})}\snr{f}^{m} \dx\right)^{1/m}\le \frac{2^{8n}\mathbf{I}^{f}_{1,m}(x_{0},s)}{(\tau\theta)^{4n}}\qquad \mbox{for all} \ \ 0<\sigma\le s/4 .
\end{eqnarray}
We then deduce that
\begin{flalign}
\sup_{\sigma\le s/4}\mf{K}\left[\left(\sigma^{m}\mint_{B_{\sigma}(x_{0})}\snr{f}^{m} \dx\right)^{1/m}\right]^{1/(p-1)}\stackrel{\eqref{d.4}}{\le} \frac{2^{\frac{8nq}{p-1}}\mf{K}\left(\mathbf{I}^{f}_{1,m}(x_{0},s)\right)^{1/(p-1)}}{(\tau\theta)^{\frac{4nq}{p-1}}}
\stackrel{\eqref{fpm.1}}{<}\frac{\hat{\varepsilon}(\varepsilon_{0}\tau\theta)^{8np}}{2^{56np(q+2)}c_{3}H_{2}^{1+2/p}},\label{d.3}
\end{flalign}
for all $0<s\le \hat{\rr}$ with $\mf{K}(\cdot)$ defined in \eqref{ik}, and that
\begin{eqnarray}\label{d.3.1.1}
\lim_{\sigma\to 0}\mf{K}\left[\left(\sigma^{m}\mint_{B_{\sigma}(x_{0})}\snr{f}^{m} \dx\right)^{1/m}\right]^{1/(p-1)}=0.
\end{eqnarray}
 Finally in \eqref{ru.0} we fix $\bar{\varepsilon}=\hat{\varepsilon}$ and $\bar{\rr}=\hat{\rr}$, thus determining a ball $B_{\rr}(x_{0})\Subset \Omega$ with $\rr\in (0,\hat{\rr}]$ such that 
\begin{eqnarray}\label{55}
\snr{(Du)_{B_{\rr}(x_{0})}}< M\qquad\mbox{and}\qquad \mf{F}(u;B_{\rr}(x_{0}))<\hat{\varepsilon}.
\end{eqnarray}
Once the matter of determining the dependencies of the constants and fixing the threshold radii has been settled, we are ready to state
\begin{theorem}\label{t.ex}
Under assumptions \eqref{assf}-\eqref{p0}, \eqref{f} and \eqref{fpm.2}, let $u\in W^{1,p}(\Omega,\mathbb{R}^{N})$ be a local minimizer of \eqref{exfun}, $x_{0}\in \mathcal{R}_{u}$ be a point and $M\equiv M(x_{0})$ be the positive constant in \eqref{ru.0}. There exist $\hat{\varepsilon}\equiv \hat{\varepsilon}(\textnormal{\texttt{data}}_{\textnormal{c}},M^{q-p})\in (0,1)$ and $\hat{\rr}\equiv \hat{\rr}(\textnormal{\texttt{data}}_{\textnormal{c}},M^{q-p},f(\cdot))<d_{x_{0}}$ such that if $\bar{\varepsilon}\equiv \hat{\varepsilon}$ and $\bar{\rr}\equiv \hat{\rr}$ in \eqref{ru.0}, then
whenever $B_{\varsigma}(x_{0})\subset B_{\rr}(x_{0})$ is any ball, it is
\begin{eqnarray}\label{68.1}
\snr{(Du)_{B_{\varsigma}(x_{0})}}<2^{6}(M+1)
\end{eqnarray}
and
\begin{flalign}\label{68}
\mf{F}(u;B_{\varsigma}(x_{0}))\le c_{4}\left(\frac{\varsigma}{\rr}\right)^{\alpha_{0}}\mf{F}(u;B_{\rr}(x_{0}))+c_{5}\sup_{\sigma\le\rr/4  }\mf{K}\left[\left(\sigma^{m}\mint_{B_{\sigma}(x_{0})}\snr{f}^{m} \dx\right)^{1/m}\right]^{1/(p-1)},
\end{flalign}
with $c_{4},c_{5}\equiv{c}_{4},c_{5}(\textnormal{\texttt{data}}_{\textnormal{c}},M^{q-p})$.
\end{theorem}
\begin{proof}
The proof of Theorem \ref{t.ex} is quite involved and will be splitted into several steps eventually leading to the verification of \eqref{68.1}-\eqref{68}. 
\subsection*{Step 1: extra notation}
For $j\in \N\cup\{0\}$, $\tau,\theta$ as in Proposition \ref{p1}-\ref{p3} and $\nu$ as in Section \ref{itex}, we define $\tau_{j}:=\tau^{j}$, $\theta_{j}:=\theta^{j}$, $\nu_{j}:=\nu^{j}$ with $\tau_{0}\equiv\theta_{0}\equiv\nu_{0}=1$, $r_{1}:=\nu_{1}\rr$ and, for $s> 0$ we introduce the following symbols:
\begin{flalign*}
&\mf{F}(s):=\mf{F}(u;B_{s}(x_{0})),\qquad \qquad\quad \ \  A(s):=\snr{(Du)_{B_{s}(x_{0})}}, \nonumber \\
&\mf{C}(s):=A(s)^{p/2}+H_{1}\mf{F}(s)^{p/2},\qquad \mf{S}(s):=\left(s^{m}\mint_{B_{s}(x_{0})}\snr{f}^{m} \ \dx\right)^{1/m}.
\end{flalign*}
We shall refer to $\mf{C}(\cdot)$ as the composite excess functional. According to a by now standard terminology \cite{dumi}, given any ball $B\Subset \Omega$, we shall refer to the validity of \eqref{5} on $B$ as the "nondegenerate regime", while the occurrence of the complementary condition \eqref{32}$_{1}$ means that the problem is in a "degenerate regime".
\subsection*{Step 2: iterations at a fixed scale}
We estimate
\begin{flalign}
\left\{
\begin{array}{c}
\displaystyle
\ \mf{F}(r_{1})\stackrel{\eqref{tri.1}}{\le}2^{3+2n/p}\mf{F}(\rr)\stackrel{\eqref{55}_{2}}{<}2^{3+2n/p}\hat{\varepsilon}\stackrel{\eqref{be}}{<}2^{-6nq}(\tau\theta)^{4npq}\varepsilon_{2}<\varepsilon_{2}\\[8pt]\displaystyle 
\ A(r_{1})\stackrel{\eqref{55}_{1}}{<}M+\snr{A(r_{1})-A(\rr)}\le M+2^{2n/p}\mf{F}(\rr)\stackrel{\eqref{55}_{2}}{<}M+2^{2n/p}\hat{\varepsilon}\stackrel{\eqref{be}}{\le} M+\varepsilon_{2}.
\end{array}
\right.\label{a.10}
\end{flalign}
This is enough to start iterations.
\subsection*{Step 3: a technical inductive lemma} Let us record an inductive lemma that will ultimately allow transferring to successive scales the boundedness of averages while preserving the nonlinear potential $\mathbf{I}_{1,m}^{f}(\cdot)$.
\begin{lemma}\label{til}
Let $M\ge 0$ be the constant in \eqref{ru.0}, $\gamma\in \left(0,2^{5p}(M+1)^{p/2}\right)$ any number, $\tau\in (0,2^{-10})$ the small parameter from Propositions \ref{p1}-\ref{p2}, $r\in (0,r_{1}]$ a radius and assume that for integers $k\ge i\ge 0$ the following inequalities:
\begin{eqnarray}\label{ls.130}
\varepsilon_{0}A(\tau_{j}r)>\mf{F}(\tau_{j}r)\qquad \mbox{for all} \ \ j\in \{i,\cdots,k\},
\end{eqnarray}
and
\begin{eqnarray}\label{ls.131}
\left\{
\begin{array}{c}
\displaystyle 
\ \mf{C}(\tau_{j}r)\le \gamma,\quad \mf{C}(\tau_{j+1}r)\ge \frac{\gamma}{16} \ \ \mbox{for all} \ \ j\in \{i,\cdots,k\},\quad \mf{C}(\tau_{i}r)\le \frac{\gamma}{4}\\[8pt]\displaystyle
\ \left(\sum_{j=i}^{k}\mf{S}(\tau_{j}r)\right)^{\frac{p}{2(p-1)}}\le \frac{2\gamma}{H_{2}}
\end{array}
\right.
\end{eqnarray}
hold with $H_{2}$ described in \eqref{hhh}. Then
\begin{flalign}\label{ls.132}
 \mf{C}(\tau_{k+1}r)\le \gamma\qquad \mbox{and}\qquad \sum_{j=i}^{k+1}\mf{F}(\tau_{j}r)^{p/2}\le \frac{\gamma}{2H_{1}}
\end{flalign}
and
\begin{eqnarray}\label{ls.133}
\sum_{j=i}^{k+1}\mf{F}(\tau_{j}r)^{p/2}\le 2\mf{F}(\tau_{i}r)^{p/2}+\frac{2^{7p}\gamma^{(2-p)/p}}{\varepsilon_{1}\tau^{n/2}}\sum_{j=i}^{k}\mf{S}(\tau_{j}r),
\end{eqnarray}
with $H_{1}$, $\varepsilon_{1}$ as in \eqref{hhh} and Propositions \ref{p1}-\ref{p2} respectively.
\end{lemma}
\begin{proof}
The proof is a variation of the one of \cite[Lemma 6.1]{kumi}. There are only three different points: first, we need to enforce the original statement with assumption \eqref{ls.130} to guarantee the stability of the nondegenerate regime over the whole set of indices $\{i,\cdots,k\}$; second, we have to restrict the range of admissible values for $\gamma$ in order to apply Propositions \ref{p1}-\ref{p2}. Keeping in mind these two details, the proof is the same.
\end{proof}

\subsection*{Step 4: maximal iteration chains} Given any nonempty set of indices $\mathcal{J}_{0}\subset \N\cup \{0\}$, following \cite[Section 12.4]{kumig}, for $\kk\in \N$ we introduce the maximal iteration chain of length $\kk$ starting at $\iota$:
\begin{eqnarray*}
\mathcal{C}^{\kk}_{\iota}:=\left\{j\in \N\cup\{0\} \colon \iota \le j\le \iota+\kk,\  \iota \in \mathcal{J}_{0}, \ \iota+\kk+1\in \mathcal{J}_{0}, \ j\not \in \mathcal{J}_{0} \ \mbox{if} \ j>\iota\right\}.
\end{eqnarray*}
In other words, $\mathcal{C}^{\kk}_{\iota}=\{\iota,\iota+1,\cdots,\iota+\kk\}$ and all its elements lie outside $\mathcal{J}_{0}$ but $\iota$, which belongs to $\mathcal{J}_{0}$. Furthermore, $\mathcal{C}^{\kk}_{\iota}$ is maximal, in the sense that there are no other sets of the same type properly containing it. In the same way, we define the infinite maximal chain starting at $\iota$ as
\begin{eqnarray*}
\mathcal{C}^{\infty}_{\iota}:=\left\{j\in \N\cup\{0\} \colon \iota \le j<\infty ,\  \iota \in \mathcal{J}_{0},  \ j\not \in \mathcal{J}_{0} \ \mbox{if} \ j>\iota\right\}.
\end{eqnarray*}
We then consider two alternatives:
\begin{eqnarray}\label{ls.100}
\mf{C}(r_{1})> \frac{H_{2}}{\varepsilon_{0}^{p/2}}\mf{K}_{\mf{s}}^{\frac{p}{2(p-1)}}\qquad \mbox{or}\qquad \mf{C}(r_{1})\le \frac{H_{2}}{\varepsilon_{0}^{p/2}}\mf{K}_{\mf{s}}^{\frac{p}{2(p-1)}},
\end{eqnarray}
where we set $\mf{K}_{\mf{s}}:=\sup_{\sigma\le \rr/4}\mf{K}(\mf{S}(\sigma))$, with $\mf{K}(\cdot)$ and $r_{1}$ defined in \eqref{ik} and in \textbf{Step 1} respectively, $H_{2}$ is the constant in \eqref{hhh} and $\varepsilon_{0}$ is the same parameter appearing in \eqref{5}.
\subsection*{Step 5: starting from large composite excess - \eqref{ls.100}$_{1}$ holds} With $\nu\in (0,1)$ as in \textbf{Step 1}, we characterize the set $\mathcal{J}_{0}$ as
\begin{eqnarray*}
\mathcal{J}_{0}:=\left\{j\in \N\cup\{0\}\colon \mf{C}(\nu_{j}r_{1})> \frac{H_{2}}{\varepsilon_{0}^{p/2}}\mf{K}_{\mf{s}}^{\frac{p}{2(p-1)}}\right\},
\end{eqnarray*}
which is nonempty because by $\eqref{ls.100}_{1}$ we have that $0\in \mathcal{J}_{0}$. Looking at the iteration chains introduced in \textbf{Step 4}, we have two possibilities:
\begin{itemize}
    \item[\emph{i.}] there is at least one maximal iteration chain $\mathcal{C}^{\kk}_{\iota}$ for some $\iota\in \mathcal{J}_{0}$, $\kk\le \infty$;
    \item[\emph{ii.}] $\mathcal{J}_{0}=\mathbb{N}\cup\{0\}$.
\end{itemize}
We start with (\emph{i}.) and consider the worst possible scenario: there are infinitely many finite iteration chains $\mathcal{C}^{\kk_{d}}_{\iota_{d}}$ corresponding to discrete sequences $\{\iota_{d}\}_{d\in \N}, \{\kk_{d}\}_{d\in \N}\subset \N$. Notice that by maximality it is
\begin{flalign}\label{ls.104}
  \iota_{d+1}\ge \iota_{d}+\kk_{d}+1 \ \mbox{for all} \ d\in \N  \ \Longrightarrow \ \{\iota_{d}\}_{d\in \N}  \ \mbox{is strictly increasing and} \ \iota_{d}\to \infty.
\end{flalign}
For $d\in \N$, set
\begin{flalign*}
&\texttt{I}_{0}:=(\nu_{\iota_{1}}r_{1},r_{1}],\qquad\qquad \  \texttt{K}_{d}:=(\nu_{\iota_{d}+\kk_{d}+1}r_{1},\nu_{\iota_{d}+1}r_{1}]\nonumber \\ &\texttt{I}_{d}^{1}:=(\nu_{\iota_{d}+1}r_{1},\nu_{\iota_{d}}r_{1}],\qquad \texttt{I}_{d}^{2}:=(\nu_{\iota_{d+1}}r_{1},\nu_{\iota_{d}+\kk_{d}+1}r_{1}].
\end{flalign*}
The above positions make sense in the light of \eqref{ls.104}. Moreover, by \eqref{ls.104} we can split interval $(0,r_{1}]$ into the disjoint union
\eqn{ls.150}
$$
(0,r_{1}]=\left(\texttt{I}_{0}\cup \texttt{I}^{1}_{1}\cup \texttt{K}_{1}\right)\cup\left(\bigcup_{d\in \N}\texttt{I}^{2}_{d}\cup \texttt{I}^{1}_{d+1}\cup\texttt{K}_{d+1}\right)=:\texttt{B}_{0}\cup\left( \bigcup_{d\in \N}\texttt{B}_{d}\right),
$$
where we introduced the blocks $\texttt{B}_{0}:=\texttt{I}_{0}\cup \texttt{I}^{1}_{1}\cup\texttt{K}_{1}$ and $\texttt{B}_{d}:=\texttt{I}^{2}_{d}\cup \texttt{I}^{1}_{d+1}\cup\texttt{K}_{d+1}$.
Precisely $\texttt{I}_{0}$, $\texttt{K}_{d}$, $\texttt{I}_{d}^{1}$, $\texttt{I}_{d}^{2}$ are mutually disjoint thus $\{\texttt{B}_{d}\}_{d\in \N\cup\{0\}}$ are disjoint as well and only $\texttt{I}_{d}^{2}$ can be empty. Before proceeding further, let us remark that given any $\varsigma\in (0,r_{1}]$ for which there is $j\in \N\cup \{0\}$ so that $\nu_{j+1}r_{1}<\varsigma\le \nu_{j}r_{1}$, by \eqref{ls.42.1}-\eqref{tri.1.2} it holds that
\begin{eqnarray}\label{ls.134}
\frac{1}{2^{3p+n+2}}\mf{C}(\nu_{j+1}r_{1})\le \mf{C}(\varsigma)\le 2^{3p+n+2}\mf{C}(\nu_{j}r_{1}).
\end{eqnarray}
By definition of maximal iteration chain, we have 
\begin{flalign}\label{ls.101}
\mf{C}(\nu_{\iota_{d}}r_{1})>\frac{H_{2}}{\varepsilon_{0}^{p/2}}\mf{K}_{\mf{s}}^{\frac{p}{2(p-1)}}\quad \mbox{and}\quad \mf{C}(\nu_{j}r_{1})\le \frac{H_{2}}{\varepsilon_{0}^{p/2}}\mf{K}_{\mf{s}}^{\frac{p}{2(p-1)}} \ \ \mbox{for all} \ \ j\in \{\iota_{d}+1,\cdots,\iota_{d}+\kk_{d}\},
\end{flalign}
so whenever $\varsigma\in \texttt{K}_{d}$ we can find $j_{\varsigma}\in \{\iota_{d}+1,\cdots,\iota_{d}+\kk_{d}\}$ satisfying $\nu_{j_{\varsigma}+1}r_{1}<\varsigma\le \nu_{j_{\varsigma}}r_{1}$ and
\begin{eqnarray}\label{ls.102}
\mf{C}(\varsigma)\stackrel{\eqref{ls.134}}{\le}2^{3p+n+2}\mf{C}(\nu_{j_{\varsigma}}r_{1})\stackrel{\eqref{ls.101}_{2}}{\le}\frac{2^{3p+n+2}H_{2}}{\varepsilon_{0}^{p/2}}\mf{K}_{\mf{s}}^{\frac{p}{2(p-1)}},
\end{eqnarray}
while if $\varsigma\in \texttt{I}_{0}$ or $\varsigma\in \texttt{I}^{2}_{d}$ then there is $j_{\varsigma}\in \{0,\cdots,\iota_{1}-1\}$ or $j_{\varsigma}\in \{\iota_{d}+\kk_{d}+1,\cdots,\iota_{d+1}-1\}$ so that $\nu_{j_{\varsigma}+1}r_{1}<\varsigma\le \nu_{j_{\varsigma}}r_{1}$ and
\begin{eqnarray}\label{ls.102.1}
\mf{C}(\varsigma)\stackrel{\eqref{ls.134}}{\ge} \frac{1}{2^{3p+n+2}}\mf{C}(\nu_{j_{\varsigma}+1}r_{1})\ge \frac{H_{2}}{2^{3p+n+2}\varepsilon_{0}^{p/2}}\mf{K}_{\mf{s}}^{\frac{p}{2(p-1)}},
\end{eqnarray}
where we used the definition of maximal iteration chain to assure that
\begin{eqnarray}\label{ls.102.1.1}
\{0,\cdots,\iota_{1}\}, \ \{\iota_{d}+\kk_{d}+1,\cdots, \iota_{d+1}\}\subset \mathcal{J}_{0}.
\end{eqnarray}
Notice that the occurrence $\texttt{I}_{d}^{2}=\{\emptyset\}$ generates two consecutive chains. This is possible only if $\iota_{d+1}=\iota_{d}+\kk_{d}+1$, which means that \eqref{ls.101}$_{2}$ is verified for all $j\in \{\iota_{d}+1,\cdots,\iota_{d+1}-1,\iota_{d+1},\iota_{d+1}+1,\cdots\iota_{d+1}+\kk_{d+1}\}\setminus \{\iota_{d+1}\}$ and
\begin{flalign}
\frac{2^{3p+n+2}H_{2}}{\varepsilon_{0}^{p/2}}\mf{K}_{\mf{s}}^{\frac{p}{2(p-1)}}\stackrel{\eqref{ls.101}}{\ge}2^{3p+n+2}\mf{C}(\nu_{\iota_{d}+\kk_{d}}r_{1})\stackrel{\eqref{ls.134}}{\ge}\mf{C}(\nu_{\iota_{d+1}}r_{1})>\frac{H_{2}}{\varepsilon_{0}^{p/2}}\mf{K}_{\mf{s}}^{\frac{p}{2(p-1)}},\label{ls.106.1}
\end{flalign}
therefore for $\varsigma\in \texttt{K}_{d}\cup \texttt{I}_{d+1}^{1}\cup \texttt{K}_{d+1}$ there is $j_{\varsigma}\in \{\iota_{d}+1,\cdots,\iota_{d+1}+\kk_{d+1}\}$ so that $\nu_{j_{\varsigma}+1}r_{1}<\varsigma\le \nu_{j_{\varsigma}}r_{1}$ and
\begin{eqnarray}\label{lms.10}
\mf{C}(\varsigma)\stackrel{\eqref{ls.134},\eqref{ls.106.1}}{\le} \frac{2^{6p+2n+4}H_{2}}{\varepsilon_{0}^{p/2}}\mf{K}_{\mf{s}}^{\frac{p}{2(p-1)}}.
\end{eqnarray}
\subsubsection*{Step 5.1: two alternatives at different scales}
At this stage, we need to consider two occurrences:
\begin{eqnarray}\label{ls.112}
\varepsilon_{0}A(r_{1})\le \mf{F}(r_{1})\qquad \mbox{or}\qquad \varepsilon_{0}A(r_{1})>\mf{F}(r_{1}).
\end{eqnarray}
Assume that $\eqref{ls.112}_{1}$ holds and, with $\theta$ being the same small parameter appearing in Proposition \ref{p3}, let us introduce the new set of indices
\begin{eqnarray*}
\mathcal{J}_{1}:=\left\{j\in \N\cup\{0\}\colon \varepsilon_{0}A(\theta_{j}r_{1})\le \mf{F}(\theta_{j}r_{1})\right\},
\end{eqnarray*}
which is nonempty as $0\in \mathcal{J}_{1}$, cf. \eqref{ls.112}$_{1}$. 
\subsubsection*{Step 5.2: the degenerate regime is stable}
If $\mathcal{J}_{1}\equiv \N\cup\{0\}$, with $\eqref{a.10}_{1}$ at hand, we apply Proposition \ref{p3} to get
\begin{eqnarray*}
\left\{
\begin{array}{c}
\displaystyle 
\ \mf{F}(\theta_{1} r_{1})\stackrel{\eqref{51}}{\le}\theta^{\gamma_{0}}\mf{F}(r_{0})+c_{1}\mf{K}(\mf{S}(r_{1}))^{1/(p-1)}\le \frac{\theta^{\gamma_{0}}}{2^{6nq}}\varepsilon_{2}+c_{3}\mf{K}_{\mf{s}}^{1/(p-1)}\stackrel{\eqref{d.3}}{<}\varepsilon_{2} \\[8pt]\displaystyle
\ A(r_{1})\le \varepsilon_{0}^{-1}\mf{F}(r_{1})<\varepsilon_{2}\varepsilon_{0}^{-1}\stackrel{\eqref{52.1.1}}{<}1.
\end{array}
\right.
\end{eqnarray*}
By induction, fix $j\in \N$ and suppose that
\eqn{aaa}
$$
\mf{F}(\theta_{i}r_{1})<\varepsilon_{2}\qquad \mbox{for all} \ \ i\in \{0,\cdots,j\},
$$
which, given that the degeneracy condition defining $\mathcal{J}_{1}$ holds for all $i\in \{0,\cdots,j\}$, readily implies 
$$
A(\theta_{i}r_{1})\le \varepsilon_{0}^{-1}\mf{F}(\theta_{i}r_{1})<\varepsilon_{0}^{-1}\varepsilon_{2}\stackrel{\eqref{52.1.1}}{<}1\qquad \mbox{for all} \ \ i\in \{0,\cdots,j\}.
$$
 Since $\mathcal{J}_{1}\equiv \N\cup \{0\}$, we can then apply repeatedly Proposition \ref{p3} and obtain
\begin{eqnarray}\label{ls.108}
\mf{F}(\theta_{j+1}r_{1})&\stackrel{\eqref{51}}{\le}&\theta^{\gamma_{0}}\mf{F}(\theta_{j}r_{1})+c_{1}\mf{K}(\mf{S}(\theta_{j}r_{1}))^{1/(p-1)}\nonumber \\
&\le&\theta^{\gamma_{0}(j+1)}\mf{F}(r_{1})+c_{1}\sum_{i=0}^{j}\theta^{\gamma_{0}(j-i)}\mf{K}(\mf{S}(\theta_{i}r_{1}))^{1/(p-1)}\nonumber \\
&\le&\theta^{\gamma_{0}(j+1)}\mf{F}(r_{1})+c_{3}\mf{K}_{\mf{s}}^{1/(p-1)}\stackrel{\eqref{a.10}_{1},\eqref{d.3}}{<}\varepsilon_{2}.
\end{eqnarray}
In particular, for all $i\in \{0,\cdots,j+1\}$ it is
\begin{eqnarray}\label{ls.109}
A(\theta_{i}r_{1})&\le&\frac{1}{\varepsilon_{0}}\mf{F}(\theta_{i}r_{1})\stackrel{\eqref{aaa},\eqref{ls.108}}{\le}\frac{\varepsilon_{2}}{\varepsilon_{0}}\stackrel{\eqref{52.1.1}}{<}1.
\end{eqnarray}
The arbitrariety of $j\in\N$ and \eqref{ls.108} yield that the iteration of Proposition \ref{p3} is stable and, as a consequence, that \eqref{ls.108}-\eqref{ls.109} hold for all $j\in \N$. Now, since we assumed that $\mathcal{J}_{1}\equiv \N\cup\{0\}$, given any $\varsigma\in (0,r_{1}]$, regardless which one of the intervals in decomposition \eqref{ls.150} it belongs to, we can find $j_{\varsigma}\in \N\cup\{0\}$ so that $\theta_{j_{\varsigma}+1}r_{1}<\varsigma\le \theta_{j_{\varsigma}}r_{1}$ and
\begin{eqnarray}\label{lms.24}
\mf{F}(\varsigma)&\stackrel{\eqref{ls.42.1}}{\le}&\frac{2^{6}}{\theta^{n/p}}\mf{F}(\theta_{j_{\varsigma}}r_{1})\stackrel{\eqref{ls.108}}{\le}\frac{2^{6}}{\theta^{n/p}}\left(\theta^{\gamma_{0}j_{\varsigma}}\mf{F}(r_{1})+c_{3}\mf{K}_{\mf{s}}^{1/(p-1)}\right)\nonumber \\
&\le&\frac{2^{6}}{\theta^{n/p+\gamma_{0}}}\left(\frac{\varsigma}{r_{1}}\right)^{\gamma_{0}}\mf{F}(r_{1})+\frac{2^{6}c_{3}}{\theta^{n/p}}\mf{K}_{\mf{s}}^{1/(p-1)}\nonumber \\
&\stackrel{\eqref{tri.1}}{\le}&\frac{2^{12+2n+2\gamma_{0}}}{\theta^{n/p+\gamma_{0}}}\left(\frac{\varsigma}{\rr}\right)^{\gamma_{0}}\mf{F}(\rr)+\frac{2^{6}c_{3}}{\theta^{n/p}}\mf{K}_{\mf{s}}^{1/(p-1)}.
\end{eqnarray}
On the other hand, if $\varsigma\in (r_{1},\rr]$, by means of \eqref{ls.42.1} we have
\begin{eqnarray}\label{lms.24.1}
\mf{F}(\varsigma)\le 2^{6+2n+2\gamma_{0}}\left(\frac{\varsigma}{\rr}\right)^{\gamma_{0}}\mf{F}(\rr).
\end{eqnarray}
Moreover, for $\varsigma\in (0,r_{1}]$ as above, we have
\begin{eqnarray}\label{lms.14}
A(\varsigma)&\le&\snr{A(\theta_{j_{\varsigma}}r_{1})-A(\varsigma)}+A(\theta_{j_{\varsigma}}r_{1})\stackrel{\eqref{ls.109}}{<}\frac{1}{\theta^{n/p}}\mf{F}(\theta_{j_{\varsigma}}r_{1})+1\nonumber \\
&\stackrel{\eqref{ls.108}}{\le}&\frac{\theta^{j_{\varsigma}\gamma_{0}}}{\theta^{n/p}}\mf{F}(r_{1})+\frac{c_{3}}{\theta^{n/p}}\mf{K}_{\mf{s}}^{1/(p-1)}+1\stackrel{\eqref{a.10}}{\le}2,
\end{eqnarray}
while for $\varsigma\in (r_{1},\rr]$, proceeding as before we have
\begin{eqnarray}\label{lms.25}
A(\varsigma)\stackrel{\eqref{55}_{1}}{\le} 2^{n}\mf{F}(\rr)+M\stackrel{\eqref{55}_{2},\eqref{be}}{<}M+1.
\end{eqnarray}
All in all, for any $\varsigma\in (0,\rr]$ we have
\begin{flalign}\label{final.0}
 \left\{
\begin{array}{c}
\displaystyle 
\ \mf{F}(\varsigma)\le \frac{2^{14+2n}}{\theta^{n/p+1}}\left(\frac{\varsigma}{\rr}\right)^{\gamma_{0}}\mf{F}(\rr)+\frac{2^{8}c_{3}}{\theta^{n/p}}\mf{K}_{\mf{s}}^{1/(p-1)}\\[15pt]\displaystyle
\ A(\varsigma)\le 2(M+1).
\end{array}
\right.
\end{flalign}
\subsubsection*{Step 5.3: first exit time}
Next we consider the occurrence $\mathcal{J}_{1}\not \equiv\mathbb{N}\cup \{0\}$, which means that there is $j_{1}\in \N$ satisfying
\begin{eqnarray*}
j_{1}:=\min\left\{j\in \N\colon \varepsilon_{0}A(\theta_{j}r_{1})>\mf{F}(\theta_{j}r_{1})\right\};
\end{eqnarray*}
in other words,
\begin{flalign}\label{ls.151}
\varepsilon_{0}A(\theta_{j}r_{1})\le\mf{F}(\theta_{j}r_{1}) \ \mbox{for all} \ j\in \{0,\cdots,j_{1}-1\}\quad\mbox{and}\quad \varepsilon_{0}A(\theta_{j_{1}}r_{1})>\mf{F}(\theta_{j_{1}}r_{1}).
\end{flalign}
We stress that $j_{1}\ge 1$ because of $\eqref{ls.112}_{1}$. We then define $r_{1}':=\theta_{j_{1}}r_{1}$, $r_{1}'':=\tau_{1}r_{1}'$ and consider another set of indices
$$
\mathcal{J}_{2}:=\left\{j\in \N\cup\{0\}\colon \varepsilon_{0}A(\tau_{j}r_{1}')>\mf{F}(\tau_{j}r_{1}')\right\},
$$
notice that it is nonempty as $0\in \mathcal{J}_{2}$ owing to the very definition of $j_{1}$. The minimality of $j_{1}$ and $\eqref{ls.151}_{1}$ imply that \eqref{ls.108} holds for all $j\in \{0,\cdots,j_{1}-1\}$ and
\begin{flalign}
A(r_{1}')\le\snr{A(r_{1}')-A(\theta_{j_{1}-1}r_{1})}+A(\theta_{j_{1}-1}r_{1})\stackrel{\eqref{ls.109}}{<}\frac{1}{\theta^{n/p}}\mf{F}(\theta_{j_{1}-1}r_{1})+1\stackrel{\eqref{ls.108},\eqref{a.10}_{1}}{\le}2.\label{ls.152}
\end{flalign}
Moreover, since \eqref{ls.108} and $\eqref{ls.151}_{1}$ hold for $j=j_{1}-1$, we get that
\begin{eqnarray}\label{ls.167}
\mf{F}(r_{1}')\le\theta^{\gamma_{0}j_{1}}\mf{F}(r_{1})+c_{3}\mf{K}_{\mf{s}}^{1/(p-1)}\stackrel{\eqref{a.10}_{1}}{\le} 2^{3+2n/p}\hat{\varepsilon}+c_{3}\mf{K}_{\mf{s}}^{1/(p-1)}.
\end{eqnarray}
By \eqref{ls.152}, \eqref{ls.167}, \eqref{be} and \eqref{d.3} we obtain
\begin{eqnarray}\label{ls.167.1}
A(r_{1}'')\le \frac{1}{\tau^{n/p}}\mf{F}(r_{1}')+A(r_{1}')<\frac{2^{3+2n/p}\hat{\varepsilon}}{\tau^{n/p}}+\frac{c_{3}\mf{K}_{\mf{s}}^{1/(p-1)}}{\tau^{n/p}}+2\le 3.
\end{eqnarray}
\subsubsection*{Step 5.4: the nondegenerate regime is stable}
Suppose now that $\mathcal{J}_{2}\equiv \N\cup \{0\}$ and introduce the nonhomogeneous excess functional:
\begin{eqnarray*}
(0,\rr]\ni \sigma\mapsto \mf{N}(x_{0},\sigma):=\frac{2^{8npq}H_{1}}{(\tau\theta)^{4npq}}\mf{F}(\sigma)^{p/2}+\frac{2^{16npq}H_{2}}{(\tau\theta)^{4npq}}\mf{K}\left(\mathbf{I}^{f}_{1,m}(x_{0},\sigma)\right)^{\frac{p}{2(p-1)}},
\end{eqnarray*}
where $H_{1}$, $H_{2}$ are as in \eqref{hhh}, and assume that
\begin{eqnarray}\label{ls.153}
\frac{\gamma}{8}:=A(r_{1}'')^{p/2}>\frac{\mf{N}(x_{0},r_{1}')}{16} \ \stackrel{\eqref{ls.167.1}}{\Longrightarrow} \ \gamma\le 2^{3+p}.
\end{eqnarray} 
Notice that by the very definition of $\mf{N}(x_{0},r_{1}')$ we have
\begin{flalign}
\mf{C}(r_{1}'')\stackrel{\eqref{tri.1}}{\le}\frac{2^{3p}H_{1}}{\tau^{n/2}}\mf{F}(r_{1}')^{p/2}+A(r_{1}'')^{p/2}\le\frac{2^{3p}(\tau\theta)^{8npq}}{2^{8npq}\tau^{n/2}}\mf{N}(x_{0},r_{1}')+A(r_{1}'')^{p/2}\stackrel{\eqref{ls.153}}{\le}\frac{2\gamma}{2^{6npq}}+\frac{\gamma}{8}\le \frac{\gamma}{4}.\label{ls.163}
\end{flalign}
Furthermore, via \eqref{tri.1}, \eqref{tri.1.1} and \eqref{ls.153} it is
\begin{eqnarray*}
A(r_{1}'')^{p/2}\le\frac{2^{6p}}{\tau^{n}}\mf{F}(r_{1}')^{p/2}+A(\tau_{1}r_{1}'')^{p/2}\le \frac{(\tau\theta)^{2npq}}{2^{6npq}H_{1}}\mf{N}(x_{0},r_{1}')+A(\tau r_{1}'')^{p/2}\le \frac{1}{2}A(r_{1}'')^{p/2}+A(\tau r_{1}'')^{p/2},
\end{eqnarray*}
which in turn implies that
\begin{eqnarray}\label{ls.155}
A(\tau_{1}r_{1}'')^{p/2}\stackrel{\eqref{ls.36}}{\ge}\frac{A(r_{1}'')^{p/2}}{2}.
\end{eqnarray}
Next, keeping in mind that $\mathcal{J}_{2}\equiv \N\cup\{0\}$, we extend the validity of \eqref{ls.155} for all $j\in \N$, i.e.:
\begin{eqnarray}\label{ls.160}
A(\tau_{j}r_{1}'')^{p/2}\ge \frac{A(r_{1}'')^{p/2}}{2}\qquad \mbox{for all} \ \ j\in \N.
\end{eqnarray}
The case $j=0$ is a triviality, case $j=1$ is \eqref{ls.155}, therefore by contradiction let us assume that there is a finite exit time index $J\ge 2$ so that
\begin{flalign}\label{ls.161}
A(\tau_{J}r_{1}'')^{p/2}<\frac{A(r_{1}'')^{p/2}}{2}\quad \mbox{and}\quad A(\tau_{j}r_{1}'')^{p/2}\ge \frac{A(r_{1}'')^{p/2}}{2} \ \ \mbox{for all} \ \ j\in \{0,\cdots,J-1\}.
\end{flalign}
Let us show the implication
\begin{flalign}
A(\tau_{j}r_{1}'')^{p/2}\ge \frac{A(r_{1}'')^{p/2}}{2}  \ \mbox{for all}  \ j\in \{0,\cdots,J-1\} \ \Longrightarrow \ \mf{C}(\tau_{j}r_{1}'')\le \gamma  \ \mbox{for all}  \ j\in \{0,\cdots,J-1\}.\label{ls.162}
\end{flalign}
The induction basis is \eqref{ls.163}. We then assume that
\begin{eqnarray}\label{ls.163.1}
\mf{C}(\tau_{j}r_{1}'')\le \gamma\quad \mbox{for all} \ \ j\in \{0,\cdots,k\},
\end{eqnarray}
for some $k\le J-2$ and prove that $\mf{C}(\tau_{k+1}r_{1}'')\le \gamma$. Proceeding as done for \eqref{d.13} we obtain
\begin{eqnarray}\label{ls.157}
\left(\sum_{j=0}^{\infty}\mf{S}(\tau_{j}r_{1}'')\right)^{\frac{p}{2(p-1)}}&=&\left(\sum_{j=0}^{\infty}\mf{S}(\tau_{j+1}r_{1}')\right)^{\frac{p}{2(p-1)}}\stackrel{\eqref{d.13}}{\le}\frac{2^{\frac{2np}{p-1}}\mathbf{I}^{f}_{1,m}(x_{0},r_{1}')^{\frac{p}{2(p-1)}}}{(\tau\theta)^{\frac{np}{p-1}}}\nonumber \\
&\le& \frac{(\tau\theta)^{2npq}}{2^{12npq}H_{2}}\mf{N}(x_{0},r_{1}')\stackrel{\eqref{ls.153}}{\le} \frac{\gamma}{2H_{2}}.
\end{eqnarray}
Next, we have by definition that
\begin{eqnarray}\label{ls.164}
\mf{C}(\tau_{j}r_{1}'')\ge A(\tau_{j}r_{1}'')^{p/2}\stackrel{\eqref{ls.161}}{\ge}\frac{A(r_{1}'')^{p/2}}{2}\stackrel{\eqref{ls.153}}{=}\frac{\gamma}{16}\qquad \mbox{for all} \ \ j\in \{0,\cdots,J-1\},
\end{eqnarray}
therefore, by \eqref{ls.163.1}, \eqref{ls.157} and recalling that $\mathcal{J}_{2}\equiv \N\cup \{0\}$ and that $k+1\le J-1$ we apply Lemma \ref{til} with $i=0$ and $k$ being the number fixed here to get that $\mf{C}(\tau_{k+1}r_{1}'')\le \gamma$ and \eqref{ls.162} is verified. Being $\mathcal{J}_{2}\equiv \N\cup\{0\}$ and with \eqref{ls.163}, \eqref{ls.163.1}, \eqref{ls.157} and \eqref{ls.164} available, we can now exploit Lemma \ref{til} with $i=0$ and $k=J-2$ to get
\begin{eqnarray}\label{ls.165}
\sum_{j=0}^{J-1}\mf{F}(\tau_{j}r_{1}'')^{p/2}&\stackrel{\eqref{ls.132}_{2}}{\le}&\frac{\gamma}{2H_{1}}\stackrel{\eqref{ls.153}}{=}\frac{4A(r_{1}'')^{p/2}}{H_{1}},
\end{eqnarray}
so
\begin{eqnarray*}
\snr{A(\tau_{J}r_{1}'')^{p/2}-A(r_{1}'')^{p/2}}&\le&\sum_{j=0}^{J-1}\snr{A(\tau_{j+1}r_{1}'')^{p/2}-A(\tau_{j}r_{1}'')^{p/2}}\nonumber \\
&\stackrel{\eqref{tri.1}}{\le}&\frac{2^{3p}}{\tau^{n/2}}\sum_{j=0}^{J-1}\mf{F}(\tau_{j}r_{1}'')^{p/2}\stackrel{\eqref{ls.165}}{\le}\frac{2^{3(p+1)}A(r_{1}'')^{p/2}}{H_{1}\tau^{n/2}}\stackrel{\eqref{hhh}}{\le}\frac{A(r_{1}'')^{p/2}}{4}
\end{eqnarray*}
and we can conclude with
\begin{eqnarray*}
A(\tau_{J}r_{1}'')^{p/2}\ge A(r_{1}'')^{p/2}-\snr{A(\tau_{J}r_{1}'')^{p/2}-A(r_{1}'')^{p/2}}\ge \frac{3A(r_{1}'')^{p/2}}{4},
\end{eqnarray*}
in contradiction with \eqref{ls.162}. This and the arbitrariety of $J\ge 2$ prove \eqref{ls.160}. With \eqref{ls.160} at hand, we can apply \eqref{ls.162} for all $j\in \N\cup \{0\}$ to conclude with 
\begin{eqnarray}\label{ls.166}
A(\tau_{j}r_{1}'')^{p/2}\le \mf{C}(\tau_{j}r_{1}'')\le \gamma\stackrel{\eqref{ls.153}}{\le}2^{3+p}\qquad \mbox{for all} \ \ j\in \N\cup\{0\}.
\end{eqnarray}
Let us consider the complementary case to \eqref{ls.153}, i.e.:
\begin{eqnarray}\label{ls.168}
A(r_{1}'')^{p/2}\le \frac{\mf{N}(x_{0},r_{1}')}{16}=:\frac{\gamma}{8} \ \Longrightarrow \ \gamma\le 2,
\end{eqnarray}
where the last inequality is justified by means of \eqref{ls.167}, \eqref{be}, \eqref{d.3} and \eqref{hhh}. We aim at proving that
\begin{eqnarray}\label{ls.170}
\mf{C}(\tau_{j}r_{1}'')\le \gamma\qquad \mbox{for all} \ \ j\in \N\cup\{0\}.
\end{eqnarray}
Notice that
\begin{flalign}
\mf{C}(r_{1}'')\stackrel{\eqref{tri.1}}{\le}\frac{2^{3p}H_{1}}{\tau^{n/2}}\mf{F}(r_{1}')^{p/2}+A(r_{1}'')^{p/2}\le \frac{(\tau\theta)^{2npq}}{2^{6npq}}\mf{N}(x_{0},r_{1}')+A(r_{1}'')^{p/2}\stackrel{\eqref{ls.168}}{\le}\frac{\gamma}{4}.\label{ls.169}
\end{flalign}
We then define $k:=\min\{s\in \N\cup \{0\}\colon \mf{C}(\tau_{s+1}r_{1}'')>\gamma\}$ as the smallest integer minus one for which \eqref{ls.170} fails and introduce the set
\begin{eqnarray*}
\mathcal{I}_{k}:=\left\{j\in \N\cup\{0\}\colon \mf{C}(\tau_{j}r_{1}'')\le \gamma/4, \ j<k+1\right\}
\end{eqnarray*}
and let $b:=\max\mathcal{I}_{k}$. The inequality in \eqref{ls.169} yields that $\mathcal{I}_{k}\not =\{\emptyset\}$ as $0\in \mathcal{I}_{k}$ and if $j\in \{b+1,\cdots,k+1\}$ it is $\mf{C}(\tau_{j}r_{1}'')\ge \gamma/4>\gamma/16$. Finally, the minimality of $k$ yields that if $j\in \{b,\cdots,k\}$ it is $\mf{C}(\tau_{j}r_{1}'')\le \gamma$. To apply Lemma \ref{til}, we only need to check the validity of $\eqref{ls.131}_{2}$ with the value of $\gamma$ displayed in \eqref{ls.168}, which can be proven as done for \eqref{ls.157}. Lemma \ref{til} applies and renders that $\mf{C}(\tau_{k+1}r_{1}'')\le \gamma$ in contradiction with the definition of $k$. This proves \eqref{ls.170}. By \eqref{ls.152}, \eqref{ls.167.1} and \eqref{ls.166}-\eqref{ls.170} (keep in mind the definition of $r_{1}''$) we deduce that
\begin{eqnarray}\label{ls.171}
A(\tau_{j}r_{1}')\le 2^{5}\qquad \mbox{for all} \ \ j\in \N\cup\{0\}
\end{eqnarray}
so recalling that $\mathcal{J}_{2}\equiv \N\cup\{0\}$, we can apply repeatedly Propositions \ref{p1}-\ref{p2} to get that
\begin{eqnarray}\label{ls.172}
\mf{F}(\tau_{j+1}r_{1}')\le\tau^{(j+1)\beta_{0}}\mf{F}(r_{1}')+c_{0}\sum_{i=0}^{j}\tau^{(j-i)\beta_{0}}\mf{S}(\tau_{i}r_{1}')^{1/(p-1)}\le\tau^{(j+1)\beta_{0}}\mf{F}(r_{1}')+c_{3}\mf{K}_{\mf{s}}^{1/(p-1)}.
\end{eqnarray}
The fact that $\mathcal{J}_{2}\equiv \N\cup\{0\}$ yields that there is only one change of scale, therefore we can ignore again the presence of blocks $\{\texttt{B}_{d}\}_{d\in \N\cup\{0\}}$ and split the interval $(0,\rr]$ as $(0,r_{1}']\cup(r_{1}',r_{1}]\cup (r_{1},\rr]$. If $\varsigma\in (0,r_{1}']$, we can find $j_{\varsigma}\in \N\cup\{0\}$ so that $\tau_{j_{\varsigma}+1}r_{1}'<\varsigma\le \tau_{j_{\varsigma}}r_{1}'$ and
\begin{eqnarray}\label{lms.26}
\mf{F}(\varsigma)&\stackrel{\eqref{ls.42.1}}{\le}&\frac{2^{6}}{\tau^{n/p}}\mf{F}(\tau_{j_{\varsigma}}r_{1}')\stackrel{\eqref{ls.172}}{\le}\frac{2^{6}\tau^{j_{\varsigma}\beta_{0}}}{\tau^{n/p}}\mf{F}(r_{1}')+\frac{2^{6}c_{3}}{\tau^{n/p}}\mf{K}_{\mf{s}}^{1/(p-1)}\nonumber \\
&\stackrel{\eqref{ls.167}}{\le}&\frac{2^{6}\tau^{j_{\varsigma}\beta_{0}}\theta^{j_{1}\gamma_{0}}}{\tau^{n/p}}\mf{F}(r_{1})+\frac{2^{6}c_{3}}{\tau^{n/p}}\mf{K}_{\mf{s}}^{1/(p-1)}(\tau^{j_{\varsigma}\beta_{0}}+1)\nonumber \\
&\stackrel{\eqref{tri.1}}{\le}&\frac{2^{14+2n}}{\tau^{n/p+1}}\left(\frac{\varsigma}{\rr}\right)^{\alpha_{0}}\mf{F}(\rr)+\frac{2^{8}c_{3}}{\tau^{n/p}}\mf{K}_{\mf{s}}^{1/(p-1)},
\end{eqnarray}
where we used that $\alpha_{0}=\min\{\beta_{0},\gamma_{0}\}$ by definition and also
\begin{eqnarray}\label{lms.15}
A(\varsigma)&\stackrel{\eqref{ls.171}}{\le}&\snr{A(\varsigma)-A(\tau_{j_{\varsigma}}r_{1}')}+2^{5}\le\frac{1}{\tau^{n/p}}\mf{F}(\tau_{j_{\varsigma}}r_{1}')+2^{5}\nonumber \\
&\stackrel{\eqref{ls.167},\eqref{ls.172}}{\le}&\frac{\tau^{j_{\varsigma}\beta_{0}}\theta^{j_{\varsigma}\gamma_{0}}}{\tau^{n/p}}\mf{F}(r_{1})+\frac{2c_{3}}{\tau^{n/p}}\mf{K}_{\mf{s}}^{1/(p-1)}+2^{5}\nonumber \\
&\stackrel{\eqref{ls.167}}{\le}&2^{3+2n/p}\tau^{-n/p}\hat{\varepsilon}+2c_{3}\mf{K}_{\mf{s}}^{1/(p-1)}+2^{5}\stackrel{\eqref{be},\eqref{d.3}}{\le}1+2^{5}.
\end{eqnarray}
On the other hand, if $\varsigma\in (r_{1}',r_{1}]$ there is $j_{\varsigma}\in \{0,\cdots,j_{1}-1\}$ so that $\theta_{j_{\varsigma}+1}r_{1}<\varsigma\le \theta_{j_{\varsigma}}r_{1}$ and
\begin{eqnarray}\label{lms.2}
\mf{F}(\varsigma)&\stackrel{\eqref{ls.42.1}}{\le}&\frac{2^{6}}{\theta^{n/p}}\mf{F}(\theta_{j_{\varsigma}}r_{1})\stackrel{\eqref{ls.108}}{\le}\frac{2^{6}}{\theta^{n/p+\gamma_{0}}}\left(\frac{\varsigma}{r_{1}}\right)^{\gamma_{0}}\mf{F}(r_{1})+\frac{2^{6}c_{3}}{\theta^{n/p}}\mf{K}_{\mf{s}}^{1/(p-1)}\nonumber \\
&\stackrel{\eqref{tri.1}}{\le}&\frac{2^{14+2n}}{\theta^{n/p+1}}\left(\frac{\varsigma}{\rr}\right)^{\gamma_{0}}\mf{F}(\rr)+\frac{2^{6}c_{3}}{\theta^{n/p}}\mf{K}_{\mf{s}}^{1/(p-1)}
\end{eqnarray}
and
\begin{eqnarray}\label{lms.3}
A(\varsigma)&\stackrel{\eqref{ls.109}}{\le}&\snr{A(\varsigma)-A(\theta_{j_{\varsigma}}r_{1})}+1\stackrel{\eqref{ls.108}}{\le}\frac{\theta^{j_{\varsigma}\gamma_{0}}}{\theta^{n/p}}\mf{F}(r_{1})+\frac{c_{3}}{\theta^{n/p}}\mf{K}_{\mf{s}}^{1/(p-1)}+1\stackrel{\eqref{a.10}_{1},\eqref{tri.1},\eqref{d.3}}{\le}2
\end{eqnarray}
where of course we applied \eqref{ls.108}-\eqref{ls.109} within the admissible range $j_{\varsigma}\in \{0,\cdots,j_{1}-1\}$. Finally, if $\varsigma\in (r_{1},\rr]$, applying \eqref{tri.1} and \eqref{a.10} we get
\begin{eqnarray}\label{lms.4}
\mf{F}(\varsigma)\le 2^{8+2n}\left(\frac{\varsigma}{\rr}\right)^{\alpha_{0}}\mf{F}(\rr)\qquad \mbox{and}\qquad A(\varsigma)\le M+1.
\end{eqnarray}
To summarize, we have just proven that for all $\varsigma\in (0,\rr]$ it is
\begin{flalign}\label{final.1}
 \left\{
\begin{array}{c}
\displaystyle 
\ \mf{F}(\varsigma)\le \frac{2^{14+2n}}{(\tau\theta)^{n/p+1}}\left(\frac{\varsigma}{\rr}\right)^{\alpha_{0}}\mf{F}(\rr)+\frac{2^{8}c_{3}}{(\tau\theta)^{n/p}}\mf{K}_{\mf{s}}^{1/(p-1)}\\[15pt]\displaystyle
\ A(\varsigma)\le 2^{5}(M+1).
\end{array}
\right.
\end{flalign}
\begin{remark}\label{stin}
\emph{
We stress that if instead of $\eqref{ls.112}_{1}$ we started from $\eqref{ls.112}_{2}$, to prove \eqref{final.1} in case of stability of the nondegenerate regime we would just need to replace $r_{1}'$ with $r_{1}$, $r_{1}''$ with $\tau_{1} r_{1}$, \eqref{ls.152}-\eqref{ls.167} and \eqref{ls.108}-\eqref{ls.109}, with \eqref{a.10} and \eqref{55} in \emph{Step 5.4} and to notice that this time in \eqref{ls.153} it is $\gamma\le 2^{3}(M+2)^{p/2}$, cf. $\eqref{a.10}_{2}$, while \eqref{ls.168} remains unchanged due to \eqref{a.10}$_{1}$, \eqref{be}, \eqref{fpm.1}, \eqref{hhh}; or to directly jump to \emph{Step 5.5} below, again with $r_{1}$ instead of $r_{1}'$, $r_{1}''$ replaced by $\tau_{1}r_{1}$ in case the nondegenerate regime is unstable.}
\end{remark}
\subsubsection*{Step 5.5: second exit time} Assume now that $\mathcal{J}_{2}\not =\N\cup\{0\}$. This means that there is $j_{2}\in \N$ with the property that
\begin{eqnarray*}
j_{2}:=\min\left\{j\in \N\colon \varepsilon_{0}A(\tau_{j}r_{1}')\le \mf{F}(\tau_{j}r_{1}')\right\}.
\end{eqnarray*}
Notice that being $\mathcal{J}_{2}\not =\{\emptyset\}$, it is $j_{2}\ge 1$ cf. \emph{Step 5.3} and, by minimality it is
\begin{flalign}\label{ls.175}
\varepsilon_{0}A(\tau_{j}r_{1}')>\mf{F}(\tau_{j}r_{1}') \ \ \mbox{for all} \ \ j\in \{0,\cdots,j_{2}-1\} \quad \mbox{and}\quad \varepsilon_{0}A(\tau_{j_{2}}r_{1}')\le \mf{F}(\tau_{j_{2}}r_{1}').
\end{flalign}
We can then repeat the same procedure\footnote{The only detail to be careful about is the passage from $r_{1}'$ to $r_{1}''=\tau_{1}r_{1}$, so by \eqref{ls.175}$_{1}$, any application of Lemma \ref{til} is possible over nonempty (possibly improper) subsets of $\{0,\cdots,j_{2}-2\}$. This forces the indices $J$ and $k$ employed in the proof of \eqref{ls.166} and of \eqref{ls.170} respectively to verify $J\le j_{2}$ and $k\le j_{2}-2$. This choice renders \eqref{ls.176} with no other variations to the procedure described in \emph{Step 5.4}. Notice that there is no loss of generality in assuming $j_{2}\ge 2$, otherwise \eqref{ls.176} would trivially follow by means of \eqref{ls.152} and \eqref{ls.167}.} leading to \eqref{ls.171}-\eqref{ls.172} iterating this time on $j$ belonging to the finite set $\{0,\cdots, j_{2}-1\}$ to get
\begin{flalign}\label{ls.176}
A(\tau_{j}r_{1}')\le 2^{5} \qquad \mbox{and}\qquad \mf{F}(\tau_{j+1}r_{1}')\le \tau^{(j+1)\beta_{0}}\mf{F}(r_{1}')+c_{3}\mf{K}_{\mf{s}}^{1/(p-1)}.
\end{flalign}
Set $r_{2}':=\tau_{j_{2}}r_{1}'$. We claim that
\begin{eqnarray}\label{lms.1}
r_{2}' \ \ \mbox{cannot belong to} \ \ \texttt{I}_{0} \ \ \mbox{or to} \ \ \texttt{I}_{d}^{2} \ \ \mbox{for all} \ \ d\in \N.
\end{eqnarray}
If this were the case we would have that
\begin{eqnarray}
A(\tau_{j_{2}-1}r_{1}')&\le& \frac{1}{\tau^{n/p}}\mf{F}(\tau_{j_{2}-1}r_{1}')+A(r_{2}')\nonumber \\
&\stackrel{\eqref{ls.175}_{1}}{\le}&\frac{\varepsilon_{0}}{\tau^{n/p}}A(\tau_{j_{2}-1}r_{1}')+A(r_{2}')\ \stackrel{\eqref{ls.36}}{\Longrightarrow }\ 2A(r_{2}')\ge A(\tau_{j_{2}-1}r_{1}'),\label{ls.123}
\end{eqnarray}
therefore keeping in mind that the validity of $\eqref{ls.175}_{1}$ and $\eqref{ls.176}_{1}$ allow applying Propositions \ref{p1}-\ref{p2}, we get
\begin{eqnarray*}
\mf{F}(r_{2}')&\stackrel{\eqref{30},\eqref{31}}{\le}&\tau^{\beta_{0}}\mf{F}(\tau_{j_{2}-1}r_{1}')+c_{0}\mf{S}(\tau_{j_{2}-1}r_{1}')^{1/(p-1)}\stackrel{\eqref{ls.175}_{1},\eqref{ls.123}}{\le}2\tau^{\beta_{0}}\varepsilon_{0}A(r_{2}')+c_{3}\mf{K}_{\mf{s}}^{1/(p-1)}\nonumber \\
&\stackrel{\eqref{ls.102.1}}{\le}&2\tau^{\beta_{0}}\varepsilon_{0}A(r_{2}')+\frac{c_{3}\varepsilon_{0}2^{2(3p+n+2)/p}}{H_{2}^{2/p}}\left(H_{1}^{2/p}\mf{F}(r_{2}')+A(r_{2}')\right)\nonumber \\
&\stackrel{\eqref{tri.1}}{\le}&\varepsilon_{0}A(r_{2}')\left(2\tau^{\beta_{0}}+\frac{c_{3}2^{2(3p+n+2)/p}}{H_{2}^{2/p}}\right)+\frac{c_{3}\varepsilon_{0}2^{6+2(3p+n+2)/p}}{\tau^{n/p}}\left(\frac{H_{1}}{H_{2}}\right)^{2/p}\mf{F}(\tau_{j_{2}-1}r_{1}')\nonumber \\
&\stackrel{\eqref{ls.175}_{1}, \eqref{ls.123}}{\le}&\varepsilon_{0} A(r_{2}')\left[2\tau^{\beta_{0}}+\frac{c_{3}2^{2(3p+n+2)/p}}{H_{2}^{2/p}}+\frac{c_{3}\varepsilon_{0}2^{7+2(3p+n+2)/p}}{\tau^{n/p}}\left(\frac{H_{1}}{H_{2}}\right)^{2/p}\right]\stackrel{\eqref{ls.36.1.1},\eqref{hhh}}{<}\varepsilon_{0}A(r_{2}'),
\end{eqnarray*}
in plain contradiction with \eqref{ls.175}$_{2}$. This means that the nondegenerate regime is unstable only in $\texttt{I}^{1}_{d+1}$ or in $\texttt{K}_{d+1}$ for all $d\in \N\cup\{0\}$. Now, let $\{\texttt{B}_{d}\}_{d\in \N\cup\{0\}}$ be the blocks introduced in \textbf{Step 5}; we aim to show that they behave as independent units of the iterative process described so far. We recall that by \eqref{lms.1} it is
\eqn{lms.9}
$$\texttt{I}_{0}\subseteq (r_{2}',r_{1}']\cup(r_{1}',r_{1}].$$ If $\varsigma\in (r_{1}',r_{1}]$, there is $j_{\varsigma}\in \{0,\cdots,j_{1}-1\}$ verifying $\theta_{j_{\varsigma}+1}r_{1}<\varsigma\le \theta_{j_{\varsigma}}r_{1}$, so we can proceed as done in \emph{Step 5.4} to recover \eqref{lms.2}-\eqref{lms.3}. Next, when $\varsigma\in (r_{2}',r_{1}']$ we can find $j_{\varsigma}\in \{0,\cdots,j_{2}-1\}$ so that $\tau_{j_{\varsigma}+1}r_{1}'<\varsigma\le\tau_{j_{\varsigma}}r_{1}'$,
\begin{eqnarray}\label{lms.5}
\mf{F}(\varsigma)&\stackrel{\eqref{ls.42.1}}{\le}&\frac{2^{6}}{\tau^{n/p}}\mf{F}(\tau_{j_{\varsigma}}r_{1}')\stackrel{\eqref{ls.176}}{\le}\frac{2^{6}\tau^{j_{\varsigma}\beta_{0}}}{\tau^{n/p}}\mf{F}(r_{1}')+\frac{2^{6}c_{3}}{\tau^{n/p}}\mf{K}_{\mf{s}}^{1/(p-1)}\nonumber \\
&\stackrel{\eqref{ls.167}}{\le}&\frac{2^{6}}{\tau^{n/p+\beta_{0}}}\left(\frac{\varsigma}{r_{1}'}\right)^{\beta_{0}}\theta^{j_{1}\gamma_{0}}\mf{F}(r_{1})+\frac{2^{7}c_{3}}{\tau^{n/p}}\mf{K}_{\mf{s}}^{1/(p-1)}\nonumber \\
&=&\frac{2^{6}}{\tau^{n/p+1}}\left(\frac{\varsigma}{r_{1}}\right)^{\alpha_{0}}\mf{F}(r_{1})+\frac{2^{7}c_{3}}{\tau^{n/p}}\mf{K}_{\mf{s}}^{1/(p-1)}\stackrel{\eqref{tri.1}}{\le}\frac{2^{14+2n}}{\tau^{n/p+1}}\left(\frac{\varsigma}{\rr}\right)^{\alpha_{0}}\mf{F}(\rr)+\frac{2^{8}c_{3}}{\tau^{n/p}}\mf{K}_{\mf{s}}^{1/(p-1)}
\end{eqnarray}
and
\begin{eqnarray}\label{lms.6}
A(\varsigma)&\le&\frac{1}{\tau^{n/p}}\mf{F}(\tau_{j_{\varsigma}}r_{1}')+A(\tau_{j_{\varsigma}}r_{1}')\stackrel{\eqref{ls.176}}{\le}\frac{\tau^{j_{\varsigma}\beta_{0}}}{\tau^{n/p}}\mf{F}(r_{1}')+\frac{c_{3}}{\tau^{n/p}}\mf{K}_{\mf{s}}^{1/(p-1)}+2^{5}\nonumber \\
&\stackrel{\eqref{ls.167}}{\le}&\frac{2^{3+n/p}\hat{\varepsilon}}{\tau^{n/p}}+\frac{2c_{3}}{\tau^{n/p}}\mf{K}_{\mf{s}}^{1/(p-1)}+2^{5}\stackrel{\eqref{be},\eqref{d.3}}{\le}1+2^{5},
\end{eqnarray}
so, by \eqref{lms.9}, estimates \eqref{68.1}-\eqref{68} have been proven in particular for all $\varsigma\in \texttt{I}_{0}$. Now, keeping \eqref{lms.9} in mind, we can apply \eqref{tri.1},  \eqref{lms.2}-\eqref{lms.3} or \eqref{lms.5}-\eqref{lms.6} (depending on the position of $\nu_{\iota_{1}}r_{1}$ with respect to $r_{1}'$) to get
\begin{eqnarray}\label{lms.7}
\mf{F}(\nu_{\iota_{1}}r_{1})&\le& 2^{6+n}\mf{F}(\nu_{\iota_{1}-1}r_{1})\le \frac{2^{20+3n}}{(\tau\theta)^{n/p+1}}\left(\frac{\nu_{\iota_{1}-1}r_{1}}{\rr}\right)^{\alpha_{0}}\mf{F}(\rr)+\frac{2^{14+n}c_{3}}{(\tau\theta)^{n/p}}\mf{K}_{\mf{s}}^{1/(p-1)}\nonumber \\
&\le&\frac{2^{22+3n}}{(\tau\theta)^{n/p+1}}\left(\frac{\nu_{\iota_{1}}r_{1}}{\rr}\right)^{\alpha_{0}}\mf{F}(\rr)+\frac{2^{14+n}c_{3}}{(\tau\theta)^{n/p}}\mf{K}_{\mf{s}}^{1/(p-1)}
\end{eqnarray}
and 
\begin{eqnarray}\label{lms.8}
A(\nu_{\iota_{1}}r_{1})&\le& 2^{n}\mf{F}(\nu_{\iota_{1}-1}r_{1})+A(\nu_{\iota_{1}-1}r_{1})\nonumber \\
&\le&1+2^{5}+\frac{2^{14+2n}}{\tau^{n/p+1}}\mf{F}(\rr)+\frac{2^{8+n}c_{3}}{\tau^{n/p}}\mf{K}_{\mf{s}}^{1/(p-1)}\le 2+2^{5},
\end{eqnarray}
where we also used $\eqref{55}_{2}$, \eqref{be}, \eqref{fpm.1} and \eqref{d.3}. Therefore if $\varsigma \in \texttt{I}_{1}^{1}$, via \eqref{ls.42.1}, \eqref{lms.7}, \eqref{lms.8}, \eqref{be}, \eqref{fpm.1} and \eqref{d.3} we obtain
\begin{flalign}\label{lms.8.1}
\left\{
\begin{array}{c}
\displaystyle 
\ \mf{F}(\varsigma)\le2^{n+6}\mf{F}(\nu_{\iota_{1}}r_{1})\le \frac{2^{32+4n}}{(\tau\theta)^{n/p+1}}\left(\frac{\varsigma}{\rr}\right)^{\alpha_{0}}\mf{F}(\rr)+\frac{2^{20+2n}c_{3}}{(\tau\theta)^{n/p}}\mf{K}_{\mf{s}}^{1/(p-1)}\\[15pt]\displaystyle
\ A(\varsigma)\le 2^{n}\mf{F}(\nu_{\iota_{1}}r_{1})+A(\nu_{\iota_{1}}r_{1})\le 3+2^{5}.
\end{array}
\right.
\end{flalign}
Finally, if $\varsigma\in \texttt{K}_{1}$ we just need to recall \eqref{ls.102}, \eqref{hhh}, \eqref{fpm.1} and \eqref{d.3} to conclude. We have just proven that estimates
\begin{flalign}\label{final.2}
\left\{
\begin{array}{c}
\displaystyle 
\ \mf{F}(\varsigma)\le \frac{2^{32+4n}}{(\tau\theta)^{n/p+1}}\left(\frac{\varsigma}{\rr}\right)^{\alpha_{0}}\mf{F}(\rr)+\frac{2^{24+3p+2n}c_{3}}{\varepsilon_{0}(\tau\theta)^{n/p}}\left(\frac{H_{2}}{H_{1}}\right)^{2/p}\mf{K}_{\mf{s}}^{1/(p-1)}\\[15pt]\displaystyle
\ A(\varsigma)\le 4+2^{5}(M+1).
\end{array}
\right.
\end{flalign}
hold for all $\varsigma\in \texttt{B}_{0}$. We stress that the presence of $M$ in the above display accounts for the case in which iterations start from $\eqref{ls.112}_{2}$ rather than \eqref{ls.112}$_{1}$ - the procedure leading to \eqref{final.2} remains identical up to replace everywhere $r_{1}'$ by $r_{1}$, $r_{1}''$ by $\tau_{1}r_{1}$ and \eqref{ls.152} by \eqref{a.10}$_{2}$; see also Remark \ref{stin}. It is important to notice that in the light of \eqref{lms.1}, the fact that $r_{2}'$ belongs to $\texttt{B}_{0}$ or not is totally irrelevant as decay estimates are available on $\texttt{I}_{0}$, cf. \eqref{lms.9} and \eqref{lms.2}-\eqref{lms.3} or \eqref{lms.5}-\eqref{lms.6}; $\texttt{I}_{0}$ and $\texttt{I}_{1}^{1}$ differ only by one scale and estimates can be transferred as we did in the above displays, while by means of \eqref{ls.102} in $\texttt{K}_{1}$ there is essentially nothing to prove. Next, given a general block $\texttt{B}_{d}$ with $d\in \N$, it is easy to see that the procedure described in \emph{Step 5.1}-\emph{Step 5.5} can be entirely reproduced on $\texttt{B}_{d}$. In fact, if $\texttt{I}_{d}^{2}=\{\emptyset\}$ there is nothing to prove as $\texttt{B}_{d}\equiv \texttt{I}^{1}_{d+1}\cup \texttt{K}_{d+1}$ on which \eqref{lms.10} holds; otherwise set $r_{d}:=\nu_{\iota_{d}+\kk_{d}+1}r_{1}$ and notice that
\begin{flalign}
\frac{H_{2}}{\varepsilon_{0}^{p/2}}\mf{K}_{\mf{s}}^{\frac{p}{2(p-1)}}\stackrel{\eqref{ls.102.1.1}}{\le}\mf{C}(r_{d})\stackrel{\eqref{ls.134}}{\le}2^{3p+n+2}\mf{C}(\nu_{\iota_{d}+\kk_{d}}r_{1})\stackrel{\eqref{ls.101}_{2}}{\le}\frac{2^{3p+n+2}H_{2}}{\varepsilon_{0}^{p/2}}\mf{K}_{\mf{s}}^{\frac{p}{2(p-1)}}\stackrel{\eqref{fpm.1}}{\le}\hat{\varepsilon}^{p/2},\label{lms.12}
\end{flalign}
which via \eqref{be} yields that
\begin{eqnarray}\label{lms.13}
A(r_{d})\le 1\qquad \mbox{and}\qquad \mf{F}(r_{d})<\frac{(\tau\theta)^{4npq}\varepsilon_{2}}{2^{6nq}}<\varepsilon_{2}.
\end{eqnarray}
This means that we can repeat the whole construction laid down in \emph{Step 5.1}-\emph{Step 5.5} with $r_{d}$ replacing $r_{1}$ and $r_{1}'$, $r_{2}'$ this time defined as $r_{1}':=\theta_{j_{1}}r_{d}$ and $r_{2}':=\tau_{j_{2}}\theta_{j_{1}}r_{d}$ (here the values of $j_{1},j_{2}\ge 1$ may differ from those considered in the aforementioned steps) to obtain in case of stability of either the degenerate regime or of the nondegenerate regime (i.e. when $\mathcal{J}_{1}\equiv \N\cup\{0\}$ or $\mathcal{J}_{2}\equiv \N\cup\{0\}$ respectively):
\begin{eqnarray}\label{lms.17}
\mf{F}(\varsigma)&\le&\frac{2^{14+2n}}{(\theta\tau)^{n/p+1}}\left(\frac{\varsigma}{r_{d}}\right)^{\alpha_{0}}\mf{F}(r_{d})+\frac{2^{8}c_{3}}{(\tau\theta)^{n/p}}\mf{K}_{\mf{s}}^{1/(p-1)}\nonumber \\
&\stackrel{\eqref{lms.12}}{\le}&\left[\frac{2^{16+3p+3n}}{\varepsilon_{0}(\tau\theta)^{n/p+1}}\left(\frac{H_{2}}{H_{1}}\right)^{2/p}+\frac{2^{8}c_{3}}{(\tau\theta)^{n/p}}\right]\mf{K}_{\mf{s}}^{1/(p-1)}\le\frac{2^{18+3n+3p}c_{3}}{\varepsilon_{0}(\tau\theta)^{n/p+1}}\left(\frac{H_{2}}{H_{1}}\right)^{2/p}\mf{K}_{\mf{s}}^{1/(p-1)}
\end{eqnarray}
and, combining \eqref{lms.12}-\eqref{lms.13} with \eqref{lms.14}, \eqref{lms.15}, \eqref{lms.3} (of course with $r_{d}$ instead of $r_{1}$), we get
\begin{eqnarray}\label{lms.18}
A(\varsigma)&\le&2^{5}+1.
\end{eqnarray}
Both \eqref{lms.17} and \eqref{lms.18} hold for all $\varsigma\in (0,r_{d}]$. Finally, if the nondegenerate regime is unstable, we observe that \eqref{lms.1} is true also for the newly defined $r_{2}'$, which in particular cannot belong to $\texttt{I}_{d}^{2}$ and so $\texttt{I}_{d}^{2}\subseteq (r_{2}',r_{1}']\cup(r_{1}',r_{d}]$, therefore proceeding as done for deriving \eqref{lms.2}-\eqref{lms.3} and \eqref{lms.5}-\eqref{lms.6}\footnote{We do not actually need the full estimates \eqref{lms.2}-\eqref{lms.3} and \eqref{lms.5}-\eqref{lms.6}, but just their content till the line in which $r_{1}$ - replaced by $r_{d}$ - appears.} we confirm the validity of \eqref{lms.17}-\eqref{lms.18} for all $\varsigma\in (r_{2}',r_{1}']\cup(r_{1}',r_{d}]$ and a fortiori for all $\varsigma\in \texttt{I}_{d}^{2}$. At this stage, we only need to transfer estimates \eqref{lms.17}-\eqref{lms.18} to interval $\texttt{I}^{1}_{d+1}$, but this is straightforward. Indeed we have
\begin{flalign}
\mf{F}(\nu_{\iota_{d+1}}r_{1})\stackrel{\eqref{tri.1}}{\le}2^{6+n}\mf{F}(\nu_{\iota_{d+1}-1}r_{1})\stackrel{\eqref{lms.17}}{\le}\frac{2^{24+3p+4n}c_{3}}{\varepsilon_{0}(\tau\theta)^{n/p+1}}\left(\frac{H_{2}}{H_{1}}\right)^{2/p}\mf{K}_{\mf{s}}^{1/(p-1)}\label{lms.19}
\end{flalign}
and
\begin{eqnarray}\label{lms.20}
 A(\nu_{\iota_{d+1}}r_{1})&\le&2^{n}\mf{F}(\nu_{\iota_{d+1}-1}r_{1})+A(\nu_{\iota_{d+1}-1}r_{1})\nonumber \\
 &\stackrel{\eqref{lms.17},\eqref{lms.18}}{\le}&\frac{c_{3}2^{18+3p+4n}}{\varepsilon_{0}(\tau\theta)^{n/p+1}}\left(\frac{H_{2}}{H_{1}}\right)^{2/p}\mf{K}_{\mf{s}}^{1/(p-1)}+1+2^{5}\stackrel{\eqref{fpm.1},\eqref{d.3}}{\le}2+2^{5},
\end{eqnarray}
therefore if $\varsigma\in \texttt{I}^{1}_{d+1}$ i.e., $\nu_{\iota_{d+1}+1}r_{1}<\varsigma\le \nu_{\iota_{d+1}}r_{1}$, by means of \eqref{ls.42.1}, \eqref{lms.19}, \eqref{lms.20}, \eqref{fpm.1} and \eqref{d.3} it is
\begin{eqnarray}\label{lms.21}
\mf{F}(\varsigma)\le \frac{c_{3}2^{30+3p+5n}}{\varepsilon_{0}(\tau\theta)^{n/p+1}}\left(\frac{H_{2}}{H_{1}}\right)^{2/p}\mf{K}_{\mf{s}}^{1/(p-1)}
\end{eqnarray}
and
\begin{flalign}
A(\varsigma)\stackrel{\eqref{lms.20}}{\le}2^{6+n}\mf{F}(\nu_{\iota_{d+1}}r_{1})+2+2^{5}\stackrel{\eqref{lms.19}}{\le}\frac{2^{30+3p+5n}c_{3}}{\varepsilon_{0}(\tau\theta)^{n/p+1}}\left(\frac{H_{2}}{H_{1}}\right)^{2/p}\mf{K}_{\mf{s}}^{1/(p-1)}+2+2^{5}\le 3+2^{5}.\label{lms.22}
\end{flalign}
Moreover, if $\varsigma\in \texttt{K}_{d+1}$ we can directly conclude with \eqref{ls.102}, thus whenever $\varsigma \in \texttt{B}_{d}$ it is
\begin{eqnarray}\label{final.3}
\left\{
\begin{array}{c}
\displaystyle 
\ \mf{F}(\varsigma)\le \frac{c_{3}2^{30+3p+5n}}{\varepsilon_{0}(\tau\theta)^{n/p+1}}\left(\frac{H_{2}}{H_{1}}\right)^{2/p}\mf{K}_{\mf{s}}^{1/(p-1)}\\[15pt]\displaystyle
\ A(\varsigma)\le 4+2^{5}.
\end{array}
\right.
\end{eqnarray}
Finally, if there is only a finite number of finite iteration chains, say $e_{*}\in \N$, we find $\{\iota_{1},\cdots,\iota_{e_{*}}\}\subset \N$, $\{\kk_{1},\cdots,\kk_{e_{*}}\}\subset \N$ with $\iota_{1}<\cdots<\iota_{e_{*}}$ determining chains $\{\mathcal{C}_{\iota_{d}}^{\kk_{d}}\}_{d\in \{1,\cdots,e_{*}\}}$. We then consider blocks $\{\texttt{B}_{d}\}_{d\in \{0,\cdots,e_{*}-1\}}$. On all blocks, the previous results apply so \eqref{final.2}-\eqref{final.3} are verified. Since there is only a finite number of iteration chains, we have that $\{j\in \N\colon j\ge \iota_{e_{*}}+\kk_{e_{*}}+1\}\subset \mathcal{J}_{0}$, thus for all $\varsigma\in (0,r_{1}]\setminus \bigcup_{d\in \{0,\cdots,e_{*}-1\}}\texttt{B}_{d}\equiv (0,\nu_{\iota_{e_{*}}+\kk_{e_{*}}+1}r_{1}]$ we can find $j_{\varsigma}\ge\iota_{e_{*}}+\kk_{e_{*}}+1 $ so that $\nu_{j_{\varsigma}+1}r_{1}<\varsigma\le \nu_{j_{\varsigma}}r_{1}$ and 
\begin{eqnarray}\label{lms.35}
\mf{C}(\varsigma)\stackrel{\eqref{ls.134}}{\ge}\frac{H_{2}}{2^{3p+n+2}\varepsilon_{0}^{p/2}}\mf{K}_{\mf{s}}^{\frac{p}{2(p-1)}}.
\end{eqnarray}
Recalling that $\iota_{e_{*}}+\kk_{e_{*}}\not \in \mathcal{J}_{0}$, it is
\begin{flalign}
\frac{H_{2}}{\varepsilon_{0}^{p/2}}\mf{K}_{\mf{s}}^{\frac{p}{2(p-1)}}<\mf{C}(\nu_{\iota_{e_{*}}+\kk_{e_{*}}+1}r_{1})\stackrel{\eqref{ls.134}}{\le}2^{3p+n+2}\mf{C}(\nu_{\iota_{e_{*}}+\kk_{e_{*}}}r_{1})\stackrel{\eqref{ls.101}}{\le} \frac{2^{3p+n+2}H_{2}}{\varepsilon_{0}^{p/2}}\mf{K}_{\mf{s}}^{\frac{p}{2(p-1)}},\label{namehere}
\end{flalign}
therefore via \eqref{fpm.1} and \eqref{be} we can prove the validity of \eqref{lms.13}. Furthermore, \eqref{lms.35} assures the applicability of the same contradiction argument described in \emph{Step 5.5} eventually leading to \eqref{lms.1}, which in turn yields that the nondegenerate regime is stable in $(0,\nu_{\iota_{e_{*}}+\kk_{e_{*}}+1}r_{1}]$. This means that we can set $r_{d}:=\nu_{\iota_{e_{*}}}+\kk_{\iota_{e_{*}}+1}r_{1}$ and apply the estimates derived in \emph{Step 5.2}-\emph{Step 5.4} with $r_{d}$ instead of $r_{1}$ depending on which regime holds\footnote{It can start with the degenerate regime and then switch to the nondegenerate one, which is stable or it directly starts with the nondegenerate regime, stable over the whole interval $(0,\nu_{\iota_{e_{*}}+\kk_{e_{*}}+1}r_{1}]$.}, to conclude with the validity of \eqref{final.3} on $(0,\nu_{\iota_{e_{*}}+\kk_{e_{*}}+1}r_{1}]$ thus \eqref{final.2}-\eqref{final.3} hold on the whole interval $(0,r_{1}]$. Finally, if $\varsigma\in (r_{1},\rr]$, via standard argument we recover \eqref{lms.4}
and \eqref{final.2} is again confirmed.
\subsubsection*{Step 5.6: an infinite iteration chain} The occurrence of an infinite iteration chain can be described by introducing a number $e_{*}\in \N$ and a finite set of integers $\{\iota_{1},\cdots,\iota_{e_{*}}\}\subset \N$ and $\{\kk_{1},\cdots,\kk_{e_{*}}\}$ with $\{\kk_{1},\cdots,\kk_{e_{*}-1}\}\subset \N$ and $\kk_{e_{*}}\equiv \infty$ which determine a finite number of finite iteration chains $\{\mathcal{C}^{\iota_{d}}_{\kk_{d}}\}_{d\in \{1,\cdots,e_{*}-1\}}$ and an infinite iteration chain $\mathcal{C}^{\infty}_{\iota_{e_{*}}}$, that must be unique by maximality. Assume first that $e_{*}\ge 2$. For $d\in \{1,\cdots,e_{*}\}$ and with the same notation introduced at the beginning of \textbf{Step 5}, we can determine the intervals $\texttt{I}_{0}$, $\texttt{I}^{1}_{d}$, $\texttt{K}_{d}$ and $\texttt{I}^{2}_{d}$  with $\texttt{K}_{e_{*}}=(0,\nu_{\iota_{e_{*}}+1}r_{1}]$ and the corresponding blocks $\{\texttt{B}_{d}\}_{d\in \{0,\cdots,e_{*}-1\}}$ with $\texttt{B}_{e_{*}-1}= (0,\nu_{\iota_{e_{*}-1}+\kk_{e_{*}-1}+1}r_{1}]$. Blocks $\texttt{B}_{0},\cdots, \texttt{B}_{e_{*}-2}$ can be treated as done in \emph{Step 5.5} with resulting estimates \eqref{final.2} and \eqref{final.3}, while concerning the first two components $\texttt{I}^{2}_{e_{*}-1}$-$\texttt{I}^{1}_{e_{*}}$ of the final block $\texttt{B}_{e_{*}-1}$, we see that either $\texttt{I}^{2}_{e_{*}-1}=\{\emptyset\}$ and so we can conclude by means of \eqref{lms.10} or $\texttt{I}^{2}_{e_{*}-1}\not =\{\emptyset\}$, \eqref{tri.1} and \eqref{ls.101} validate estimates \eqref{lms.12}-\eqref{lms.13} (with $d=e_{*}-1$) and, as in \emph{Step 5.5}, \eqref{final.3} follows. On $\texttt{K}_{e_{*}}$ we see that \eqref{ls.102} holds for all $\varsigma\in (0,\nu_{\iota_{e_{*}}+1}r_{1}]\equiv \texttt{K}_{e_{*}}$ and this immediately leads to the conclusion. Finally, if $e_{*}=1$ there is only the infinite iteration chain $\mathcal{C}^{\infty}_{\iota_{1}}$. Indeed the only block to be considered is $\texttt{B}_{0}$ with $\texttt{K}_{1}=(0,\nu_{\iota_{1}+1}r_{1}]$. By means of \eqref{a.10}, on $\texttt{I}_{0}$-$\texttt{I}_{1}^{1}$, the estimates in \eqref{final.2} hold true, while \eqref{ls.102} covers all $\varsigma\in \texttt{K}_{1}$, and the validity of \eqref{final.2} and \eqref{final.3} is confirmed again.
\subsubsection*{Step 5.7: occurrence \emph{(}ii.\emph{)}} To complete the analysis of the decay of the excess functional $\mf{F}(\cdot)$ and of the controlled boundedness of the average in case $\eqref{ls.100}_{1}$ is in force, we only need to consider occurrence (\emph{ii}.) at the beginning of \textbf{Step 5}. Since $\mathcal{J}_{0}=\N\cup\{0\}$, the lower bound in \eqref{ls.102.1} holds for all $\varsigma\in (0,r_{1}]$, therefore we have three possibilities: $\eqref{ls.112}_{1}$ is in force, the degenerate regime is stable and \eqref{final.0} holds true, cf. \emph{Step 5.2}; or with $\eqref{ls.112}_{1}$ still valid there is a first exit time $r_{1}'$, see \emph{Step 5.3} but the validity of \eqref{ls.102.1} for all $\varsigma\in (0,r_{1}]$ assures that there cannot be a second exit time $r_{2}'$ as the contradiction argument leading to \eqref{lms.1} can now be extended to the full interval $(0,r_{1}]$, therefore the nondegenerate regime is stable and \eqref{final.1} holds, cf. \emph{Step 5.4}; or if $\eqref{ls.100}_{2}$ is satisfied, the nondegenerate regime is directly in force and, as before, the validity of \eqref{ls.102.1} for all $\varsigma\in (0,r_{1}]$ guarantees its stability thus yielding again \eqref{final.1}. In any case, \eqref{final.0} or \eqref{final.1} hold true.
\subsection*{Step 6: starting from small composite excess - $\eqref{ls.100}_{2}$ holds} We define the set $\mathcal{J}_{0}$ as
\begin{eqnarray*}
\mathcal{J}_{0}:=\left\{j\in \N\cup\{0\}\colon \mf{C}(\nu_{j}r_{1})\le \frac{H_{2}}{\varepsilon_{0}^{p/2}}\mf{K}_{\mf{s}}^{\frac{p}{2(p-1)}}\right\},
\end{eqnarray*}
which is nonempty by means of $\eqref{ls.100}_{2}$. If $\mathcal{J}_{0}\equiv \N\cup\{0\}$, by \eqref{ls.101}$_{2}$, that in this case holds for all $j\in \N\cup\{0\}$, we obtain \eqref{ls.102} for all $\varsigma\in (0,r_{1}]$, which, together with \eqref{55} and \eqref{ls.42.1} yields that
\begin{eqnarray}\label{final.4}
\left\{
\begin{array}{c}
\displaystyle
\ \mf{F}(\varsigma)\le 2^{n+8}\left(\frac{\varsigma}{\rr}\right)^{\alpha_{0}}\mf{F}(\rr)+\frac{2^{3p+n+2}}{\varepsilon_{0}}\left(\frac{H_{2}}{H_{1}}\right)^{2/p}\mf{K}_{\mf{s}}^{1/(p-1)}\\[15pt]\displaystyle 
\ A(\varsigma)\le M+1.
\end{array}
\right.
\end{eqnarray}
Next, we assume the existence of infinitely many finite iteration chains $\mathcal{C}^{\kk_{d}}_{\iota_{d}}$ corresponding to discrete sequence $\{\iota_{d}\}_{d\in \N}, \{\kk_{d}\}_{d\in \N}\subset \N$ for which \eqref{ls.104} holds and, as $d\in \N$ consider the disjoint intervals $\texttt{I}_{0}$-$\texttt{I}^{1}_{d}$-$\texttt{K}_{d}$-$\texttt{I}^{2}_{d}$ defined at the beginning of \textbf{Step 5} and the related blocks $\{\texttt{B}_{d}\}_{d\in \N\cup\{0\}}$. Let us examine what happens on a generic block $\texttt{B}_{d}$. Now, for any $\varsigma \in \texttt{I}_{0}$ or $\varsigma\in \texttt{I}^{2}_{d}$ we can find $j_{\varsigma}\in \{0,\cdots,\iota_{1}-1\}\subset \mathcal{J}_{0}$ or $j_{\varsigma}\in \{\iota_{d}+\kk_{d}+1,\cdots,\iota_{d+1}-1\}\subset \mathcal{J}_{0}$ so that $\nu_{j_{\varsigma}+1}r_{1}<\varsigma\le \nu_{j_{\varsigma}}r_{1}$ and
\begin{eqnarray}\label{lms.30}
\mf{C}(\varsigma)\stackrel{\eqref{ls.134}}{\le}2^{3p+n+3}\mf{C}(\nu_{j_{\varsigma}}r_{1})\le \frac{2^{3p+n+2}H_{2}}{\varepsilon_{0}^{p/2}}\mf{K}_{\mf{s}}^{\frac{p}{2(p-1)}}.
\end{eqnarray}
Similarly, for $\varsigma\in \texttt{I}^{1}_{d+1}$, i.e. $\nu_{\iota_{d+1}+1}r_{1}<\varsigma\le \nu_{\iota_{d+1}}r_{1}$ we can use \eqref{ls.134} and the fact that $\iota_{d+1}\in \mathcal{J}_{0}$ for all $d\in \N\cup\{0\}$ to conclude again with \eqref{lms.30}. Notice that in contrast with what happens in \textbf{Step 5}, now it does not matter if $\texttt{I}^{2}_{d}$ is empty or not, the fact that $\iota_{d+1}\in \mathcal{J}_{0}$ assures the validity of \eqref{lms.30} for all $\varsigma\in \texttt{I}_{d+1}^{1}$ and this is enough to proceed by checking what happens in $\texttt{K}_{d+1}$. To this aim, let us consider the disjoint union $\texttt{K}_{d+1}=\texttt{K}_{d+1}^{1}\cup\texttt{K}_{d+1}^{2}$, where it is $\texttt{K}_{d+1}^{1}:=(\nu_{\iota_{d+1}+\kk_{d+1}}r_{1},\nu_{\iota_{d+1}+1}r_{1}]$ and $\texttt{K}_{d+1}^{2}:=(\nu_{\iota_{d+1}+\kk_{d+1}+1}r_{1},\nu_{\iota_{d+1}+\kk_{d+1}}r_{1}]$. If $\varsigma\in \texttt{K}^{1}_{d+1}$, there is $j_{\varsigma}\in \{\iota_{d+1}+1,\cdots,\iota_{d+1}+\kk_{d+1}-1\} \subset \mathcal{C}_{\kk_{d+1}}^{\iota_{d+1}}$ so that $\nu_{j_{\varsigma}+1}r_{1}<\varsigma\le \nu_{j_{\varsigma}}r_{1}$,
\begin{eqnarray}\label{lms.31}
\mf{C}(\varsigma)\stackrel{\eqref{ls.134}}{\ge}\frac{1}{2^{3p+n+2}}\mf{C}(\nu_{j_{\varsigma}+1}r_{1})>\frac{H_{2}}{2^{3p+n+2}\varepsilon_{0}^{p/2}}\mf{K}_{\mf{s}}^{\frac{p}{2(p-1)}}
\end{eqnarray}
and, since $\iota_{d+1}\in \mathcal{J}_{0}$ and $\iota_{d+1}+1\not \in \mathcal{J}_{0}$ is it
\begin{eqnarray}\label{lms.32}
\frac{H_{2}}{\varepsilon_{0}^{p/2}}\mf{K}_{\mf{s}}^{\frac{p}{2(p-1)}}<\mf{C}(\nu_{\iota_{d+1}+1}r_{1})\stackrel{\eqref{ls.134}}{\le}2^{3p+n+2}\mf{C}(\nu_{\iota_{d+1}}r_{1})\le \frac{2^{3p+n+2}H_{2}}{\varepsilon_{0}^{p/2}}\mf{K}_{\mf{s}}^{\frac{p}{2(p-1)}}. 
\end{eqnarray}
It may be that $\texttt{K}_{d+1}^{1}=\left\{\emptyset\right\}$, i.e. $\kk_{d+1}=1$, but in this case we can combine \eqref{lms.30} with \eqref{ls.134} to show that if $\varsigma\in \texttt{B}_{0}$ and $\texttt{K}^{1}_{1}=\left\{\emptyset\right\}$ or $\varsigma\in \texttt{B}_{d}$ with $\texttt{K}^{1}_{d+1}=\left\{\emptyset\right\}$ it is $\mf{C}(\varsigma)\le 2^{6p+2n+4}H_{2}\varepsilon_{0}^{-p/2}\mf{K}_{\mf{s}}^{\frac{p}{2(p-1)}}$; thus if $\texttt{K}_{d+1}^{1}=\left\{\emptyset\right\}$ for some $d\in \N\cup\{0\}$, there is nothing to prove on the corresponding block $\texttt{B}_{d}$. By \eqref{lms.32}, \eqref{fpm.1}, \eqref{d.3} and \eqref{be}, we deduce that \eqref{lms.13} holds in this case as well with $r_{d}:=\nu_{\iota_{d+1}+1}r_{1}$, which means that we can repeat the procedure described in \emph{Step 5.5}, from display \eqref{lms.17} to inequality \eqref{lms.20}; of course now $\texttt{K}_{d+1}^{1}$ replaces $\texttt{I}^{2}_{d}$. This in particular assures the validity of estimates \eqref{lms.17}-\eqref{lms.18} and \eqref{lms.19}-\eqref{lms.20} (with $\nu_{\iota_{d+1}+\kk_{d+1}}r_{1}$ instead of $\nu_{\iota_{d+1}}r_{1}$). Finally, via \eqref{lms.19}-\eqref{lms.20} and \eqref{ls.42.1} we can transfer decay estimates to all $\varsigma\in \texttt{K}_{d+1}^{2}$, thus recovering the estimates in \eqref{final.3}. When there is only a finite number of finite iteration chains, say $e_{*}\in \N$, we can find finite subsets $\{\iota_{1},\cdots,\iota_{e_{*}}\},\{\kk_{1},\cdots,\kk_{e_{*}}\}\subset \N$ with $\iota_{1}<\cdots<\iota_{e_{*}}$ determining chains $\mathcal{C}^{\kk_{d}}_{\iota_{d}}$. For $d\in \{1,\cdots,e_{*}\}$, we consider intervals $\texttt{I}_{0}$-$\texttt{I}_{d}^{1}$-$\texttt{K}_{d}^{1}$-$\texttt{K}_{d}^{2}$-$\texttt{I}_{d}^{2}$ and blocks $\{\texttt{B}_{d}\}_{d\in \{0,\cdots,e_{*}-1\}}$. On each of such blocks, the results obtained above apply and \eqref{final.3} holds true. The fact that there is only a finite number of chains yields that $\{j\in \N\colon j\ge \iota_{e_{*}}+\kk_{e_{*}}+1\}\subset \mathcal{J}_{0},$ therefore, whenever $\varsigma\in (0,r_{1}]\setminus \bigcup_{d\in \{0,\cdots,e_{*}-1\}}\texttt{B}_{d}\equiv (0,\nu_{\iota_{e_{*}}+\kk_{e_{*}}+1}r_{1}]$ there is $j_{\varsigma}\ge \iota_{e_{*}}+\kk_{e_{*}}+1$ so that $\nu_{j_{\varsigma}+1}r_{1}<\varsigma\le \nu_{j_{\varsigma}}r_{1}$ and $\mf{C}(\varsigma)\le 2^{3p+n+2}\varepsilon_{0}^{-p/2}H_{2}\mf{K}_{\mf{s}}^{\frac{p}{2(p-1)}},$
by means of \eqref{ls.134}. 
\subsubsection*{Step 6.1: an infinite iteration chain} As done in \emph{Step 5.6}, we introduce a number $e_{*}\in \N$ and a finite set of integers $\{\iota_{1},\cdots,\iota_{e_{*}}\}\subset \N$ and $\{\kk_{1},\cdots,\kk_{e_{*}}\}$ with $\{\kk_{1},\cdots,\kk_{e_{*}-1}\}\subset \N$ and $\kk_{e_{*}}\equiv \infty$ which determine a finite number of finite iteration chains $\{\mathcal{C}^{\iota_{d}}_{\kk_{d}}\}_{d\in \{1,\cdots,e_{*}-1\}}$ and an infinite iteration chain $\mathcal{C}^{\infty}_{\iota_{e_{*}}}$, unique by maximality. Assume first that $e_{*}\ge 2$. We determine intervals $\texttt{I}_{0}$, $\texttt{I}^{1}_{d}$, $\texttt{K}_{d}^{1}$, $\texttt{K}^{2}_{d}$, $\texttt{I}^{2}_{d}$ and blocks $\{\texttt{B}_{d}\}_{d\in \{0,\cdots,e_{*}-1\}}$ with $\texttt{K}_{e_{*}}=(0,\nu_{\iota_{e_{*}}+1}r_{1}]$ and $\texttt{B}_{e_{*}-1}=(0,\nu_{\iota_{e_{*}-1}+\kk_{e_{*}-1}+1}r_{1}]$. On blocks $\texttt{B}_{0}$-$\texttt{B}_{e_{*}-2}$ the results obtained in the first part of \textbf{Step 6} apply and \eqref{final.3} holds true. Concerning the first two components $\texttt{I}_{e_{*}-1}^{2}$-$\texttt{I}_{e_{*}}^{1}$ of the final block $\texttt{B}_{e_{*}-1}$, regardless on whether $\texttt{I}^{2}_{e_{*}-1}$ is empty or not, we have that \eqref{lms.30} is valid. Next, looking at $\texttt{K}_{e_{*}}$, we notice that by the definition of $\mathcal{C}_{\iota_{e_{*}}}^{\infty}$ it is $\iota_{e_{*}}\in \mathcal{J}_{0}$ and if $j\ge \iota_{e_{*}}+1$ then $j\not \in \mathcal{J}_{0}$. Setting this time $r_{d}:=\nu_{\iota_{e_{*}}+1}r_{1}$ we get that
\begin{eqnarray}\label{ls.34}
\left\{
\begin{array}{c}
\displaystyle 
\ \mf{C}(r_{d})\stackrel{\eqref{ls.134}}{\le}2^{3p+n+2}\mf{C}(\nu_{\iota_{e_{*}}}r_{1})\le \frac{2^{3p+n+2}H_{2}}{\varepsilon_{0}^{p/2}}\mf{K}_{\mf{s}}^{\frac{p}{2(p-1)}}\\[15pt]\displaystyle
\ \frac{H_{2}}{2^{3p+n+2}\varepsilon_{0}^{p/2}}\mf{K}_{\mf{s}}^{\frac{p}{2(p-1)}}\stackrel{\eqref{ls.134}}{<}\mf{C}(\varsigma)\quad \mbox{for all} \ \ \varsigma\in (0,\nu_{\iota_{e_{*}}+1}r_{1}],
\end{array}
\right.
\end{eqnarray}
therefore, by means of \eqref{fpm.1}, \eqref{d.3} and \eqref{be} we obtain that \eqref{lms.13} holds. Moreover, as already observed in \emph{Step 5.5}, $\eqref{ls.34}_{2}$ guarantees that the nondegenerate regime is stable on $(0,r_{d}]$, so we can proceed as in \emph{Step 5.2}-\emph{Step 5.4} to conclude with 
\begin{eqnarray}\label{nina}
\left\{
\begin{array}{c}
\displaystyle 
\ \mf{F}(\varsigma)\le \frac{2^{3p+n+12}c_{3}}{(\tau\theta)^{n/p+1}\varepsilon_{0}}\left(\frac{H_{2}}{H_{1}}\right)^{2/p}\mf{K}_{\mf{s}}^{1/(p-1)}\\[15pt]\displaystyle
\ A(\varsigma)\le 1+2^{5}.
\end{array}
\right.
\end{eqnarray}
Finally, if $e_{*}=1$ there is only the infinite iteration chain $\mathcal{C}^{\infty}_{\iota_{1}}$. Indeed the only block to be considered is $\texttt{B}_{0}$ with $\texttt{K}_{1}=(0,\nu_{\iota_{1}+1}r_{1}]$. For all $\varsigma\in \texttt{I}_{0}\cup \texttt{I}_{1}^{1}$, we have that $\mf{C}(\varsigma)\le 2^{3p+n+2}\varepsilon_{0}^{-p/2}H_{2}\mf{K}_{\mf{s}}^{\frac{p}{2(p-1)}}$
and, since $\iota_{1}\in \mathcal{J}_{0}$ and $j\ge \iota_{1}+1 \ \Longrightarrow j\not \in \mathcal{J}_{0}$ we get \eqref{ls.34} with $e_{*}=1$, which means that we can recover \eqref{nina} exactly as explained for the general case $e_{*}\ge 2$. To summarize, whenever starting from small composite excess,
\begin{flalign}\label{final.5}
\left\{
\begin{array}{c}
\displaystyle 
\ \mf{F}(\varsigma)\le  2^{n+8}\left(\frac{\varsigma}{\rr}\right)^{\alpha_{0}}\mf{F}(\rr)+ \frac{c_{3}2^{30+3p+5n}}{\varepsilon_{0}(\tau\theta)^{n/p+1}}\left(\frac{H_{2}}{H_{1}}\right)^{2/p}\mf{K}_{\mf{s}}^{1/(p-1)}\\[15pt]\displaystyle
\ A(\varsigma)\le 4+2^{5}(M+1)
\end{array}
\right.
\end{flalign}
holds true for all $\varsigma\in (0,\rr]$.
\subsection*{Step 7: conclusions} Setting
\eqn{c4c5}
$$
c_{4}:=\frac{2^{32+4n}}{(\tau\theta)^{n/p+1}},\qquad \qquad \quad c_{5}:=\frac{c_{3}2^{3p+30+5n}}{\varepsilon_{0}(\tau\theta)^{n/p+1}}\left(\frac{H_{2}}{H_{1}}\right)^{2/p},
$$
thus fixing dependencies $c_{4},c_{5}\equiv c_{4},c_{5}(\textnormal{\texttt{data}}_{\textnormal{c}},M^{q-p})$ and collecting estimates \eqref{final.0}, \eqref{final.1}, \eqref{final.2}, \eqref{final.3}, \eqref{final.4}, \eqref{final.5} we obtain \eqref{68.1}-\eqref{68} and the proof is complete.

 \end{proof}
 
Now we want to see if the content of Theorem \ref{t.ex} can be replicated on smaller scales. To do so, we need to assume that
\begin{eqnarray}\label{fpm}
\mathbf{I}_{1,m}^{f}(x,\sigma)\to 0\qquad \mbox{locally uniformly in} \ \ x\in \Omega.
\end{eqnarray}
We then have
\begin{proposition}\label{ret}
Under assumptions \eqref{assf}-\eqref{p0}, \eqref{f} and \eqref{fpm}, let $u\in W^{1,p}(\Omega,\mathbb{R}^{N})$ be a local minimizer of \eqref{exfun}, $x_{0}\in \mathcal{R}_{u}$ be a point and $M\equiv M(x_{0})$ be the positive constant in \eqref{ru.0}. If $\hat{\varepsilon}\equiv \hat{\varepsilon}(\textnormal{\texttt{data}}_{\textnormal{c}},M^{q-p})$ is the small parameter in \eqref{be}, $\hat{\rr}\equiv \hat{\rr}(\textnormal{\texttt{data}}_{\textnormal{c}},M^{q-p},f(\cdot))$ is the threshold radius determined by \eqref{fpm.1} and $\bar{\varepsilon}\equiv \hat{\varepsilon}$, $\bar{\rr}\equiv \hat{\rr}$ in \eqref{ru.0}, there is an open neighborhood $B(x_{0})\subset \mathcal{R}_{u}$ and a positive radius $\rr_{x_{0}}\equiv \rr_{x_{0}}(\textnormal{\texttt{data}}_{\textnormal{c}},M^{q-p},f(\cdot))$ such that
\begin{flalign}\label{ret.1}
\mf{F}(u;B_{s}(x))\le c_{6}\left(\frac{s}{\varsigma}\right)^{\alpha_{0}}\mf{F}(u;B_{\varsigma}(x))+c_{5}\sup_{\sigma\le \varsigma/4}\mf{K}\left[\left(\sigma^{m}\mint_{B_{\sigma}(x)}\snr{f}^{m} \ \dx\right)^{1/m}\right]^{1/(p-1)}
\end{flalign}
and
\begin{eqnarray}\label{ret.1.1}
\snr{(Du)_{B_{\varsigma}(x)}}<2^{6}(M+1)
\end{eqnarray}
hold for any $x\in B(x_{0})$ and all $0<s\le \varsigma\le \rr_{x_{0}}$, with $c_{6}\equiv c_{6}(\textnormal{\texttt{data}}_{\textnormal{c}},M^{q-p})$ and $c_{5}$ as in \eqref{c4c5}.
\end{proposition}
\begin{proof}
The arguments presented in Section \ref{rs} yield that $\mathcal{R}_{u}$ is an open set and, via the absolute continuity of Lebesgue integral, also that if $x_{0}\in \mathcal{R}_{u}$, there is an open neighborhood $B(x_{0})\subset \mathcal{R}_{u}$ for which \eqref{70} holds with $\bar{\varepsilon}\equiv\hat{\varepsilon}$ and, via \eqref{fpm}, \eqref{fpm.1} is uniformly verified in $B(x_{0})$. This means that Theorem \ref{t.ex} applies and estimates
\begin{flalign}\label{ret.2}
\mf{F}(u;B_{\varsigma}(x))\le c_{4}\left(\frac{\varsigma}{\rr_{x_{0}}}\right)^{\alpha_{0}}\mf{F}(u;B_{\rr_{x_{0}}}(x))+c_{5}\sup_{\sigma\le \rr_{x_{0}}/4}\mf{K}\left[\left(\sigma^{m}\mint_{B_{\sigma}(x)}\snr{f}^{m} \dx\right)^{1/m}\right]^{1/(p-1)}
\end{flalign}
and $\snr{(Du)_{B_{\varsigma}(x)}}<2^{6}(M+1)$ hold true for all $x\in B(x_{0})$ so \eqref{ret.1.1} is proven. Next, we define $r_{x_{0}}:=(2^{10}c_{4})^{-1/\alpha_{0}}\rr_{x_{0}}$. Notice that via \eqref{70} with $\bar{\varepsilon}\equiv\hat{\varepsilon}$ and $\bar{\rr}\equiv \hat{\rr}$, \eqref{fpm.1} - on the whole $B(x_{0})$ - and \eqref{ret.2}, for all $\varsigma\in (0,r_{x_{0}}]$ it is $\mf{F}(u;B_{\varsigma}(x))<\hat{\varepsilon}/2$, therefore by \eqref{ret.1.1} - keep in mind Remark \ref{pvmo.3} below - we see that Theorem \ref{t.ex} applies with $B_{\varsigma}(x)$ replacing $B_{\rr_{x_{0}}}(x)$ so
\begin{flalign}\label{ret.4}
\mf{F}(u;B_{s}(x))\le c_{4}\left(\frac{s}{\varsigma}\right)^{\alpha_{0}}\mf{F}(u;B_{\varsigma}(x))+c_{5}\sup_{\sigma\le \varsigma/4}\mf{K}\left[\left(\sigma^{m}\mint_{B_{\sigma}(x)}\snr{f}^{m} \dx\right)^{1/m}\right]^{1/(p-1)}
\end{flalign}
holds for all $s,\varsigma$ so that $0<s\le \varsigma\le r_{x_{0}}$. On the other hand, if $0<s\le r_{x_{0}}<\varsigma\le \rr_{x_{0}}$ we can apply \eqref{ret.4} with $\varsigma= r_{x_{0}}$ and \eqref{ls.42.1} with $\nu=(2^{10}c_{4})^{-1/\alpha_{0}}$ and $\kk=0$ to get
\begin{eqnarray*}
\mf{F}(u;B_{s}(x))&\le&c_{4}\left(\frac{s}{r_{x_{0}}}\right)^{\alpha_{0}}\mf{F}(u;B_{r_{x_{0}}}(x))+c_{5}\sup_{\sigma\le r_{x_{0}}/4}\mf{K}\left[\left(\sigma^{m}\mint_{B_{\sigma}(x)}\snr{f}^{m} \dx\right)^{1/m}\right]^{1/(p-1)}\nonumber \\
&\le&2^{16+\frac{10n}{p\alpha_{0}}}c_{4}^{2+\frac{np}{\alpha_{0}}}\left(\frac{s}{\varsigma}\right)^{\alpha_{0}}\mf{F}(u;B_{\varsigma}(x))+c_{5}\sup_{\sigma\le \varsigma/4}\mf{K}\left[\left(\sigma^{m}\mint_{B_{\sigma}(x)}\snr{f}^{m} \dx\right)^{1/m}\right]^{1/(p-1)}.
\end{eqnarray*}
Finally, when $0<r_{x_{0}}<s\le \varsigma\le \rr_{x_{0}}$ we see that $(2^{10}c_{4})^{-1/\alpha_{0}}\le s/\varsigma\le (2^{10}c_{4})^{1/\alpha_{0}}$, therefore \eqref{ls.42.1} yields that
\begin{eqnarray*}
\mf{F}(u;B_{s}(x))\le 2^{16+\frac{10n}{p\alpha_{0}}}c_{4}^{1+\frac{n}{p\alpha_{0}}}\left(\frac{s}{\varsigma}\right)^{\alpha_{0}}\mf{F}(u;B_{\varsigma}(x)).
\end{eqnarray*}
Setting $c_{6}:=2^{16+\frac{10n}{p\alpha_{0}}}c_{4}^{2+\frac{n}{p\alpha_{0}}}$ and merging the three previous displays, we obtain \eqref{ret.1}. The proof is complete.
\end{proof}
\begin{remark}\label{retrem}
\emph{We stress that if we just want to reiterate \eqref{68} to small scales without extending it to all points belonging to open neighborhoods of $x_{0}\in \mathcal{R}_{u}$, we can weaken assumption \eqref{fpm} to \eqref{fpm.2} to obtain
\begin{flalign}\label{retes}
\mf{F}(u;B_{s}(x_{0}))\le c_{6}\left(\frac{s}{\varsigma}\right)^{\alpha_{0}}\mf{F}(u;B_{\varsigma}(x))+c_{5}\sup_{\sigma\le \varsigma/4}\mf{K}\left[\left(\sigma^{m}\mint_{B_{\sigma}(x_{0})}\snr{f}^{m} \ \dx\right)^{1/m}\right]^{1/(p-1)}
\end{flalign}
for all $s,\varsigma$ so that $0<s\le \varsigma\le \rr$, where $\rr$ is the radius on which \eqref{55} originally holds.
}
\end{remark}
As a byproduct of the proof of Theorem \ref{t.ex} we find the following intermediate result.
\begin{corollary}\label{bmo}
Under assumptions \eqref{assf}-\eqref{p0}, \eqref{f} and \eqref{fpm.2}, let $u\in W^{1,p}(\Omega,\mathbb{R}^{N})$ be a local minimizer of \eqref{exfun}, $x_{0}\in \mathcal{R}_{u}$ be a point and $M\equiv M(x_{0})$ be the positive constant in \eqref{ru.0}. There are parameters $\hat{\varepsilon}\equiv \hat{\varepsilon}(\textnormal{\texttt{data}}_{\textnormal{c}},M^{q-p})$ and $\hat{\rr}\equiv \hat{\rr}(\textnormal{\texttt{data}}_{\textnormal{c}},f(\cdot),M^{q-p})$ such that if $\bar{\varepsilon}\equiv \hat{\varepsilon}$ and $\bar{\rr}\equiv \hat{\rr}$ in \eqref{ru.0}, then
\begin{eqnarray}\label{bmo.1}
\sup_{\sigma\le \rr}\mf{F}(u;B_{\sigma}(x_{0}))\le c_{7}\hat{\varepsilon},
\end{eqnarray}
where $c_{7}:=c_{6}+c_{5}$, $c_{7}\equiv c_{7}(\textnormal{\texttt{data}}_{\textnormal{c}},M^{q-p})$. 
\end{corollary}
Let us compare our Corollary \ref{bmo} with \cite[Proposition 5.1]{kumi}, where the authors derive a $BMO$ result analogous to \eqref{bmo.1} for nonhomogeneous systems with standard $p$-growth. Given the uniform ellipticity of the differential operator considered there, the authors obtain pointwise $BMO$-estimates under weaker assumptions than \eqref{fpm.2}: precisely they only need to impose the Morrey type constraint
\eqn{wa}
$$\rr^{m}\mint_{B_{\rr}(x_{0})}\snr{f}^{m} \ \dx\lesssim \varepsilon,$$ which is guaranteed for instance if $f$ belongs to the Marcinkiewicz space $L(n,\infty)$, see \cite[Theorem 1.5]{kumi}. The smallness condition \eqref{wa} is in turn implied by \eqref{fpm.2}, cf. \eqref{fpm.1} and \eqref{d.3}, but the corresponding natural function space criterion $f\in L(n,\infty)$ is too weak to assure the validity of \eqref{fpm.2}. Similar results to \cite[Theorem 1.5, Proposition 5.1]{kumi} also hold in the quasiconvex setting when the integrand $F(\cdot)$ features standard $p$-growth as there is no need to provide uniform control on the gradient average during the iterative procedure, cf. \cite{dumi} (of course, when $p=q$, $\mathcal{F}(\cdot)\equiv \bar{\mathcal{F}}(\cdot)$). On the other hand, in the genuine $(p,q)$-growth case (i.e.: if $p<q$), the nonuniform Legendre-Hadamard ellipticity \eqref{ellr} of $F(\cdot)$ forces all the constants involved in the process of linearization to feature dependencies on $M^{q-p}$, therefore the whole machinery designed in the proof of Theorem \ref{t.ex}, crucially relying on the finiteness of $\mathbf{I}^{f}_{1,m}(\cdot)$, is fundamental for ruling out the possible blow up of the bounding constants appearing in \eqref{68}. This is not just a technical issue: indeed, at the level of partial regularity, the controlled boundedness of the average can be seen as a "linearized" counterpart of the boundedness of the $L^{\infty}$-norm of the gradient of minima of convex, nonuniformly elliptic variational integrals, which is in turn guaranteed by the boundedness of $\mathbf{I}^{f}_{1,m}(\cdot)$, see \cite{bm,demi1}.

\begin{remark}
\emph{It is worth mentioning that the above discussion is no longer true within the (slightly) weakened framework described in Remark \ref{assumptions}. In fact, with \eqref{assf.1.1}, \eqref{mm.1} and \eqref{mm.2} replacing $\eqref{assf}_{2}$, $\eqref{assf}_{3}$ and $\eqref{assf}_{4}$ respectively, all the bounding constants are nondecreasing, unknown functions of $M$ and this dependency does not necessarily disappear when $p=q$, so the boundedness of $\mathbf{I}^{f}_{1,m}(\cdot)$ is required to prevent the blow-up of such constants.}
\end{remark}

\subsection{VMO estimates}\label{pt1} 
Let $x_{0}\in \mathcal{R}_{u}$ be any point, with $M\equiv M(x_{0})$ being the positive constant in \eqref{ru.0} and $\bar{\varepsilon},\bar{\rr}$ still to be determined, assume \eqref{fpm} 
and, for reasons that will be clear in a few lines, introduce constants
\begin{flalign}\label{hhh.1}
&H_{3}:=H_{1}\max\left\{2^{10p}c_{6}^{p/2},\frac{2^{6p}}{\tau^{8n}}\right\},\quad  H_{4}:=H_{2}^{1+p/2}\max\left\{\frac{2^{16npq}c_{0}^{p/2}c_{5}^{p/2}H_{1}}{(\tau\theta)^{\frac{2npq}{p-1}}},\left(\frac{2^{16(p+n)}H_{1}}{\varepsilon_{1}(\tau\theta)^{8n}}\right)^{\frac{p}{2(p-1)}}\right\}.
\end{flalign}
In the above displays, $H_{1}$ and $H_{2}$ are the same constants defined in \eqref{hhh} and by definition it is $H_{3}>H_{1}$, $H_{4}>H_{2}$.  
We first fix 
\eqn{***}
 $$\bar{\varepsilon}\equiv \varepsilon_{*}:=2^{-10}c_{7}^{-1}\hat{\varepsilon}$$ with $\hat{\varepsilon}\equiv \hat{\varepsilon}(\textnormal{\texttt{data}}_{\textnormal{c}},M^{q-p})$ taken from \eqref{be}. Then, from assumption \eqref{fpm} and the absolute continuity of Lebesgue integral, we see that we can find a threshold radius $\rr_{*}\equiv \rr_{*}(\textnormal{\texttt{data}}_{\textnormal{c}},M^{q-p},f(\cdot))\in (0,\hat{\rr}]$, where $\hat{\rr}\equiv \hat{\rr}(\textnormal{\texttt{data}}_{\textnormal{c}},M^{q-p},f(\cdot))$ is the threshold radius appearing in Theorem \ref{t.ex}, such that \eqref{fpm.1} with $c_{7}\stackrel{\eqref{c4c5}}{>}c_{3}$ instead of $c_{3}$ and $H_{4}\stackrel{\eqref{hhh.1}}{>}H_{2}$ replacing $H_{2}$ holds in a neighborhood of $x_{0}$, i.e.:
\begin{flalign}\label{fpm.1.1.1}
c_{8}\mf{K}\left(\mathbf{I}^{f}_{1,m}(x,\rr_{*})\right)^{\frac{1}{p-1}}<\frac{\hat{\varepsilon}}{10} \quad \mbox{for all} \ \ x\in B_{d_{x_{0}}}(x_{0}),\qquad\qquad c_{8}:=\frac{2^{64np(q+2)}c_{7}H_{4}^{1+2/p}}{(\varepsilon_{0}\tau\theta)^{4np(q+2)}}.
\end{flalign}
We stress that, as a direct consequence, \eqref{d.3} and \eqref{d.3.1.1} are uniformly satisfied on $B_{d_{x_{0}}}(x_{0})$. By \eqref{fpm} and \eqref{fpm.1.1.1} we see that Proposition \ref{ret} applies and there is an open neighborhood $B(x_{0})\subset \mathcal{R}_{u}$ and a positive radius $\rr_{x_{0}}\equiv \rr_{x_{0}}(\textnormal{\texttt{data}}_{\textnormal{c}},M^{q-p},f(\cdot))\in (0,\rr_{*}]$ such that \eqref{ret.1}-\eqref{ret.1.1} hold for any $x\in B(x_{0})$ and all $0<s\le \varsigma\le \rr_{x_{0}}$. Clearly, we can always assume that $B(x_{0})\subset B_{d_{x_{0}}}(x_{0})$, cf. Remark \ref{rx}.
Now, let $r\in (0,1)$ be any number. Since \eqref{fpm} holds, we can select a radius $\rr''\equiv \rr''(\textnormal{\texttt{data}}_{\textnormal{c}},M^{q-p},f(\cdot),r)\in (0,\rr_{x_{0}}]$ such that
\begin{eqnarray}\label{74.1.1}
c_{5}\sup_{\sigma\le \rr''}\mf{K}\left[\left(\sigma^{m}\mint_{B_{\sigma}(x)}\snr{f}^{m} \dx\right)^{1/m}\right]^{1/(p-1)}\le \frac{r}{2}\qquad \mbox{for all} \ \ x\in B_{d_{x_{0}}}(x_{0}),
\end{eqnarray}
and by \eqref{ret.1} with $s\equiv \rr''$ and $\varsigma\equiv \rr_{x_{0}}$, \eqref{fpm.1.1.1}, \eqref{d.4} and \eqref{70} with $\bar{\varepsilon}\equiv \varepsilon_{*}$, $\bar{\rr}\equiv \rr_{*}$ we have
\begin{eqnarray}\label{74}
\mf{F}(u;B_{\rr''}(x))&\le& c_{6}\left(\frac{\rr''}{\rr_{x_{0}}}\right)^{\alpha_{0}}\mf{F}(u;B_{\rr_{x_{0}}}(x))+c_{5}\sup_{\sigma\le \rr_{x_{0}}/4}\mf{K}\left[\left(\sigma^{m}\mint_{B_{\sigma}(x)}\snr{f}^{m} \ \dx\right)\right]^{1/(p-1)}\nonumber \\
&\le&c_{6}\varepsilon_{*}+\frac{1}{2}\le \hat{\varepsilon}+ \frac{1}{2}\le 1.
\end{eqnarray}
Choose now $s_{r}\equiv s_{r}(\textnormal{\texttt{data}}_{\textnormal{c}},M^{q-p},f(\cdot),r)\in (0,\rr'']$ so small that
\eqn{74.1.2}
$$c_{7}(s_{r}/\rr'')^{\alpha_{0}}\le r/2.$$
Inserting inequalities \eqref{74.1.1}-\eqref{74.1.2} in \eqref{ret.1} with $\varsigma\equiv\rr''$ and $s\equiv s_{r}$, we deduce that for any $r\in (0,1)$ there exists a radius $s_{r}\equiv s_{r}(\textnormal{\texttt{data}}_{\textnormal{c}},M^{q-p},f(\cdot),r)\in (0,\rr'')$ such that
\begin{eqnarray*}
s\le s_{r} \ \Longrightarrow \ \mf{F}(u;B_{s}(x))\le r\qquad \mbox{for all} \ \ x\in B(x_{0}).
\end{eqnarray*}
This and a standard covering argument yield that $Du\in VMO_{\loc}(\mathcal{R}_{u},\mathbb{R}^{N\times n})$.
We have just proven
\begin{theorem}\label{t1}
Under assumptions \eqref{assf}-\eqref{p0}, \eqref{f} and \eqref{fpm}, let $u\in W^{1,p}(\Omega,\mathbb{R}^{N})$ be a local minimizer of \eqref{exfun}. Then, here exists an open set $\Omega_{u}\subset \Omega$ such that
\begin{eqnarray*}
\snr{\Omega\setminus \Omega_{u}}=0\qquad \mbox{and}\qquad Du\in VMO_{\loc}(\Omega_{u},\mathbb{R}^{N\times n}).
\end{eqnarray*}
In particular, $\Omega_{u}$ can be characterized as
\begin{flalign*}
&\Omega_{u}=\left\{\frac{}{} x_{0}\in \Omega\colon \exists \ M\equiv M(x_{0})\subset (0,\infty)\colon
  \snr{(Du)_{B_{\rr}(x_{0})}}<M \ \mbox{and} \ \mf{F}(u;B_{\rr}(x_{0}))<\varepsilon_{*} \ \mbox{for some} \ \rr\in (0,\rr_{*}]\frac{}{}\right\},
\end{flalign*}
where $\varepsilon_{*}\equiv \varepsilon_{*}(\textnormal{\texttt{data}}_{\textnormal{c}},M^{q-p})$ has been defined in \eqref{***} and $\rr_{*}\equiv \rr_{*}(\textnormal{\texttt{data}}_{\textnormal{c}},M^{q-p},f(\cdot))$ is as in \eqref{fpm.1.1.1}.
\end{theorem}
A pointwise consequence of the arguments developed so far is the following
\begin{corollary}\label{pvmo}
Under assumptions \eqref{assf}-\eqref{p0}, \eqref{f} and \eqref{fpm.2}, let $u\in W^{1,p}(\Omega,\mathbb{R}^{N})$ be a local minimizer of \eqref{exfun}, $x_{0}\in \mathcal{R}_{u}$ be a point and $M\equiv M(x_{0})$ be the positive constant in \eqref{ru.0}. There exist a number $\varepsilon_{*}\equiv \varepsilon_{*}(\textnormal{\texttt{data}}_{\textnormal{c}},M^{q-p})$ and a threshold radius $\rr_{*}\equiv \rr_{*}(\textnormal{\texttt{data}}_{\textnormal{c}},M^{q-p},f(\cdot))$ such that if $\bar{\varepsilon}\equiv \varepsilon_{*}$ and $\bar{\rr}\equiv \rr_{*}$ in \eqref{ru.0}, 
then 
\begin{eqnarray}\label{pvmo.2}
\lim_{\rr\to 0}\mf{F}(u;B_{\rr}(x_{0}))=0.
\end{eqnarray}
\end{corollary}
\begin{remark}\label{pvmo.3}
\emph{Let us record a couple of useful observations.
\begin{itemize}
    \item If $x_{0}\in \mathcal{R}_{u}$ with $M\equiv M(x_{0})$ positive constant in \eqref{ru.0}, $\bar{\varepsilon}\equiv \varepsilon_{*}$, $\bar{\rr}\equiv \rr_{*}<d_{x_{0}}$, there exists an open neighborhood of $x_{0}$ on which \eqref{70} is verified in correspondence of the same parameters, therefore if \eqref{fpm.1.1.1} holds uniformly in $B(x_{0})$ - and this is always the case if we assume \eqref{fpm} - then the convergence in \eqref{pvmo.2} is uniform on $B(x_{0})$.
    \item The results in Proposition \ref{bmo}, Theorem \ref{t1} and Corollary \ref{pvmo} are still true if we replace $M$ with $2^{6}(M+1)$ without affecting the various parameters involved, since they are computed in correspondence of values of $40000(M+1)$, cf. Propositions \ref{p1}-\ref{p2}.
\end{itemize}
}
\end{remark}

\section{Partial gradient continuity}\label{pgc}
In this section, we prove Theorems \ref{t4}, \ref{t2} and \ref{t5}. Let $x_{0}\in \mathcal{R}_{u}$ be any point with $\bar{\varepsilon},\bar{\rr}$ still to be determined and $M\equiv M(x_{0})$ being the positive constant prescribed by \eqref{ru.0}. We then set
\begin{eqnarray}\label{pe.0}
\varepsilon':=\frac{\varepsilon_{*}(\tau\theta)^{8npq}}{2^{16npq}} \ \Longrightarrow \ \varepsilon'\equiv \varepsilon'(\textnormal{\texttt{data}}_{\textnormal{c}},M^{q-p}),
\end{eqnarray}
where $\varepsilon_{*}\equiv \varepsilon_{*}(\textnormal{\texttt{data}}_{\textnormal{c}},M^{q-p})$ is the same smallness parameter from Theorem \ref{t1} and Corollary \ref{pvmo} and $\tau\equiv \tau(\textnormal{\texttt{data}},M^{q-p})$, $\theta\equiv \theta(\textnormal{\texttt{data}},\mu(\cdot),M^{q-p})$ are the parameters from Propositions \ref{p1}-\ref{p2} and Proposition \ref{p3} respectively,. Notice that if $\hat{\varepsilon}\equiv \hat{\varepsilon}(\textnormal{\texttt{data}}_{\textnormal{c}},M^{q-p})$ is the threshold quantity introduced in Theorem \ref{t.ex} and $\varepsilon_{2}\equiv \varepsilon_{2}(\textnormal{\texttt{data}}_{\textnormal{c}},M^{q-p})$ is the small parameter appearing in Proposition \ref{p3}, by definition it is 
\begin{eqnarray}\label{pe.1}
0<\varepsilon'<\varepsilon_{*}<\hat{\varepsilon}<\varepsilon_{2}.
\end{eqnarray}
Next, keeping \eqref{con.1.1} in mind, we can also find a limiting radius $\rr'\equiv \rr'(\textnormal{\texttt{data}}_{\textnormal{c}},M^{q-p},f(\cdot))\in (0,\rr_{*}]$ such that,
\begin{eqnarray}\label{pe.3}
0<\rr'\le \rr_{*}\le \hat{\rr},
\end{eqnarray}
where $\rr_{*}\equiv \rr_{*}(\textnormal{\texttt{data}}_{\textnormal{c}},M^{q-p},f(\cdot))$ is the threshold radius determined in Theorem \ref{t1} and Corollary \ref{pvmo} and $\hat{\rr}\equiv \hat{\rr}(\textnormal{\texttt{data}}_{\textnormal{c}},M^{q-p},f(\cdot))$ is the limiting radius from Theorem \ref{t.ex}, and
\begin{eqnarray}\label{pe.2}
c_{8}\mf{K}\left(\mathbf{I}^{f}_{1,m}(x_{0},\rr')\right)^{\frac{1}{p-1}}\le \frac{\varepsilon'(\tau\theta)^{n+\frac{4nq}{p-1}}}{2^{10npq}} ,
\end{eqnarray}
with $c_{8}$ being the constant in \eqref{fpm.1.1.1}, which implies that
\begin{eqnarray}\label{pe.2.1}
\sup_{\sigma\le \rr/4}\mf{K}\left[\left(\sigma^{m}\mint_{B_{\sigma}(x_{0})}\snr{f}^{m} \dx\right)^{1/m}\right]^{1/(p-1)}\stackrel{\eqref{d.3}}{\le} \frac{\varepsilon'(\varepsilon_{0}\tau\theta)^{8np}}{2^{56np(q+2)}c_{7}H_{4}^{1+2/p}}.
\end{eqnarray}
In \eqref{ru.0} we fix $\bar{\varepsilon}=\varepsilon'$ and $\bar{\rr}=\rr'$ so now for $x_{0}\in\mathcal{R}_{u}$ there is $\rr\in (0,\rr']$ such that
\begin{eqnarray}\label{pe.6}
\mf{F}(u;B_{\rr}(x_{0}))<\varepsilon'\qquad \mbox{and}\qquad \snr{(Du)_{B_{\rr}(x_{0})}}< M.
\end{eqnarray}
We then recall the definition of two quantities that already played a crucial role in the proof of Theorem \ref{t.ex}: the composite excess functional, i.e.:
\begin{eqnarray*}
(0,\rr]\ni s\mapsto \mf{C}(x_{0},s):=H_{1}\mf{F}(u;B_{s}(x_{0}))^{p/2}+\snr{(Du)_{B_{s}(x_{0})}}^{p/2},
\end{eqnarray*}
and the nonhomogeneous excess functional
\begin{eqnarray*}
(0,\rr]\ni s\mapsto \mf{N}(x_{0},s):=H_{3}\mf{F}(u;B_{s}(x_{0}))^{p/2}+\frac{2^{16npq}H_{4}}{(\tau\theta)^{4npq}}\mf{K}\left(\mathbf{I}^{f}_{1,m}(x_{0},s)\right)^{\frac{p}{2(p-1)}},
\end{eqnarray*}
with $H_{1}$ as in \eqref{hhh} and $H_{3}$, $H_{4}$ being defined in \eqref{hhh.1}. Let us further streamline the notation introduced in \textbf{Step 1} of the proof of Theorem \ref{t.ex}. For $j\in \N\cup\{-1,0\}$ we set $\rr_{j}:=\tau^{j+1}\rr$, $B_{j}:=B_{\rr_{j}}(x_{0})$ and
\begin{flalign*}
&\mf{F}_{j}:=\mf{F}(u;B_{j}),\qquad \ti{\mf{F}}_{j}:=\ti{\mf{F}}(u;B_{j}),\qquad  A_{j}:=\snr{(Du)_{B_{j}}},\nonumber \\
&\mf{C}_{j}:=\mf{C}(x_{0},\rr_{j}),\quad \ \ \mf{S}_{j}:=\left(\rr_{j}^{m}\mint_{B_{j}}\snr{f}^{m} \ \dx\right)^{1/m}.
\end{flalign*}
Notice that \eqref{pe.0}-\eqref{pe.2} and \eqref{pe.6} allow verifying the assumptions of Corollary \ref{pvmo} (recall Remark \ref{pvmo.3}), so
\begin{eqnarray}\label{pe.8}
\lim_{\sigma\to 0}\mf{F}(u;B_{\sigma}(x_{0}))=0 \ \Longrightarrow \ \lim_{\sigma\to 0}\mf{N}(x_{0},\sigma)=0.
\end{eqnarray}
In particular, Theorem \ref{t.ex}, Proposition \ref{ret} and Remark \ref{retrem} apply so \eqref{68.1}-\eqref{68} and \eqref{retes} hold in this setting as well. 
Once made these preliminary observations, let us rephrase Lemma \ref{til} in a way that better fits our needs. Precisely, we shall replace assumption \eqref{ls.130} with a more suitable one, involving the nonhomogeneous excess functional $\mf{N}(\cdot)$.
\begin{lemma}\label{til.1}
Let $x_{0}\in \mathcal{R}_{u}$, with $M\equiv M(x_{0})$ being the positive constant in \eqref{ru.0}, $\gamma$ be any positive number and $\tau\in (0,2^{-10})$ be the small parameter from Propositions \ref{p1}-\ref{p2}. Assume that $\bar{\varepsilon}\equiv \varepsilon'$, $\bar{\rr}\equiv \bar{\rr}'$ in \eqref{ru.0}, with $\varepsilon'$, $\rr'$ defined in \eqref{pe.0}, \eqref{pe.2} respectively; that
\begin{eqnarray}\label{pe.9}
\mf{N}(x_{0},\varsigma)\le 2\gamma\qquad \mbox{for some} \ \ \varsigma\le \rr
\end{eqnarray}
and that, for integers $k\ge i\ge 0$ the following inequalities:
\begin{flalign}\label{ls.131.1}
\mf{C}(\tau_{j}\varsigma_{0})\le \gamma,\quad \mf{C}(\tau_{j+1}\varsigma_{0})\ge \frac{\gamma}{16} \ \ \mbox{for all} \ \ j\in \{i,\cdots,k\},\qquad \mf{C}(\tau_{i}\varsigma_{0})\le \frac{\gamma}{4}
\end{flalign}
hold with $\varsigma_{0}:=\tau\varsigma$. 
Then it is
\begin{flalign}\label{ls.132.1}
\left\{
\begin{array}{c}
\displaystyle 
\ \mf{C}(\tau_{k+1}\varsigma_{0})\le \gamma \\[8pt]\displaystyle
\ \sum_{j=i}^{k+1}\mf{F}(\tau_{j}\varsigma_{0})^{p/2}\le \frac{\gamma}{2H_{1}}\\[8pt]\displaystyle
\ \sum_{j=i}^{k+1}\mf{F}(\tau_{j}\varsigma_{0})^{p/2}\le 2\mf{F}(\tau_{i}\varsigma_{0})^{p/2}+\frac{2^{7p}\gamma^{(2-p)/p}}{\varepsilon_{1}\tau^{n/2}}\sum_{j=i}^{k}\mf{S}(\tau_{j}\varsigma_{0}),
\end{array}
\right. 
\end{flalign}
with $H_{1}$, $\varepsilon_{1}$ as in \eqref{hhh} and in Propositions\ref{p1}-\ref{p2} respectively.
\end{lemma}
\begin{proof}
The proof is again the same as \cite[Lemma 6.1]{kumi}. There are only two points that require some care. First, in \cite[Lemma 6.1]{kumi} it is assumed that
$$
\left(\sum_{j=i}^{k}\mf{S}(\tau_{j}\varsigma_{0})\right)^{\frac{p}{2(p-1)}}\le \frac{2\gamma}{H_{4}},
$$
which comes as a direct consequence of the definition of the nonhomogeneous excess functional, of \eqref{pe.9} and of \eqref{d.13}; in fact:
$$
\left(\sum_{j=i}^{k}\mf{S}(\tau_{j}\varsigma_{0})\right)^{\frac{p}{2(p-1)}}\le \frac{2^{4np}}{(\tau\theta)^{2np}}\left(\mathbf{I}^{f}_{1,m}(x_{0},\sigma)\right)^{\frac{p}{2(p-1)}}\le \frac{(\tau\theta)^{2npq}}{2^{12npq}H_{4}}\mf{N}(x_{0},\varsigma)\le \frac{2\gamma}{H_{4}}.
$$
Second, in \cite[\emph{Step 2} of the proof of Lemma 6.1]{kumi} it is shown that 
\begin{eqnarray}\label{pe.10}
\varepsilon_{0}A(\tau_{j}\varsigma_{0})\le \mf{F}(\tau_{j}\varsigma_{0})\quad \mbox{cannot hold for all} \ \ j\in \{i,\cdots,k\},
\end{eqnarray}
by means of a contradiction argument based on the alternating application of Propositions \ref{p1}-\ref{p2} and Proposition \ref{p3}. However, the change of scale occurring here between Propositions \ref{p1}-\ref{p2} and Proposition \ref{p3}, prevents us from applying directly the strategy in \cite{kumi}, so we need to follow a different approach to secure \eqref{pe.10}. Assume by contradiction that 
\begin{eqnarray}\label{pe.11}
\varepsilon_{0}A(\tau_{j}\varsigma_{0})\le \mf{F}(\tau_{j}\varsigma_{0})\quad \mbox{holds for some} \ \ j\in \{i,\cdots,k\}.
\end{eqnarray}
Keeping in mind that \eqref{pe.0}-\eqref{pe.2} and \eqref{pe.6} assure the validity of the content of Remark \ref{retrem}, we have
\begin{eqnarray}\label{pe.12}
H_{1}\mf{F}(\tau_{j+1}\varsigma_{0})^{p/2}&\stackrel{\eqref{retes}}{\le}&H_{1}\left[c_{6}\tau^{\alpha_{0}(j+2)}\mf{F}(\varsigma)+c_{5}\sup_{s\le \varsigma/4}\mf{K}\left(\mf{S}(s)\right)^{1/(p-1)}\right]^{p/2}\nonumber \\
&\le&2^{(p-2)/2}H_{1}c_{6}^{p/2}\mf{F}(\varsigma)^{p/2}+2^{(p-2)/2}c_{5}^{p/2}\sup_{s\le \varsigma/4}\mf{K}\left(\mf{S}(s)\right)^{\frac{p}{2(p-1)}}\nonumber \\
&\stackrel{\eqref{d.3}}{\le}&\frac{2^{(p-2)/2}H_{1}c_{6}^{p/2}\mf{N}(x_{0},\varsigma)}{H_{3}}+\frac{2^{\frac{4npq}{p-1}+\frac{p-2}{2}}H_{1}c_{5}^{p/2}}{(\tau\theta)^{\frac{2npq}{p-1}}}\mf{K}\left(\mathbf{I}^{f}_{1,m}(x_{0},\varsigma)\right)^{\frac{p}{2(p-1)}}\nonumber \\
&\le&\mf{N}(x_{0},\varsigma)\left(\frac{2^{(p-2)/2}H_{1}c_{6}^{p/2}}{H_{3}}+\frac{2^{\frac{4npq}{p-1}+\frac{p-2}{2}}H_{1}c_{5}^{p/2}(\tau\theta)^{4npq}}{H_{4}2^{16npq}(\tau\theta)^{\frac{2npq}{p-1}}}\right)\nonumber \\
&\stackrel{\eqref{pe.9}}{\le}&\gamma\left(\frac{2^{p/2}H_{1}c_{6}^{p/2}}{H_{3}}+\frac{H_{1}c_{5}^{p/2}}{H_{4}2^{10npq}}\right)\stackrel{\eqref{hhh.1}}{\le}\frac{\gamma}{2^{6}}.
\end{eqnarray}
Moreover, the contradiction assumption \eqref{pe.11} yields that
\begin{eqnarray}\label{pe.13}
A(\tau_{j}\varsigma_{0})\le \frac{1}{\varepsilon_{0}}\mf{F}(\tau_{j}\varsigma_{0})\le \frac{1}{\varepsilon_{0}}\left(\frac{\mf{C}(\tau_{j}\varsigma_{0})}{H_{1}}\right)^{2/p}\stackrel{\eqref{ls.131.1}_{1}}{\le}\frac{1}{\varepsilon_{0}}\left(\frac{\gamma}{H_{1}}\right)^{2/p}\stackrel{\eqref{hhh}}{\le}\left(\frac{\gamma}{2^{6}}\right)^{2/p},
\end{eqnarray}
therefore we can conclude with
\begin{eqnarray*}
\mf{C}(\tau_{j+1}\varsigma_{0})&\stackrel{\eqref{tri.1.1}}{\le}&\frac{2^{3p}}{\tau^{n/2}}\mf{F}(\tau_{j}\varsigma_{0})^{p/2}+A(\tau_{j}\varsigma_{0})^{p/2}+H_{1}\mf{F}(\tau_{j+1}\varsigma_{0})^{p/2}\nonumber \\
&\stackrel{\eqref{pe.12},\eqref{pe.13}}{\le}&\frac{2^{3p}\mf{C}(\tau_{j}\varsigma_{0})}{\tau^{n/2}H_{1}}+\frac{\gamma}{2^{5}}\stackrel{\eqref{hhh},\eqref{ls.131.1}_{1}}{\le}\frac{3\gamma}{2^{6}}<\frac{\gamma}{16},
\end{eqnarray*}
thus contradicting the second inequality in \eqref{ls.131.1}$_{1}$ and \eqref{pe.11} is proven. 
The rest of the proof is the same as \cite[Lemma 6.1]{kumi}, with $H_{4}$ in place of $H_{2}$.
\end{proof}
We stress that in the statement and in the (partial) proof of the above lemma we used the same notation described in \textbf{Step 1} of the proof of Theorem \ref{t.ex} for coherence.
At this stage, we consider two possibilities: 
\begin{eqnarray}\label{gr}
\frac{\gamma}{8}:=A_{0}^{p/2}>\frac{\mf{N}(x_{0},\rr)}{16}\qquad \mbox{or}\qquad A_{0}^{p/2}\le \frac{\mf{N}(x_{0},\rr)}{16}=:\frac{\gamma}{8},
\end{eqnarray}
and to observe that in any case it is
\begin{eqnarray}\label{pe.14}
\mf{N}(x_{0},\rr)\le 2\gamma.
\end{eqnarray}

\subsection{Large gradient regime}
We start by proving a technical lemma.
\begin{lemma}\label{t.lem}
Assume \eqref{gr}$_{1}$ and let $x_{0}\in \mathcal{R}_{u}$ be any point with $M\equiv M(x_{0})$ being the positive constant in \eqref{ru.0}. If $\bar{\varepsilon}\equiv \varepsilon'$ and $\bar{\rr}\equiv \rr'$ in \eqref{ru.0}, with $\varepsilon'$, $\rr'$ as in \eqref{pe.0}, \eqref{pe.2} respectively, it holds that
\begin{eqnarray}\label{pe.15}
\sum_{j=0}^{\infty}\mf{F}_{j}^{p/2}\le \frac{\mf{N}(x_{0},\rr)}{H_{1}}
\end{eqnarray}
and it is
\begin{eqnarray}\label{pe.16}
\frac{\gamma}{16}\le A_{j}^{p/2}\le \gamma\qquad \mbox{for all} \ \ j\in \N\cup\{0\}.
\end{eqnarray}
\end{lemma}
\begin{proof}
The inequality on the right-hand side of \eqref{pe.16} can be proven as done for \eqref{ls.166}, while its left-hand side is derived by following the same arguments leading to \eqref{ls.160} and using $\eqref{gr}_{1}$. This in particular yields that
\begin{eqnarray}\label{pe.18}
\mf{C}_{j}\le \gamma\qquad \mbox{for all} \ \ j\in \N\cup\{0\}.
\end{eqnarray}
The only differences with the procedure developed in \emph{Step 5.4} of the proof of Theorem \ref{t.ex} is that we need to replace $r_{1}''$ with $\rr_{0}$, $r_{1}'$ with $\rr$ and to apply Lemma \ref{til.1} with $\varsigma\equiv \rr$ instead of Lemma \ref{til} - just keep in mind that now in the definition of $\mf{N}(\cdot)$ the constants defined in \eqref{hhh.1} substitute those in \eqref{hhh} and that all averages are bounded by means of \eqref{68.1} so no upper bound on $\gamma>0$ is needed. Next, let us check the validity of \eqref{pe.15}.
By \eqref{pe.16} we see that $\mf{C}_{j}\ge \gamma/16$ for all $j\in \N\cup\{0\}$ and by \eqref{gr}$_{1}$ and \eqref{68.1} it is $\gamma\le 2^{3+3p}(M+1)^{p/2}$, so recalling also \eqref{pe.14}, \eqref{pe.18}, \eqref{ls.163} and \eqref{ls.157} (with $\rr_{0}$ instead of $r_{1}''$), we can apply Lemma \ref{til.1} with $i=0$ and for all $k\in \N$. We then have
\begin{eqnarray*}
\sum_{j=0}^{\infty}\mf{F}_{j}^{p/2}&\stackrel{\eqref{ls.132.1}_{3}}{\le}&2\mf{F}_{0}^{p/2}+\frac{2^{7p}\gamma^{\frac{2-p}{p}}}{\tau^{n/2}\varepsilon_{1}}\sum_{j=0}^{\infty}\mf{S}_{j}\stackrel{\eqref{d.13},\eqref{pe.14}}{\le}2\mf{F}_{0}^{p/2}+\frac{2^{8p+4n}\mf{N}(x_{0},\rr)^{\frac{2-p}{p}}}{\varepsilon_{1}(\tau\theta)^{4n}}\mathbf{I}^{f}_{1,m}(x_{0},\rr)\nonumber \\
&\le&2\mf{F}_{0}^{p/2}+\frac{2^{8p+4n}\mf{N}(x_{0},\rr)}{\varepsilon_{1}(\tau\theta)^{4n}H_{4}^{2(p-1)/p}}\stackrel{\eqref{hhh.1}}{\le}2\mf{F}_{0}^{p/2}+\frac{\mf{N}(x_{0},\rr)}{4H_{1}}\nonumber \\
&\stackrel{\eqref{tri.1}}{\le}&\frac{2^{3p+1}\mf{F}_{-1}^{p/2}}{\tau^{n/2}}+\frac{\mf{N}(x_{0},\rr)}{4H_{1}}\le \frac{\mf{N}(x_{0},\rr)}{2H_{1}},
\end{eqnarray*}
where we also used the definition of $\mf{N}(\cdot)$. This guarantees the validity of \eqref{pe.15} and the proof is complete.
\end{proof}
Now we are ready for the core lemma of this section.
\begin{lemma}\label{l.pe}
Assume 
\begin{eqnarray}\label{d.22}
\snr{(Du)_{B_{\tau\varsigma}(x_{0})}}^{p/2}>\frac{\mf{N}(x_{0},\varsigma)}{16}\qquad \mbox{for some} \ \ \varsigma\le \rr
\end{eqnarray}
and let $x_{0}\in \mathcal{R}_{u}$ be any point with $M\equiv M(x_{0})$ being the positive constant in \eqref{ru.0}. If $\bar{\varepsilon}\equiv \varepsilon'$ and $\bar{\rr}\equiv \rr'$ in \eqref{ru.0}, with $\varepsilon'$, $\rr'$ as in \eqref{pe.0}, \eqref{pe.2} respectively, then the limits in \eqref{lp} and \eqref{lv} exist and the precise representatives of $Du$ and $V_{p}(Du)$ are attained. Moreover it holds that
\begin{eqnarray}\label{d.14}
\snr{Du-(Du)_{B_{\varsigma}(x_{0})}}\le c\mf{F}(u;B_{\varsigma}(x_{0}))+c\mf{K}\left(\mathbf{I}^{f}_{1,m}(x_{0},\varsigma)\right)^{\frac{1}{p-1}},
\end{eqnarray}
with $c\equiv c(\textnormal{\texttt{data}},M^{q-p})$ and
\begin{eqnarray}\label{d.15}
\snr{V_{p}(Du(x_{0}))-(V_{p}(Du))_{B_{\varsigma}(x_{0})}}\le c\ti{\mf{F}}(u;B_{\varsigma}(x_{0}))+c\mf{K}\left(\mathbf{I}^{f}_{1,m}(x_{0},\varsigma)\right)^{\frac{p}{2(p-1)}},
\end{eqnarray}
for $c\equiv c(\textnormal{\texttt{data}},M^{q-p})$.
\end{lemma}
\begin{proof}
For the sake of simplicity, we split the proof in two steps.
\subsection*{Step 1: \eqref{d.14}-\eqref{d.15} at the $\rr$-scale}
Let $\varsigma=\rr$ and set $\gamma/8:=A_{0}^{p/2}$. Proceeding as for \cite[(6.53) from the proof of Lemma 6.4]{kumi}, we estimate
\begin{eqnarray}\label{d.16}
\snr{V_{p}((Du)_{B_{j+1}})-V_{p}((Du)_{B_{j}})}&\le&c\mf{F}_{j}^{p/2},
\end{eqnarray}
with $c\equiv c(\textnormal{\texttt{data}},M^{q-p})$ and it is
\begin{eqnarray}\label{d.18}
A_{j}^{2-p}\stackrel{\eqref{pe.16}}{\le}\left(\frac{16}{\gamma}\right)^{2(p-2)/p}\qquad \mbox{for all} \ \ j\in \N\cup\{0\}.
\end{eqnarray}
Then, for integers $0\le i\le k-1$ we bound
\begin{eqnarray}\label{d.17}
\snr{(Du)_{B_{k}}-(Du)_{B_{i}}}&\le&\sum_{j=i}^{k-1}\snr{(Du)_{B_{j+1}}-(Du)_{B_{j}}}\stackrel{\eqref{Vm}}{\le} c\sum_{j=i}^{k-1}\frac{\snr{V_{p}((Du)_{B_{j+1}})-V_{p}((Du)_{B_{j}})}}{(A_{j+1}+A_{j})^{(p-2)/2}}\nonumber \\
&\stackrel{\eqref{d.16},\eqref{d.18}}{\le}&c\gamma^{(2-p)/p}\sum_{j=i}^{\infty}\mf{F}_{j}^{p/2}\stackrel{\eqref{pe.15}}{\le}\frac{c\mf{N}(x_{0},\rr)}{H_{1}\gamma^{(p-2)/p}}\le c\gamma^{(2-p)/p}\mf{N}(x_{0},\rr),
\end{eqnarray}
for $c\equiv c(\textnormal{\texttt{data}},M^{q-p})$. Here we also used that $H_{1}\ge 1$, cf. \eqref{hhh}. This shows that the sequence $\{(Du)_{B_{j}}\}_{j\in \N\cup\{0\}}$ is Cauchy therefore there exists the limit $\lim_{j\to \infty}(Du)_{B_{j}}=:\ell\in\mathbb{R}^{N\times n}$, which defines the precise representative of $Du$ at $x_{0}$, i.e.: $\ell=Du(x_{0})$. In fact, by \eqref{pe.15} we see that the series in the last line of \eqref{d.17} converges. Moreover, given any $\sigma\le \rr_{0}$ there exists $j_{\sigma}\in \N$ such that $\tau^{j_{\sigma}+1}\rr_{0}<\sigma\le \tau^{j_{\sigma}}\rr_{0}$, therefore
\begin{eqnarray}\label{d.19}
\lim_{\sigma\to 0}\snr{\ell-(Du)_{B_{\sigma}(x_{0})}}&\le& \lim_{j_{\sigma}\to \infty}\left[\snr{\ell-(Du)_{B_{j_{\sigma}}}}+\snr{(Du)_{B_{\sigma}(x_{0})}-(Du)_{B_{j_{\sigma}}}}\right]\nonumber \\
&\le&\lim_{j_{\sigma}\to \infty}\left[\snr{\ell-(Du)_{B_{j_{\sigma}}}}+\tau^{-n/p}\mf{F}_{j_{\sigma}}\right]\stackrel{\eqref{d.17},\eqref{pe.8}}{=}0
\end{eqnarray}
and \eqref{lp} is proven. We then send $k\to \infty$ in \eqref{d.17} to get
\begin{eqnarray}\label{d.20}
\snr{Du(x_{0})-(Du)_{B_{i}}}&\stackrel{\eqref{d.19}}{\le}&c\gamma^{(2-p)/p}\mf{N}(x_{0},\rr),
\end{eqnarray}
with $c\equiv c(\textnormal{\texttt{data}},M^{q-p})$ and, by triangular inequality, 
\begin{eqnarray*}
\snr{Du(x_{0})-(Du)_{B_{-1}(x_{0})}}&\le&\snr{Du(x_{0})-(Du)_{B_{0}}}+\snr{(Du)_{B_{-1}}-(Du)_{B_{0}}}\nonumber \\
&\stackrel{\eqref{d.17}}{\le}&c\gamma^{(2-p)/p}\mf{N}(x_{0},\rr)+c\mf{F}_{-1}\stackrel{\eqref{gr}_{1}}{\le}c\mf{N}(x_{0},\rr)^{2/p},
\end{eqnarray*}
with $c\equiv c(\textnormal{\texttt{data}},M^{q-p})$, which is \eqref{d.14}. The proof of \eqref{d.15} and \eqref{lv} is quite similar. In fact with integers $0\le i\le k-1$ we have
\begin{eqnarray}\label{pe.30}
\snr{(V_{p}(Du))_{B_{k}}-(V_{p}(Du))_{B_{i}}}&\le&\sum_{j=i}^{k-1}\snr{(V_{p}(Du))_{B_{j+1}}-(V_{p}(Du))_{B_{j}}}\nonumber \\
&\le&\sum_{j=i}^{k-1}\mint_{B_{j+1}}\snr{V_{p}(Du)-(V_{p}(Du))_{B_{j}}} \ \dx\nonumber \\
&\le&c\sum_{j=i}^{k-1}\ti{\mf{F}}_{j}\stackrel{\eqref{equiv}}{\le}c\sum_{j=i}^{k-1}\mf{F}_{j}^{p/2}\stackrel{\eqref{pe.15}}{\le}c\mf{N}(x_{0},\rr),
\end{eqnarray}
for $c\equiv c(\textnormal{\texttt{data}},M^{q-p})$, so sequence $\{(V(Du))_{B_{j}}\}_{j\in \N\cup\{0\}}$ is Cauchy and the limit
\eqn{vvv}
$$\lim_{j\to \infty}(V_{p}(Du))_{B_{j}}=\ell_{V}\in \mathbb{R}^{N\times n}$$ exists. To be precise, the whole limit in \eqref{lv} exists and equals $\ell_{V}$, in fact combining \eqref{pe.30} with the same interpolative argument leading to \eqref{d.19} and passing to the limit we obtain \eqref{lv}. We further bound
\begin{eqnarray*}
\snr{(V_{p}(Du))_{B_{-1}}-(V_{p}(Du))_{B_{0}}}\le\frac{\ti{\mf{F}}(u;B_{-1})}{\tau^{n/2}}\stackrel{\eqref{equiv}}{\le}c\mf{F}_{-1}^{p/2}\le c\mf{N}(x_{0},\rr),
\end{eqnarray*}
with $c\equiv c(\textnormal{\texttt{data}},M^{q-p})$ and choosing $i=0$ and sending $k\to \infty$ in \eqref{pe.30} we obtain that
\begin{eqnarray*}
\snr{(V_{p}(Du))_{B_{0}}-\ell_{V}}\le c\mf{N}(x_{0},\rr).
\end{eqnarray*}
We finally notice that $\ell_{V}=(V_{p}(Du))(x_{0})$, i.e. the precise representative of $V_{p}(Du)$ at $x_{0}$ so, combining \cite[inequality (6.52)]{kumi} with \eqref{lp} and \eqref{vvv} we obtain that $(V_{p}(Du))(x_{0})=V_{p}(Du(x_{0}))$. This means that \eqref{lv} is completely proven and, merging this last information with the content of the two above displays we derive \eqref{d.15} via triangular inequality.

\subsection*{Step 2: reiteration on small scales} The key ingredients to transfer the arguments leading to \eqref{d.14}-\eqref{d.15} from $B_{-1}$ to $B_{\varsigma}(x_{0})$ for some $0<\varsigma\le \rr$ satisfying \eqref{d.22} are the decay estimate \eqref{retes} and the pointwise $VMO$ result \eqref{pvmo.2}. The latter has already been established, cf. Corollary \ref{pvmo}, while the former holds by means of conditions \eqref{pe.6} and \eqref{pe.2} (implying \eqref{pe.2.1} that is obviously true on smaller balls), see Theorem \ref{t.ex}, Proposition \ref{ret} and Remark \ref{retrem}. We are then allowed to repeat the whole procedure developed in Lemma \ref{t.lem} and in the first part of Lemma \ref{l.pe} on $B_{\varsigma}(x_{0})$ to conclude with \eqref{d.14}-\eqref{d.15}, provided that \eqref{d.22} is in force.
\end{proof}
\subsection{Small gradient regime}
Here we consider the case in which $\eqref{gr}_{2}$ is in force, thus covering also the degenerate regime.
\begin{lemma}\label{l.pe.1}
Assume that
\begin{eqnarray}\label{d.23}
\snr{(Du)_{B_{\tau\varsigma}(x_{0})}}^{p/2}\le \frac{\mf{N}(x_{0},\varsigma)}{16}\qquad \mbox{for some} \ \ \varsigma\le \rr
\end{eqnarray}
and let $x_{0}\in \mathcal{R}_{u}$ be any point with $M\equiv M(x_{0})$ being the positive constant in \eqref{ru.0}. If $\bar{\varepsilon}\equiv \varepsilon'$ and $\bar{\rr}\equiv \rr'$ in \eqref{ru.0}, with $\varepsilon'$, $\rr'$ as in \eqref{pe.0}, \eqref{pe.2} respectively, then the limits in \eqref{lp}-\eqref{lv} exist and the precise representative of $Du$ and of $V_{p}(Du)$ are attained. Moreover, \eqref{d.14}-\eqref{d.15} hold for a constant $c\equiv c(\textnormal{\texttt{data}},M^{q-p})$.
\end{lemma}
\begin{proof}
For the sake of clarity, the proof is divided into two steps.
\subsection*{Step 1: limits \eqref{lp}-\eqref{lv} exist} Set $\varsigma=\rr$ and $\gamma/8:=\mf{N}(x_{0},\rr)/16$. Clearly, we can take $\gamma>0$ otherwise $u\equiv \const$ and there is nothing to prove. Moreover, we can assume that \eqref{d.23} holds for all $\varsigma\in (0,\rr]$ otherwise if there is some $\varsigma\le \rr$ for which the opposite inequality to \eqref{d.23}, i.e. \eqref{d.22} holds, we can apply Lemma \ref{l.pe} to conclude with the limits in \eqref{lp}-\eqref{lv} and \eqref{d.14}-\eqref{d.15}. Now, since \eqref{d.23} holds for all $\varsigma\in (0,\rr]$, by \eqref{pe.8} we see that $\lim_{\sigma\to 0}(Du)_{B_{\sigma}(x_{0})}=0\in \mathbb{R}^{N\times n}$ and \eqref{lp} is proven. Combining this with $\eqref{pe.8}$ we also have that $\lim_{\sigma\to 0}(\snr{Du}^{p})_{B_{\sigma}(x_{0})}=0$, which is turn implies that $\lim_{\sigma\to 0}(V_{p}(Du))_{B_{\sigma}(x_{0})}=0\in \mathbb{R}^{N\times n}$ and \eqref{lv} is proven as well.
\subsection*{Step 2: proof of inequalities \eqref{d.14}-\eqref{d.15}}
Also in this case we can replicate the same arguments leading to \eqref{ls.170} with $\rr_{0}$ instead of $r_{1}''$, $\rr$ replacing $r_{1}'$, and Lemma \ref{til.1}\footnote{Bear in mind that \eqref{pe.14} holds, that now constants \eqref{hhh.1} replace those in \eqref{hhh} in the definition of $\mf{N}(\cdot)$ and that no upper bounds on the size of $\gamma$ are needed in the light of \eqref{68.1}.} with $\varsigma\equiv \rr$ instead of Lemma \ref{til} to get that
\begin{eqnarray}\label{pe.31}
\mf{C}_{j}\le \gamma\qquad \mbox{for all} \ \ j\in \N\cup\{0\}.
\end{eqnarray}
Now, let $\sigma\in (0,\rr_{0}]$, $j_{\sigma}\in \N\cup\{0\}$ be such that $\tau^{j_{\sigma}+1}\rr_{0}< \sigma\le \tau^{j_{\sigma}}\rr_{0}$ and bound via triangular inequality,
\begin{eqnarray*}
\snr{(Du)_{B_{\sigma}(x_{0})}}&\le&A_{j_{\sigma}}+\snr{(Du)_{B_{j_{\sigma}}}-(Du)_{B_{\sigma}(x_{0})}}\le A_{j_{\sigma}}+\tau^{-n/p}\mf{F}_{j_{\sigma}}\stackrel{\eqref{hhh}}{\le} A_{j_{\sigma}}+(H_{1}\mf{F}_{j_{\sigma}}^{p/2})^{2/p}\nonumber \\
&\le&2\left(A_{j_{\sigma}}^{p/2}+H_{1}\mf{F}_{j_{\sigma}}^{p/2}\right)^{2/p}\le 2\mf{C}_{j_{\sigma}}^{2/p}\stackrel{\eqref{pe.31}}{\le}2\gamma^{2/p}\stackrel{\eqref{gr}_{2}}{\le} 2^{1+2/p}\mf{N}(x_{0},\rr)^{2/p}.
\end{eqnarray*}
On the other hand, if $\sigma\in (\rr_{0},\rr]$ we have
\begin{eqnarray*}
\snr{(Du)_{B_{\sigma}(x_{0})}}&\le&A_{0}+\frac{2\mf{F}_{-1}}{\tau^{n/p}}\stackrel{\eqref{gr}_{2}}{\le} \frac{\mf{N}(x_{0},\rr)^{2/p}}{2^{8/p}}+\frac{2\mf{N}(x_{0},\rr)^{2/p}}{\tau^{n/p}H_{3}^{2/p}}\stackrel{\eqref{hhh},\eqref{hhh.1}}{\le} 2\mf{N}(x_{0},\rr)^{2/p}.
\end{eqnarray*}
Combining the content of the two above displays we obtain
\begin{eqnarray}\label{pe.32}
\sup_{\sigma\le\rr }\snr{(Du)_{B_{\sigma}(x_{0})}}\le 4\mf{N}(x_{0},\rr)^{2/p},
\end{eqnarray}
which in turn yields that
\begin{eqnarray*}
\snr{(Du)_{B_{\sigma}(x_{0})}-(Du)_{B_{-1}}}\stackrel{\eqref{pe.32}}{\le}8\mf{F}(x_{0},\rr)^{2/p},
\end{eqnarray*}
therefore, sending $\sigma\to 0$ and recalling \eqref{lp} we can conclude with \eqref{d.14} with $\varsigma=\rr$. Concerning \eqref{d.15}, let us first note that whenever $B\Subset \Omega$ is a ball, by H\"older inequality it is
\begin{eqnarray}\label{d.24}
\snr{(V_{p}(Du))_{B}}&\le&\left(\mint_{B}\snr{Du}^{p} \dx\right)^{1/2}\le 2^{p}\mf{F}(u;B)^{p/2}+2^{p}\snr{(Du)_{B}}^{p/2}.
\end{eqnarray}
Now, if $\sigma\le \rr_{0}$ and $j_{\sigma}\in \N\cup\{0\}$ is such that $\tau^{j_{\sigma}+1}\rr_{0}< \sigma\le \tau^{j_{\sigma}}\rr_{0}$, by triangular inequality we bound
\begin{eqnarray*}
\snr{(V_{p}(Du))_{B_{\sigma}(x_{0})}}&\stackrel{\eqref{d.24}}{\le}&\snr{(V_{p}(Du))_{B_{\sigma}(x_{0})}-(V_{p}(Du))_{B_{j_{\sigma}}}}+2^{p}\mf{F}_{j_{\sigma}}^{p/2}+2^{p}A_{j_{\sigma}}^{p/2}\nonumber \\
&\le&\frac{1}{\tau^{n/2}}\ti{\mf{F}}_{j_{\sigma}}+2^{p}\mf{F}_{j_{\sigma}}^{p/2}+2^{p}A_{j_{\sigma}}^{p/2}\nonumber \\
&\stackrel{\eqref{equiv},\eqref{pe.31}}{\le}&\left(\frac{c}{\tau^{n/2}}+2^{p}\right)\mf{F}_{j_{\sigma}}^{p/2}+2^{p}\gamma\le c\mf{N}(x_{0},\rr),
\end{eqnarray*}
with $c\equiv c(\textnormal{\texttt{data}}, M^{q-p})$. Furthermore, if $\sigma\in (\rr_{0},\rr]$, we directly have
\begin{eqnarray*}
\snr{(V_{p}(Du))_{B_{\sigma}(x_{0})}}&\stackrel{\eqref{d.24}}{\le}&\snr{(V_{p}(Du))_{B_{\sigma}(x_{0})}-(V_{p}(Du))_{B_{-1}}}\nonumber \\
&+&\snr{(V_{p}(Du))_{B_{0}}-(V_{p}(Du))_{B_{-1}}}+2^{p}\mf{F}_{0}^{p/2}+2^{p}A_{0}^{p/2}\nonumber \\
&\le&\frac{2}{\tau^{n/2}}\ti{\mf{F}}_{-1}+2^{p}A_{0}^{p/2}+2^{p}\mf{F}_{0}^{p/2}\stackrel{\eqref{equiv}}{\le}c\mf{F}_{-1}^{p/2}+cA_{0}^{p/2}\stackrel{\eqref{gr}_{2}}{\le}c\mf{N}(x_{0},\rr),
\end{eqnarray*}
for $c\equiv c(\textnormal{\texttt{data}}, M^{q-p})$. Merging the two above displays we obtain
\begin{eqnarray*}
\snr{(V_{p}(Du))_{B_{\sigma}(x_{0})}-(V_{p}(Du))_{B_{-1}}}\stackrel{\eqref{gr}_{2},\eqref{d.24}}{\le}c\mf{N}(x_{0},\rr),
\end{eqnarray*}
so sending $\sigma\to0$ and recalling \eqref{lv} we can conclude with \eqref{d.15} with $\varsigma=\rr$. To obtain \eqref{d.14}-\eqref{d.15} for all radii $\varsigma\le \rr$ satisfying \eqref{d.23}, we just need to recall \eqref{pe.0}-\eqref{pe.2}, Proposition \ref{ret}, Remark \ref{retrem} and Corollary \ref{pvmo}, that allow transferring the above estimates to smaller scales than the $\rr$-one, see also \textbf{Step 2} of the proof of Lemma \ref{l.pe}. The proof is complete.
\end{proof}
A direct consequence of Lemmas \ref{l.pe}-\ref{l.pe.1} is the following
\begin{corollary}\label{cccp}
With \eqref{con.1.1} in force, $x_{0}\in \mathcal{W}_{u}\cap \mathcal{L}_{u}$ if and only if \eqref{con.1} holds at $x_{0}$.
\end{corollary}
\begin{proof}
If $x_{0}\in \mathcal{W}_{u}\cap \mathcal{L}_{u}$ and \eqref{con.1.1} holds, condition \eqref{con.1} is trivially satisfied, cf. Section \ref{rs}. Now let $x_{0}\in \Omega$ be any point satisfying \eqref{con.1} for some positive constant $M\equiv M(x_{0})$, the threshold radius $\breve{\rr}\equiv \rr'$, where $\rr'\equiv \rr'(\textnormal{\texttt{data}}_{\textnormal{c}},M^{q-p},f(\cdot))$ has been determined at the beginning of Section \ref{pgc}, and $\breve{\varepsilon}\equiv c_{8}^{-1}2^{-10npq}(\tau\theta)^{n+4nq/(p-1)}\varepsilon'$, with $\varepsilon'\equiv \varepsilon'(\textnormal{\texttt{data}}_{\textnormal{c}},M^{q-p})$ defined in \eqref{pe.0}. With these choices \eqref{con.1.1}-\eqref{con.1} immediately imply \eqref{pe.2} and \eqref{pe.6}, therefore \eqref{pe.8} and \eqref{equiv} yield that $\lim_{\sigma\to 0}\ti{\mf{F}}(u;B_{\sigma}(x_{0}))=0$ and Lemmas \ref{l.pe}-\ref{l.pe.1} give \eqref{lp}, so $x_{0}\in \mathcal{W}_{u}\cap \mathcal{L}_{u}$.
\end{proof}
\subsection{Proof of Theorems \ref{t2} and \ref{t5}}
After determining the various smallness parameters at the beginning of Section \ref{pgc}, the result follows by matching the content of Lemmas \ref{l.pe}-\ref{l.pe.1}. We stress that we have worked under condition \eqref{con.1} which, by \eqref{equiv} is equivalent to \eqref{con.6}. The last part of Theorems \ref{t2}-\ref{t5} is contained in Corollary \ref{cccp}.
\subsection{Proof of Theorem \ref{t3}}
Let $x_{0}\in \mathcal{R}_{u}$ be a point and $M\equiv M(x_{0})>0$ be the positive, finite constant in \eqref{ru.0}. Define a small parameter $\ti{\varepsilon}$ as
\begin{eqnarray}\label{e.0.0}
\ti{\varepsilon}:=\frac{\varepsilon'(\tau\theta)^{8npq}}{c_{8}2^{16npq}} \ \stackrel{\eqref{pe.0},\eqref{fpm.1.1.1}}{\Longrightarrow} \ \ti{\varepsilon}\equiv \ti{\varepsilon}(\textnormal{\texttt{data}}_{\textnormal{c}},M^{q-p})
\end{eqnarray}
and, via assumption \eqref{con.5} (that we can always suppose to hold globally on $\Omega$, given that our results are of local nature), we can find a threshold radius $\ti{\rr}\equiv \ti{\rr}(\textnormal{\texttt{data}}_{\textnormal{c}},M^{q-p},f(\cdot))\in (0,\rr']$ ($\rr'$ being determined via \eqref{pe.2}) so that
\begin{eqnarray}\label{e.0}
\mf{K}\left(\mathbf{I}^{f}_{1,m}(x,\ti{\rr})\right)^{1/(p-1)}<\ti{\varepsilon}\qquad \mbox{for all} \ \ x\in B_{d_{x_{0}}}(x_{0}),
\end{eqnarray}
and of course we can always assume that $B(x_{0})\subset B_{d_{x_{0}}}(x_{0})$, see Remark \ref{rx}.
According to the discussion in Section \ref{rs}, it holds that
\begin{flalign}\label{pe.34}
\mathcal{R}_{u} \quad \mbox{is open and}\quad \snr{\Omega\setminus \mathcal{R}_{u}}=0,
\end{flalign}
 therefore if $\bar{\varepsilon}\equiv \ti{\varepsilon}$, $\bar{\rr}\equiv \ti{\rr}$ in \eqref{ru.0} we can find 
an open neighborhood of $x_{0}$, say $B(x_{0})$, on which it is
\begin{eqnarray}\label{e.2}
\snr{(Du)_{B_{\rr_{x_{0}}}(x)}}<M\quad \mbox{and}\quad \mf{F}(u;B_{\rr_{x_{0}}}(x))<\ti{\varepsilon}\qquad \mbox{for all} \ \ x\in B(x_{0}).
\end{eqnarray}
From \eqref{e.0.0}, \eqref{e.0} and \eqref{e.2} we see that Corollary \ref{pvmo} and Theorem \ref{t2} apply and render that the limit in \eqref{lp} exists for all $x\in B(x_{0})$ and defines the almost precise representative of $Du$ for all $x\in B(x_{0})$. Our goal is to show that the limit in \eqref{lp} is uniform, i.e. that the continuous map $B(x_{0})\ni x\mapsto (Du)_{B_{\sigma}(x)}$ with $\sigma\in (0,\rr_{x_{0}})$ uniformly converge to $Du(x)$ as $\sigma\to 0$, thus implying that $Du$ is continuous on $B(x_{0})$. This is a consequence of \eqref{lp.1} as the terms on its right-hand side uniformly converge to zero on $B(x_{0})$ as $\sigma\to 0$ by means of \eqref{con.5}, Theorem \ref{t1}, Corollary \ref{pvmo} and Remark \ref{pvmo.3}. Finally, a standard covering argument yields that $Du$ is continuous on $\mathcal{R}_{u}$. This last fact and \eqref{pe.34} yield that $\mathcal{R}_{u}$ verifies \eqref{pe.35}, so $\mathcal{R}_{u}\equiv \Omega_{u}$ and the proof is complete.
\subsection{Proof of Theorem \ref{t4}}
Implication (\emph{i}.) is a direct consequence of Theorem \ref{t3}, given that $$f\in L(n,1)(\Omega,\mathbb{R}^{N}) \ \Longrightarrow \ \lim_{\sigma\to 0}\mathbf{I}^{f}_{1,m}(x,\sigma)=0$$ uniformly in $x\in \Omega$, cf. \cite[Section 2.3]{kumi1}. Concerning implication (\emph{ii.}), since $f\in L^{d}(\Omega,\mathbb{R}^{N})$ for some $d>n$, a quick application of H\"older inequality shows 
\begin{eqnarray}\label{pe.37}
\left(s^{m}\mint_{B_{s}(x)}\snr{f}^{m} \ \dx\right)^{1/m}\le \frac{s^{1-\frac{n}{d}}}{\omega_{n}^{1/d}}\nr{f}_{L^{d}(B_{s}(x))},
\end{eqnarray}
for all $x\in \Omega$, where we also used that $m<n$ by \eqref{f}, and
\begin{eqnarray}\label{pe.40}
\mathbf{I}^{f}_{1,m}(x,\sigma)&\stackrel{\eqref{pe.37}}{\le}&\frac{1}{\omega_{n}^{1/d}}\int_{0}^{\sigma}s^{-n/d}\nr{f}_{L^{d}(B_{s}(x))} \ \ds\nonumber \\
&\stackrel{d>n}{\le}&\frac{d\sigma^{1-n/d}\nr{f}_{L^{d}(B_{\sigma}(x))}}{(d-n)\omega_{n}^{1/d}}\le \frac{d\sigma^{1-n/d}\nr{f}_{L^{d}(\Omega)}}{(d-n)\omega_{n}^{1/d}}.
\end{eqnarray}
Let $x_{0}\in \mathcal{R}_{u}$ with $M\equiv M(x_{0})$, $\hat{\varepsilon}$, $\hat{\rr}$ be as in \eqref{be}-\eqref{fpm.1} respectively and $\bar{\varepsilon}\equiv \hat{\varepsilon}$, $\bar{\rr}\equiv \hat{\rr}$ in \eqref{ru.0}. Inequality \eqref{pe.40} in turn implies \eqref{con.5} and, via \eqref{d.3} also that
\begin{eqnarray}\label{pe.39}
\sup_{s\le \sigma/4}\mf{K}\left[\left(s^{m}\mint_{B_{s}(x)}\snr{f}^{m} \dx\right)^{1/m}\right]^{1/(p-1)}&\le&\frac{2^{\frac{8nq}{p-1}}\mf{K}\left(\mathbf{I}^{f}_{1,m}(x,\sigma)\right)^{1/(p-1)}}{(\tau\theta)^{\frac{4nq}{p-1}}}\nonumber \\
&\stackrel{\eqref{pe.40}}{\le}&c_{9}\sigma^{\alpha'}\mf{K}\left(\nr{f}_{L^{d}(B_{\sigma}(x))}\right)^{1/(p-1)},
\end{eqnarray}
holding for all $x\in B_{d_{x_{0}}}(x_{0})$, where we set $\alpha':=\frac{d-n}{d(p-1)}$ and it is $c_{9}\equiv c_{9}(\textnormal{\texttt{data}}_{\textnormal{c}},d,M^{q-p})$. With $\hat{\varepsilon}\equiv \hat{\varepsilon}(\textnormal{\texttt{data}}_{\textnormal{c}},M^{q-p})$ as in \eqref{be}, from \eqref{pe.39} we see that we can find $\hat{\rr}\equiv \hat{\rr}(\textnormal{\texttt{data}}_{\textnormal{c}},M^{q-p},f(\cdot),d)$ so small that whenever $\sigma\le \hat{\rr}$ it is
$$
c_{10}\sigma^{\alpha'}\mf{K}(\nr{f}_{L^{d}(B_{\sigma}(x))})^{1/(p-1)}\le c_{10}\hat{\rr}^{\alpha'}\mf{K}(\nr{f}_{L^{d}(B_{\ti{\rr}}(x))})^{1/(p-1)}\le \frac{\hat{\varepsilon}(\varepsilon_{0}\tau\theta)^{8np}}{2^{56np(q+2)}c_{3}H_{2}^{1+2/p}},
$$
where we set $c_{10}:=c_{5}c_{9}$, $c_{5}\equiv c_{5}(\textnormal{\texttt{data}}_{\textnormal{c}},M^{q-p})$ is the same constant defined in \eqref{c4c5} and $H_{2}$ is as in \eqref{hhh}, thus \eqref{pe.39} ultimately reads as
\begin{eqnarray}\label{pe.41}
\sup_{s\le \sigma/4}\mf{K}\left[\left(s^{m}\mint_{B_{s}(x)}\snr{f}^{m} \dx\right)^{1/m}\right]^{1/(p-1)}\le \frac{\hat{\varepsilon}(\varepsilon_{0}\tau\theta)^{8np}}{2^{56np(q+2)}c_{3}H_{2}^{1+2/p}},
\end{eqnarray}
for any $\sigma\le \hat{\rr}$ and all $x\in B_{d_{x_{0}}}(x_{0})$. Notice that, by means of \eqref{pe.40} and \eqref{con.5} we can always take $\hat{\rr}$ smaller or equal than the corresponding one considered in Section \ref{ex} and, given the values of the various constants involved, \eqref{pe.41} implies \eqref{d.3} for all $x\in B_{d_{x_{0}}}(x_{0})$, therefore Proposition \ref{ret} applies and gives that there is an open neighborhood $B(x_{0})\subset \mathcal{R}_{u}$ and a positive radius $\rr_{x_{0}}$ so that \eqref{ret.1.1} and
\begin{eqnarray}\label{pe.45}
\mf{F}(u;B_{s}(x))&\le& c_{6}\left(\frac{s}{\varsigma}\right)^{\alpha_{0}}\mf{F}(u;B_{\varsigma}(x))+c_{5}\sup_{\sigma\le \varsigma/4}\mf{K}\left[\left(\sigma^{m}\mint_{B_{\sigma}(x)}\snr{f}^{m} \ \dx\right)^{1/m}\right]^{1/(p-1)}\nonumber \\
&\stackrel{\eqref{pe.39}}{\le}&c_{6}\left(\frac{s}{\varsigma}\right)^{\alpha_{0}}\mf{F}(u;B_{\varsigma}(x))+c_{10}\varsigma^{\alpha'}\mf{K}\left(\nr{f}_{L^{d}(B_{\varsigma}(x))}\right)^{1/(p-1)}
\end{eqnarray}
hold for any $x\in B(x_{0})$ and all $s,\varsigma$ so that $0<s\le \varsigma\le \rr_{x_{0}}\le \hat{\rr}$. Of course, \eqref{pe.45} holds verbatim if we replace $\alpha'$ with any $0<\alpha''<\min\{\alpha_{0},\alpha'\}$, therefore by Lemma \ref{iter+} we can conclude that
\begin{flalign}\label{e.4}
\mf{F}(u;B_{s}(x))\le c_{11}\left(\frac{s}{\varsigma}\right)^{\alpha''}\mf{F}(u;B_{\varsigma}(x))+c_{11}s^{\alpha''}\mf{K}\left(\nr{f}_{L^{d}(B_{\varsigma}(x))}\right)^{1/(p-1)}\quad \mbox{for all} \ \ x\in B(x_{0}),
\end{flalign}
with $c_{11}\equiv c_{11}(\textnormal{\texttt{data}}_{\textnormal{c}},d,M^{q-p})$. Once estimate \eqref{e.4} is available, the local H\"older continuity of $Du$ over $\mathcal{R}_{u}$ follows by standard means, see \cite{ts2}.

\end{document}